\documentclass[openany]{now} 

\usepackage{makeidx}
\usepackage{url}
\usepackage{lmodern}
\usepackage[T1]{fontenc}
\usepackage{graphicx}
\usepackage{adjustbox}
\usepackage{caption}
\usepackage{comment}
\usepackage{algorithm,algorithmic}
\usepackage{multirow}
\usepackage{amssymb, amsmath}
\makeindex

\usepackage{tikz}

\usepackage{color}
\newcommand{\red}[1]{\textcolor{red}{#1}}

\newcommand{\nc}{\newcommand}
\newcommand{\textdef}[1]{\textit{#1}\index{#1}}

\def\Ss{{\bf S}}
\def\Sn{{\bf S}^n}

\def\Snp{\Sn_+}

\def\Sntop{{\Ss^{n^2+1}_+}}
\def\R{{\bf R}}

\def\Rn{\R^n}

\def\E{{\mathbb{\bf E}}}
\def\F{{\bf \mathbb{\bf F}}}
\def\Sc{\Ss}

\def\aff{{\rm aff}\,}
\def\sing{{\rm sing}}
\def\Rnp{\Rn_+}

\nc{\ip}[2]{\mbox{$\langle #1,#2 \rangle$}}
\nc{\pf}{\noindent{\bf Proof\ \ }}
\nc{\finpf}{\hfill{$\Box$}\linespace}
\nc{\linespace}{\vspace{\baselineskip} \noindent}
\nc{\K}{\mathcal{K}}
\nc{\cl}{\mbox{\rm cl}\,}
\nc{\cond}{\mbox{\rm cond}}
\nc{\cls}{ \mbox{{\scriptsize {\rm cl}}}\,}
\nc{\conv}{\mbox{\rm conv}}
\nc{\rb}{\mbox{\rm rb}\,}
\nc{\inter}{\mbox{\rm int}\,}
\nc{\kernel}{\mbox{\rm ker}\,}
\nc{\bd}{\mbox{\rm bd}\,}
\nc{\spann}{\mbox{\rm span}\,}
\nc{\epi}{\mbox{\rm epi}\,}
\nc{\gph}{\mbox{\rm gph}\,}
\nc{\rge}{\mbox{\rm rge}\,}
\nc{\val}{\mbox{\rm\bf val}}
\nc{\rgel}{\mbox{\rm {\scriptsize rge}}\,}
\nc{\sepi}{\mbox{\rm {\scriptsize epi}}\,}
\nc{\dom}{\mbox{\rm dom}\,}
\nc{\detr}{\mbox{\rm det}\,}
\nc{\para}{\mbox{\rm par}\,}
\nc{\proj}{\mbox{\rm proj}}
\nc{\dist}{\mbox{\rm dist}}
\nc{\crit}{\mbox{\rm crit}\,}
\nc{\cone}{\mbox{\rm cone}\,}
\nc{\mult}{\mbox{\rm mult}\,}
\nc{\rad}{\mbox{\rm rad}}

\nc{\sdom}{\mbox{\rm {\scriptsize dom}}\,}
\nc{\supp}{\mbox{\rm supp}\,}

\newcommand{\Pprob}{\textbf{(P)}}

\newcommand{\Pprobp}{\textbf{(P)}\,}
\newcommand{\Dprob}{\textbf{(D)}}

\newcommand{\Dprobp}{\textbf{(D)}\,}

\newcommand{\SNL}{\text{SNL}\, }
\newcommand{\SNLp}{\text{SNL}}

\newcommand{\PSD}{\text{PSD}\, }
\newcommand{\PSDp}{\text{PSD}}
\newcommand{\CQ}{\text{CQ}\, }
\newcommand{\CQp}{\text{CQ}}
\newcommand{\MFCQ}{\text{MFCQ}\, }
\newcommand{\MFCQp}{\text{MFCQ}}
\newcommand{\LP}{\text{LP}\, }
\newcommand{\LPp}{\text{LP}}
\newcommand{\SDP}{\text{SDP}\, }
\newcommand{\SDPp}{\text{SDP}}
\newcommand{\EDMC}{\text{EDMC}\, }
\newcommand{\EDMCp}{\text{EDMC}}
\newcommand{\EDM}{\text{EDM}\, }
\newcommand{\EDMp}{\text{EDM}}
\newcommand{\LRMC}{\text{LRMC}\, }
\newcommand{\LRMCp}{\text{LRMC}}
\newcommand{\QAP}{\text{QAP}\, }
\newcommand{\QAPp}{\text{QAP}}
\newcommand{\MC}{\textbf{MC}\, }
\newcommand{\MCp}{\textbf{MC}}

\newcommand{\GPp}{\textbf{GP}}

\newcommand{\FR}{\text{FR}\, }
\newcommand{\FRp}{\text{FR}}

\newcommand{\JJ}{{\mathcal J} }

\newcommand{\RR}{{\mathcal R} }
\newcommand{\EE}{{\mathcal E} }
\newcommand{\FF}{{\mathcal F} }

\newcommand{\GG}{{\mathcal G} }

\newcommand{\PP}{{\mathcal P} }

\newcommand{\En}{{{\mathcal E}^n} }

\newcommand{\con}{\rm cond}

\newcommand{\A}{{\mathcal A}}

\newcommand{\N}{{\mathcal N\,}}

\newcommand{\bbm}{\begin{bmatrix}}
	\newcommand{\ebm}{\end{bmatrix}}
\newcommand{\bem}{\begin{pmatrix}}
	\newcommand{\eem}{\end{pmatrix}}
\newcommand{\beq}{\begin{equation}}
\newcommand{\beqs}{\begin{equation*}}
\newcommand{\bet}{\begin{table}}
\newcommand{\eeq}{\end{equation}}
\newcommand{\eeqs}{\end{equation*}}
\newcommand{\beqr}{\begin{eqnarray}}

\DeclareMathOperator{\rank}{{rank}}
\DeclareMathOperator{\ri}{{ri}}
\DeclareMathOperator{\face}{{face}}
\DeclareMathOperator{\KK}{{\mathcal K} }
\DeclareMathOperator{\kvec}{{vec}}

\DeclareMathOperator{\trace}{{tr}}
\DeclareMathOperator{\tr}{{trace}}

\DeclareMathOperator{\diag}{{diag}}
\DeclareMathOperator{\Diag}{{Diag}}

\DeclareMathOperator{\embdim}{{embdim}}
\DeclareMathOperator{\offDiag}{{offDiag}}

\DeclareMathOperator{\svec}{{s2vec}}

\DeclareMathOperator{\Range}{range}%
\DeclareMathOperator{\range}{{range}}
\DeclareMathOperator{\nul}{{null}}
\DeclareMathOperator{\Null}{{null}}

\nc{\arrow}{{\rm arrow\,}}
\nc{\Arrow}{{\rm Arrow\,}}
\nc{\BoDiag}{{\rm B^0Diag\,}}
\nc{\bodiag}{{\rm b^0diag\,}}

\nc{\Mm}{{\mathcal M}^{m} }
\nc{\Mmn}{{\mathcal M}^{mn} }
\nc{\Mnr}{{\mathcal M}_{nr} }
\nc{\Mnmr}{{\mathcal M}_{(n-1)r} }
\nc{\kwqqp}{Q{$^2$}P\,}
\nc{\kwqqps}{Q{$^2$}Ps}

\nc{\notinaho}{(X,S)\in \overline{AHO}(\A)}
\nc{\inaho}{(X,S)\in AHO(\A)}

\newcommand{\bea}{\begin{eqnarray}}%
\newcommand{\eea}{\end{eqnarray}}%
\newcommand{\beas}{\begin{eqnarray*}}%
	\newcommand{\eeas}{\end{eqnarray*}}%
\newcommand{\Rmn}{\R^{m \times n}}%
\newcommand{\Int}{{\rm int\,}}

{}
	
	\newcommand{\Hnp}[1][]{\,\mathbb{H}_+^{\ifthenelse{\equal{#1}{}}{n}{#1}}}
	\newcommand{\Hn}[1][]{\,\mathbb{H}^{\ifthenelse{\equal{#1}{}}{n}{#1}}}
	\newcommand{\Dn}[1][]{\,\mathbb{D}^{\ifthenelse{\equal{#1}{}}{n}{#1}}}
	
\title{The many faces of degeneracy\\
in\\
conic optimization\\
	{\tiny last modified on \today }
}

\author{
Dmitriy Drusvyatskiy \\
Department of Mathematics\\
University of Washington \\
\url{ddrusv@uw.edu}
\and
Henry Wolkowicz \\
Faculty of Mathematics \\
University of Waterloo \\
\url{hwolkowicz@uwaterloo.ca}
}

\begin{document}

\copyrightowner{}

\frontmatter  

\maketitle

\tableofcontents
\listoffigures
\listoftables

\mainmatter

\begin{abstract}
Slater's condition -- existence of a ``strictly feasible solution'' -- 
is a common assumption in conic optimization. Without strict feasibility, first-order optimality conditions may be 
meaningless, the dual problem may yield little information about the primal, 
and small changes in the data may render the problem infeasible. 
Hence, failure of strict feasibility can negatively impact off-the-shelf numerical methods, such as primal-dual interior point methods, in particular. 
New optimization modelling techniques and convex relaxations for hard 
nonconvex problems have shown that the loss of strict feasibility is a 
 more pronounced phenomenon than has previously been realized. 
In this text, we describe various reasons for the loss of strict 
feasibility, whether due to poor modelling choices or 
(more interestingly) rich underlying structure, and discuss ways to cope 
with it and, in many pronounced cases, how to use it as an advantage. 
In large part, we emphasize the facial reduction preprocessing technique due to its mathematical elegance, 
geometric transparency, and computational potential. 
\end{abstract}

\chapter{What this paper is about}
Conic optimization has proven to be an elegant and powerful modeling tool with
surprisingly many applications. The classical \emph{linear programming}  
\index{\LPp, linear program} 
\index{linear program, \LPp} 
problem revolutionized operations research and
is still the most widely used optimization model. This is due to the elegant theory
and the ability to solve in practice both small and large scale problems efficiently
and accurately by the well known simplex method of Dantzig
\cite{Dant:63} and by more recent interior-point methods,
e.g.,~\cite{SWright:96,int:Nesterov5}. The size (number of variables)
of linear programs that could be solved before
the interior-point revolution was on the order of tens of
thousands, whereas it immediately increased to millions for many applications.
A large part of modern success is due to \textdef{preprocessing}, which aims to identify (primal and dual slack) variables that are identically zero
on the feasible set.
The article~\cite{springerlink:10.1007/s00291-003-0130-x} is a good reference.

The story does not end with linear programming.
Dantzig himself recounts in~\cite{Dant:90}: ``the world is nonlinear''.
Nonlinear models can significantly improve on linear programs if they can be solved
efficiently. Conic optimization has shown its worth in its elegant
theory, efficient algorithms, and many applications
e.g.,~\cite{SaVaWo:97,AnjosLasserre:11,MR1857264}. Preprocessing to rectify possible loss of ``strict-feasibility'' in the primal or the dual problems is appealing for general conic optimization as well. 
In contrast to linear programming, however, the area of preprocessing for  conic optimization is in its infancy; see 
e.g.,~\cite{Cheung:2013,ScTuWominimal:07,ScTuWonumeric:07,MR3108446,permfribergandersen}
and Section \ref{sect:relworkone}, below.
In contrast to linear programming, numerical error makes preprocessing difficult in full generality.
This being said, surprisingly, there are many specific applications 
of conic optimization, where the rich underlying structure makes 
preprocessing possible, leading to greatly simplified models and strengthened 
algorithms. Indeed, exploiting structure is essential for making 
preprocessing viable. In this article, we present the background
and the elementary theory of such regularization techniques in the framework 
of \textdef{facial reduction (\FRp)}. We focus on notable case studies, where such techniques have proven to be useful.
\index{\FRp, facial reduction}

\section{Related work}
\label{sect:relworkone}
To put this text in perspective, it is instructive to consider nonlinear programming.
Nontrivial statements in constrained nonlinear optimization always rely on some regularity of the constraints. To illustrate, consider a minimization problem over a set of the form $\{x:f(x)=0\}$ for some smooth $f$. How general are such constraints? A celebrated  result of Whitney \cite{Whitney} shows that {\em any} closed set in a Euclidean space can written as a zero-set of some $C^{\infty}$-smooth function $f$. Thus, in this generality, there is little difference between minimizing over arbitrary closed sets and sets of the form $\{x:f(x)=0\}$, for smooth $f$. 
Since little can be said about optimizing over arbitrary closed sets, one must make an assumption on the equality constraint. The simplest one, eliminating Whitney's construction, is that the gradient of $f$ is nonzero on the feasible region -- the earliest form of a constraint qualification. There have been numerous papers, developing weakened versions of regularity (and optimality conditions) in nonlinear programming; some good examples are \cite{Guig:69,bw4,bw2}.

The Slater constraint qualification, we discuss in this text, is in a similar spirit, but in the context of (convex) conic optimization. Some good early references on the geometry of the Slater condition, and weakened variants, are \cite{MR45:6415,Massam:79,MR663327,w8,MR607673}. The concept of facial reduction for general convex programs was introduced in~\cite{bw1,bw3}, while
an early application to a semi-definite type best-approximation
problem was given in~\cite{w11}. Recently, there has been a significant renewed interest in facial reduction, in large part due to the success in applications for graph related problems, such as Euclidean distance
matrix completion and molecular conformation \cite{kriswolk:09,krislock:2010,ChDrWo:14,Alipanahi:2012} and in polynomial optimization \cite{perm, perm_sos, coj_spars,waki_mur_sparse,WakiKimKojimaMura:06}.
In particular, a more modern explanation of the facial reduction procedure can be found in \cite{GPatakiLiu:15, P00, MR3108446, MR2724357,MR3063940}.

We note in passing that numerous papers show that strict feasibility
holds ``generically'' with respect to unstructured perturbations. 
In contrast, optimization problems appearing in applications are often highly structured and such genericity results are of little practical use.

\section{Outline of the paper}
The paper is divided into two parts.
In Part \ref{part:theory}, we present the necessary theoretical grounding
in conic optimization, including basic optimality and duality theory, connections of Slater's condition to the \emph{distance to infeasibility} and sensitivity theory, the facial reduction procedure, and the singularity degree.
In Part \ref{part:applic}, we concentrate on illustrative examples and
applications, including matrix completion problems (semi-definite, low-rank, and Euclidean distance), relaxations of hard combinatorial problems (quadratic assignment and max-cut), and sum of squares relaxations of polynomial optimization problems.

\section{Reflections on Jonathan Borwein and \FR}
These are some reflections on Jonathan Borwein and his role in the
development of the facial reduction technique, by Henry Wolkowicz.
Jonathon Borwein passed away unexpectedly on Aug. 2, 2016. 
Jon was an extraordinary mathematician who made significant
contributions in an amazing number of very diverse areas.
Many details and personal memories by myself and many others including
family, friends, and colleagues, are presented at the memorial website 
\url{jonborwein.org}. This was a terrible loss to his family and all
his friends and colleagues, including myself.
The facial reduction process we use in this monograph originates in the work of Jon and
the second author (myself). This work took place from July of 1978 to July of 1979
when I went to Halifax to work with Jon at Dalhousie University in a
lectureship position.  The optimality conditions for the general abstract
convex program using the facially reduced problem is presented in the
two papers~\cite{bw1,bw2}. The facial reduction process is then derived
in~\cite{bw3}.

\part{Theory}
\label{part:theory}
\chapter{Convex geometry}
\label{sect:convgeom}

This section collects preliminaries of linear algebra and convex
geometry that will be routinely used in the rest of the manuscript. The
main focus is on convex duality, facial structure of convex cones, and
the primal-dual conic optimization pair. The two running examples of
linear and semi-definite programming illustrate the concepts. We have tried to include proofs of important theorems, when they are both elementary and enlightening. We have omitted arguments that are longer or that are less transparent, so as not to distract the reader from the narrative.

\section{Notation}
Throughout, we will fix a Euclidean space $\E$ with an inner product $\langle \cdot,\cdot
\rangle$ and the induced norm $\|x\|=\sqrt{\langle x,x\rangle}$. When referencing another Euclidean space (with its own inner product), we will use the letter $\F$. An open ball of radius $r>0$ around a point $x\in \E$ will be denoted by $B_{r}(x)$.
The two most  important examples of Euclidean spaces for us will be the space of $n$-vectors $\R^n$ with the 
dot product $\langle x,y\rangle=\sum_i x_iy_i$ and the space of $n\times n$ 
symmetric matrices ${\bf S}^n$ with the trace inner product 
$\langle X,Y\rangle=\tr(XY)$. Throughout, we let $e_i$ be the $i$'th coordinate vector of $\R^n$.
 Note that the trace inner product can be
equivalently written as $\tr(XY)=\sum_{i,j} X_{ij}Y_{ij}$. Thus the
trace inner product is itself the dot product between the two matrices stretched into vectors. A key property of the trace is invariance under permutations of the arguments: $\trace(AB)=\trace(BA)$ for any two matrices $A\in \R^{m\times n}$ and $B\in \R^{n\times m}$.

For any linear mapping $A\colon\E\to{\F}$, between
Euclidean spaces $\E$ and ${\F}$, the \textdef{adjoint mapping} $A^*\colon {\F}\to {\bf E}$ is the unique mapping satisfying
$$\langle Ax,y\rangle = \langle x, A^*y \rangle\quad \textrm{ for all } x \in \E\textrm{ and }y\in {\F}.$$ 
\index{trace inner product}
\index{adjoint}
Notice that the angle brackets on the left refer to the inner product in $\F$, while those on the right refer to the inner product in $\E$.

Let us look at  two important examples of adjoint maps.

\begin{example}[Adjoints of mappings between $\R^n$ and $\R^m$]
Consider a matrix $A\in\R^{m\times n}$ as a linear map from $\R^n\to\R^m$. Then the adjoint $A^*$ is simply the transpose $A^T$. To make the parallel with the next example, it is useful to make this description more explicit.
Suppose that the linear operator $A\colon\R^n\to\R^m$ is defined by 
\begin{equation}\label{eqn:row_space}
Ax=(\langle a_1,x\rangle,\ldots,\langle a_m,x\rangle),
\end{equation}
where $a_1,\ldots,a_m$ are some vectors in $\R^n$. When thinking of $A$ as a matrix, the vectors $a_i$ would be its rows, and the description \eqref{eqn:row_space} corresponds to a ``row-space'' view of matrix-vector multiplication $Ax$.
The adjoint $A^*\colon\R^m\to\R^n$ is simply the map 
\begin{equation}\label{eqn:adj_col}
A^*y=\sum_i y_i a_i.
\end{equation}
Again, when thinking of $A$ as a matrix with $a_i$ as its rows, the description \eqref{eqn:adj_col} corresponds to the ``column-space'' view of matrix-vector multiplication $A^Ty$.
\end{example}

\begin{example}[Adjoints of mappings between ${\bf S}^n$ and $\R^m$]
Consider a set of symmetric matrices $A_1,\ldots,A_m$ in ${\bf S}^n$, and 
define the linear map $A\colon {\bf S}^n\to\R^m$ by 
$$A(X)=(\langle A_1,X\rangle,\ldots,\langle A_m,X\rangle).$$
We note that any linear map $A\colon {\bf S}^n\to\R^m$ can be written in this way for some matrices $A_i\in \Sn$.
Notice the parallel to \eqref{eqn:row_space}. The adjoint $A^*\colon\R^m\to{\bf S}^n$ is given by  
$$A^*y=\sum_i y_i A_i.$$
Notice the parallel to \eqref{eqn:adj_col}. To verify that this indeed is the adjoint, simply observe the equation
$$
\langle X, \sum_i y_i A_i\rangle=\sum_i y_i\langle
A_i,X\rangle=\langle A(X),y \rangle,
$$
for any $X\in {\bf S}^n$ and $y\in\R^n$.
\end{example}

The interior, boundary, and closure of any set $C\subset\E$ will be denoted by $\inter C$, $\bd C$, and $\cl C$, respectively.
\index{$\inter C$} \index{$\bd C$} \index{$\cl C$}
A set $C$ is {\em convex}  if it contains the line segment joining any two points in $C$:
	\[
		x,y\in C,~ \alpha \in [0,1] \qquad\implies\qquad
		        \alpha x + (1-\alpha)y \in C.
	\]
The minimal affine space containing a convex set $C$ is called the {\em affine hull} of $C$, and is denoted by $\aff C$. We define the {\em relative interior} of $C$, written $\ri C$, to be the interior of $C$ relative to $\aff C$.
It is straightforward to show that a for a nonempty convex set $C$, the relative interior $\ri C$ is never empty.

A subset $\K$ of $\E$ is a \textdef{convex cone} if $\K$ is convex and
is positively homogeneous, meaning $\lambda K\subseteq K$ for all
$\lambda\geq 0$. Equivalently, $\K$ is a convex cone if, and only if, 
for any two points
$x$ and $y$ in $\K$ and any nonnegative constants $\alpha,\beta\geq
0$, the sum $\alpha x+\beta y$ lies in $\mathcal{K}$. We say that a convex cone $\K$ is {\em proper} if $\K$ is closed, has nonempty interior, and contains no lines.  The symbol $\KK^{\perp}$ refers to the orthogonal complement of $\aff \KK$.
Let us look at two most important examples of proper cones for this article.

\begin{example}[The nonnegative orthant $\R^n_+$]\label{exa:nonneg_ortho}
	\index{$\Rnp$}
The nonnegative orthant 
$$\R^n_+:=\{x\in \R^n: x_i\geq 0 \textrm{ for all } i=1,\ldots,n\}$$
is a proper convex cone in $\R^n$. The  interior of $\R^n_+$ is the set
$$\R^n_{++}:=\{x\in \R^n: x_i> 0 \textrm{ for all } i=1,\ldots,n\}$$
\end{example}

\begin{example}[The positive semi-definite cone ${\bf S}^n_+$]\label{exa:sdpp}
Consider the set of positive semi-definite matrices
$${\bf S}^n_+:=\{X\in {\bf S}^n: v^TXv\geq 0 \textrm{ for all } v\in \R^n\}.$$
It is immediate from the definition that ${\bf S}^n_+$ is a convex cone containing no lines.
 Let us quickly verify that ${\bf S}^n_+$ is proper. To see this,
 observe $$v^TXv=\trace(v^TXv)=\trace(Xvv^T)=\langle X,vv^T\rangle.$$
Thus ${\bf S}^n_+$ is closed because it is the intersection of the halfspaces $\{X\in {\bf S}^n: \langle X,vv^T \rangle\geq 0\}$ for all $v\in \R^n$, and arbitrary intersections of closed sets are closed. The interior of ${\bf S}^n_+$ is the set of positive definite matrices
$${\bf S}^n_{++}:=\{X\in {\bf S}^n: v^TXv> 0 \textrm{ for all } 0\neq v\in \R^n\}.$$
Let us quickly verify this description. Showing that ${\bf S}^n_{++}$ is open is straightforward; we leave the details to the reader. Conversely, consider a matrix $X\in {\bf S}^n_{+}\setminus {\bf S}^n_{++}$ and let $v$ be a nonzero vector satisfying  $v^TXv=0$. Then the matrix $X-tvv^T$  lies outside of ${\bf S}^n_{+}$ for every $t >0$, and therefore $X$ must lie on the boundary of  ${\bf S}^n_{+}$.
To summarize, we have shown that ${\bf S}^n_+$ is a proper convex cone.
\end{example}

\index{proper convex cone}
Given a convex cone $\K$ in $\E$, we introduce  two binary relations
$\succeq_{\K}$ and $\succ_{\K}$ on $\E$:
\begin{align*}
x\succeq_{\K} y \quad&\Longleftrightarrow \quad x-y \in \K,\\
x\succ_{\K} y \quad&\Longleftrightarrow \quad  x-y\in\inter \K.
\end{align*}
Assuming that $\K$ is proper makes the relation $\succeq_{\K}$ into a
\textdef{partial order}, meaning that for any three points $x,y,z\in \E$, the three conditions hold:
\begin{enumerate}
\item (reflexivity) $\quad x\succeq_{\K} x$
\item (antisymmetry) $\quad x\succeq_{\K} y~\textrm{ and }~ y \succeq_K x\quad \implies \quad x =y$
\item (transitivity) $\quad x\succeq_{\K} y~\textrm{ and }~ y \succeq_K z\quad \implies \quad x \succeq_{\K}z.$
\end{enumerate}

As is standard in the literature, we denote the partial order
$\succeq_{\R^n_+}$ on $\R^n$ by $\geq$ and the partial order
$\succeq_{{\bf S}^n_+}$ on ${\bf S}^n$ by $\succeq$. In particular, the
relation $x\geq y$ means $x_i\geq y_i$ for each coordinate $i$, while
the relation $X\succeq Y$ means that the matrix $X-Y$ is positive
semi-definite.

Central to conic geometry is duality.
The \textdef{dual cone}  
of $\K$ is the set
$$\K^*:=\{y\in\E: \langle x,y\rangle \geq 0 \textrm{ for all } x\in\K\}.$$   
\index{$K^*$, polar cone}
\index{$K^*$, dual cone}
The following lemma will be used extensively.
\begin{lemma}[Self-duality]
Both $\R^n_+$ and ${\bf S}^n_+$ are self-dual, meaning $(\R^n_+)^*=\R^n_+$ and $({\bf S}^n_+)^*={\bf S}^n_+$.
\end{lemma}
\begin{proof}
The equality $(\R^n_+)^*=\R^n_+$ is elementary and we leave the proof to the reader. To see that ${\bf S}^n_+$ is self-dual, recall that a matrix $X\in {\bf S}^n$ is positive semi-definite if, and only if, all of its eigenvalues are nonnegative. Fix two matrices $X,Y\succeq 0$ and let $X=\sum_i\lambda_iv_iv_i^T$ be the eigenvalue decomposition of $X$. Then we deduce
\[
\langle X,Y\rangle =\trace XY=\sum_i\lambda_i \trace   (v_iv_i^T Y) =\sum_i\lambda_i (v_i^T Yv_i)\geq 0.
\]
Therefore the inclusion ${\bf S}^n_+\subseteq ({\bf S}^n_+)^*$ holds.
Conversely, for any $X\in ({\bf S}^n_+)^*$ and any $v\in \R^n$
the inequality,
$
0\leq \langle X,vv^T\rangle=v^TXv
$, holds.
The reverse inclusion ${\bf S}^n_+\supseteq ({\bf S}^n_+)^*$ follows, and the proof is complete.
\end{proof}

%
\index{$\K^{**}:=(\K^*)^*$, second polar cone}
Finally, we end this section with the following two useful results of convex geometry.
\begin{lemma}[Dual cone of a sum]\label{lem:dual_sum}
For any two closed convex cones $\KK_1$ and $\KK_2$, equalities hold: 
\begin{align*}
(\KK_1+\KK_2)^*=\KK_1^*\cap\KK_2^* \qquad \textrm{and}\qquad
	(\KK_1 \cap \KK_2)^*=\cl \left(\KK_1^* + \KK_2^*\right).
	\end{align*}
\end{lemma}

\begin{lemma}[Double dual]
	A set $\K\subset \E$ is a closed convex cone 
	if, and only if, equality
	$\K=(\K^{*})^*$ holds.
\end{lemma}

In particular, if $\K$ is a proper convex cone, then so is its dual $\K^*$, as the reader can verify.

\section{Facial geometry}
Central to this paper is the decomposition of a cone into faces.

\begin{definition}[Faces]
Let $\K$ be a convex cone.  A convex cone $F\subseteq \K$ is called a {\em face of} $\K$, denoted $F\unlhd
\K$, if the implication holds:
\index{$F\unlhd \K$, face of a cone}
\index{face of a cone, $F\unlhd \K$}
\[
	x,y\in \K,~ x+y \in F \quad\implies\quad x,y \in F.
\]
\end{definition}
Let $\mathcal{K}$ be a closed convex cone.
Vacuously, the empty set and $\K$ itself are faces. 
A face $F\unlhd \K$ is {\em proper} if it is neither empty nor all of
$\K$. One can readily verify  from the definition that the intersection of an arbitrary collection of faces of $\K$ is
itself a face of $\K$. A fundamental result of convex geometry shows that relative interiors of all faces of $\K$ form a partition of $\K$: every point of $\K$ lies in the relative interior of some face and relative interiors of any two distinct faces are disjoint. In particular, any proper face of 
$\K$ is disjoint from $\ri \K$.
\index{proper face}

\begin{definition}[Minimal face]
 The \emph{minimal face} of a convex cone $\mathcal{K}$ containing a set $S\subseteq \K$ is the intersection of all faces of $\K$ containing $S$, and is denoted by $\face(S, \K)$.
\end{definition}

A convenient alternate characterization of minimal faces is as follows. If $S\subseteq \KK$ is
a convex set, then $\face(S, \K)$ is the smallest face of $\K$ intersecting the relative interior of $S$. In particular, equality $\face(S, \K)=\face(x, \K)$ holds for any point $x\in \ri S$.
\index{minimal face, $\face(S, \K)$}
\index{$\face(S, \K)$, minimal face}

There is a special class of faces that admit ``dual'' descriptions.
Namely, for any vector $v\in \K^*$ one can readily verify that the set $F= v^{\perp}\cap\K$ is a face of $\K$. 
\begin{definition}[Exposed faces]
Any set of the form $F= v^{\perp}\cap\K$, for some vector $v\in \mathcal{K}^*$, is called an \textdef{exposed face} of $\K$. The vector $v$ is then called an \textdef{exposing vector} of $F$. 
\end{definition}

The classical hyperplane separation theorem shows that any point $x$ in the relative boundary of $\K$ lies in some proper exposed face.
Not all faces are exposed, however, as the following example shows.

\begin{example}[Nonexposed faces]
Consider the set $Q=\{(x,y)\in \R^2: y\geq \max(0,x)^2\}$, and let $\K$ be the closed convex cone generated by $Q\times \{1\}$. Then the ray $\{(0,0)\}\times \R_+$ is a face of $\K$ but it is not exposed.
\begin{figure}[h]
\begin{center}
\begin{tikzpicture}[scale=2]
\fill[fill=gray] (-1,1) -- plot [domain=-1:1] ({\x},{max(0,\x)*max(0,\x)}) -- (-1,1) -- cycle;
\draw plot[domain=-1:1] ({\x},{max(0,\x)*max(0,\x)});
\end{tikzpicture}
\end{center}
\caption{The set $Q$.}
\end{figure}
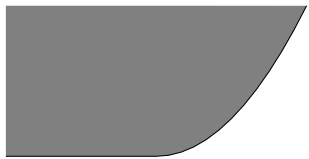
\end{example}

The following is a very useful property of exposed faces we will use.
\begin{proposition}[Exposing the intersection of exposed faces]\label{prop:inter_face_exp} {\hfill \\ }
	For any closed convex cone $\K$ and vectors $v_1, v_2\in \KK^*$, equality holds:
	$$(v_1^{\perp}\cap \K)\cap (v_2^{\perp}\cap \K)=(v_1+v_2)^{\perp}\cap \KK.$$
\end{proposition}
\begin{proof}
	The inclusion $\subseteq$ is trivial. To see the converse, note for any $x\in(v_1+v_2)^{\perp}\cap\K$ we have $\langle v_1,x\rangle\geq 0$, $\langle v_2,x\rangle\geq 0$, while $\langle v_1+v_2,x\rangle=0$. We deduce $x\in v^{\perp}_1\cap v^{\perp}_2$ as claimed.
\end{proof}

In other words, if the faces $F_1\unlhd\K$ and $F_2\unlhd\K$ are exposed by $v_1$ and $v_2$, respectively, then the intersection $F_1\cap F_2$ is a face exposed by the sum $v_1+v_2$.

A convex cone is called \textdef{facially exposed} if all of its faces are exposed. The distinction between faces and exposed faces may appear mild at first sight; however, we will see that it is exactly this distinction that can cause difficulties for preprocessing techniques for general conic problems.

\begin{definition}[Conjugate face]{\hfill \\  }
With any face $F$ of a convex cone $\K$, we associate a face of the dual cone $\K^*$,
called the \textdef{conjugate face, $F^{\triangle}:=\K^*\cap F^{\perp}$}. 
\index{$F^{\triangle}:=\K^*\cap F^{\perp}$, conjugate face} 
\end{definition}
Equivalently, $F^{\triangle}$ is the face of $\K^*$ exposed by any point
$x\in \ri F$, that is $F^{\triangle}=\K^*\cap x^{\perp}$. Thus, in
particular, conjugate faces are always exposed. Not surprisingly, one can readily verify that
equality $(F^{\triangle})^{\triangle} =F$ holds if, and only if, the face $F\unlhd \K$ is exposed.

We illustrate the concepts with our two running examples, $\R^n_+$ and ${\bf S}^n_+$, keeping in mind the parallels between the two.

\begin{example}[Faces of $\R^n_+$]
For any index set $I\subseteq \{1,\ldots,n\}$, the set 
$$F_I=\{x\in \R^n_+: x_i=0 \textrm{ for all }i\in I\}$$
is a face of $\R^n_+$, and all faces of  $\R^n_+$ are of this form.
 In particular, observe that all faces of $\R^n_+$ are linearly
isomorphic to $\R^k_+$ for some positive integer $k$. In this sense, 
$\R^n_+$ is ``\textdef{self-replicating}''. 
 The relative interior of $F_I$ consists of all points in $F_I$ with $x_i>0$ for indices $i\notin I$.
The face $F_I$ is exposed by the vector $v\in \R^n_+$ with $v_i=1$ for all $i\in I$ and $v_i=0$ for all $i\notin I$. In particular, $\R^n_+$ is a facially exposed convex cone. The face conjugate to $F_I$ is $F_{I^c}$.
\end{example}

\begin{example}[Faces of ${\bf S}^n_+$]\label{exa:face_snp}
There are a number of different ways to think about (and represent) faces of the PSD cone ${\bf S}^n_+$. In particular, one can show that faces ${\bf S}^n_+$ are in correspondence with linear subspaces of $\R^n$. More precisely,
for any $r$-dimensional linear subspace $\mathcal{R}$ of $\R^n$, the set 
\begin{equation}\label{eqn:form}
F_{\mathcal{R}}:=\{X\in {\bf S}^n_+: \range X\subseteq \mathcal{R}\}
\end{equation}
is a face of ${\bf S}^n_+$.   Conversely, any face of ${\bf S}^n_+$ can be 
written in the form \eqref{eqn:form}, where $\RR$ is the range space of any matrix $X$ lying in the relative interior of the face. The relative interior of $F_{\mathcal{R}}$ consists of all matrices $X\in{\bf S}^n_+$ whose range space coincides with $\RR$.
Moreover, for any matrix $V\in\R^{n\times r}$ satisfying $\range V=\RR$, we have the equivalent description
$$F_{\mathcal{R}}=V{\bf S}^r_+V^T.$$ In particular, $F_{\mathcal{R}}$ is linearly isomorphic to
the $r$-dimensional positive semi-definite cone ${\bf S}^r_+$. The face conjugate to $F_{\mathcal{R}}$ is $F_{\mathcal R^{\perp}}$ and can be equivalently written as $$F_{\mathcal R^{\perp}}=U {\bf S}^{n-r}_+U^T,$$ 
for any matrix $U\in\R^{n\times (n-r)}$ satisfying $\range U=\RR^{\perp}$. 
Notice that then the matrix $UU^T$ lies in the relative interior of
$F^{\triangle}$ and therefore $UU^T$ exposes the face $F_{\mathcal{R}}$.
In particular,  ${\bf S}^n_+$ is \textdef{facially exposed} and also
\textdef{self-replicating}.

\end{example}

%
%

\section{Conic optimization problems}
\label{sect:conicopt}
Modern conic optimization draws fundamentally from ``duality'': every conic optimization
problem gives rise to a related conic optimization problem, called its dual. 
Consider the \textdef{primal-dual pair}:
\begin{equation}
	\label{eq:basicpd}
\begin{array}{cc}
\Pprob \,
\begin{array}{cc}
	 \inf &\langle c,x\rangle   \\
			  \textrm{ s.t.}  &Ax=b      \\
	    &x\succeq_{\K} 0		
\end{array}    &
\,
\qquad\qquad
 \Dprob \,
\begin{array}{rcc}
	 \sup &\langle b,y\rangle \\
		   \textrm{s.t.} & A^*y\preceq_{\K^*} c  \\
\end{array}
\end{array}
\end{equation}
Here, $\K$ is a closed convex cone in $\E$ and the mapping $A\colon\E\to\F$ is linear. Eliminating the trivial case that the system $Ax=b$ has no solution, we will always assume that $b$ lies in $\Range A$, and that $\K$ has nonempty interior. Two examples of conic
optimization are of main importance for us: linear programming  (\LPp)\
corresponds to $\K=\R^n_+$, $\F=\R^m$ and semi-definite programming (\SDPp)\ corresponds to $\K={\bf S}^n_+$,  $\F=\R^m$. The adjoint $A^*$ in both cases was computed in Examples~\ref{exa:nonneg_ortho} and \ref{exa:sdpp}.
We will also use the following notation for the {\em primal} and {\em dual feasible regions}:
$$\mathcal{F}_p:=\{x\succeq_{\KK} 0: Ax=b\}\qquad \textrm{and}\qquad \mathcal{F}_d:=\{y:  A^*y\preceq_{\K^*} c\}.$$
It is important to note that the dual can be put in a primal form by introducing {\em slack variables} $s\in\E$ leading to the equivalent formulation
\begin{equation}\label{eqn:dual_primal}
\begin{array}{rcc}
	\sup &\langle b,y\rangle \\
	\textrm{s.t.} & A^*y+s=c  \\
	& y\in {\bf F}, s\in \KK^*.
\end{array}
\end{equation}

To a reader unfamiliar with conic optimization, it may be unclear how the dual arises naturally from the primal. Let us see how it can be done. The dual problem \Dprobp\, can be discovered through ``Lagrangian duality'' in convex optimization.  Define the 
	\textdef{Lagrangian function}
	$$L(x,y):=\langle c,x\rangle +\langle y,b-Ax\rangle,$$ 
	and observe the equality
	\[ \max_{y}~ L(x,y)=\begin{cases} 
      \langle c,x\rangle & Ax=b \\
      +\infty & \textrm{otherwise} 
   \end{cases}.
\]
	Thus the primal problem \Pprobp\, is equivalent to
	\[
		\min_{x\succeq_\K0}~ \left(\max_y\, L(x,y)\right).
	\]
	Formally exchanging  min/max, yields exactly the dual problem \Dprobp\,
	\begin{align*}
		\max_y \left(\min_{x\succeq_\K0} L(x,y)\right)&=\max_y \,
		\left(\langle b,y\rangle +\min_{x\succeq_\K0}\, \langle
x,c-A^*y\rangle \right) \\
		&=\max\,
		\{\langle b,y\rangle: A^*y\preceq_{\K^*} c\}.
	\end{align*}
The primal-dual pair always satisfies the \textdef{weak duality inequality}: 
for any primal feasible $x$ and any dual feasible $y$, we have
\begin{equation}
	\label{eq:weakdual}
0\leq \langle  c-A^*y,x \rangle = \langle c,x\rangle -\langle b,y\rangle.
\end{equation}
Thus for any feasible point $y$  of the dual, its objective value $\langle b,y\rangle$ lower-bounds the optimal value of the primal.
The weak duality inequality \eqref{eq:weakdual}  leads to the following sufficient conditions for optimality.
\begin{proposition}[Complementary slackness]
	\label{prop:optsuff}{\hfill \\ }
Suppose that $(x,y)$ are a primal-dual feasible pair for
\Pprob,\,\Dprobp and suppose that \textdef{complementary
slackness} holds:
\[
0= \langle  c-A^*y,x \rangle.
\]
Then $x$ is a minimizer of \Pprobp\ and $y$ is a maximizer of \Dprob.
\end{proposition}
The sufficient conditions for optimality of Proposition \ref{prop:optsuff} are often summarized as the primal-dual system:
\begin{equation}
	\label{eq:optcondone}
	\begin{array}{rcll}
		A^*y+s-c&=& 0    & \text{(dual feasibility)}\cr
		Ax-b&=& 0    & \text{(primal feasibility)}\cr
		\langle s, x \rangle &=& 0    & \text{(complementary slackness)}\cr
		s\succeq_{\K^*}0,\, x\succeq_{\K} 0&& &
		\text{(nonegativity)}.
	\end{array}
\end{equation}

Derivations of algorithms generally require necessary optimality
conditions, i.e.,~failure of the necessary conditions at a current
approximation of the optimum leads to improvement steps.
When are sufficient conditions for optimality expressed in
Proposition~\ref{prop:optsuff} necessary? In other words, when can we be
sure that optimality of a primal solution $x$ can be certified by the
existence of some dual feasible point $y$, such that the pair satisfies the 
complementary slackness condition? Conditions guaranteeing existence of
Lagrange multipliers are called 
\textdef{constraint qualifications}. The most important condition of
this type is called \textdef{strict feasibility}, or {\em Slater's condition}, and is  the main topic of this article. 
\begin{definition}[Strict feasibility/Slater condition]
We say that \Pprobp is {\em strictly feasible} if there exists a point 
$x\succ_{\K} 0$ satisfying $Ax=b$. The dual  \Dprobp  is {\em strictly 
feasible} if there exists a point $y$ satisfying $A^*y\prec_{\K^*} c$. 
\end{definition}

The following result is the cornerstone of conic optimization.

\begin{theorem}[Strong duality]\label{thm:str_dual}
If the primal objective value is finite and the problem \Pprobp is
strictly feasible, then the primal and dual optimal values
are equal, and the dual \Dprobp admits an optimal solution. In addition, 
for any $x$ that is optimal for \Pprob, there exists a vector 
$y$ such that $(x,y)$ satisfies complementary slackness. 

 Similarly, if the dual objective value is finite and the dual \Dprobp satisfies strict feasibility, then the primal and dual optimal values are equal and the primal \Pprobp admits an optimal solution. 
In addition, for any $y$ that is optimal for \Dprob,
there exists a point $x$ such that $(x,y)$ satisfies complementary slackness. 
\end{theorem}

Without a constraint qualification such as strict feasibility, 
the previous theorem is decisively false. The following examples show that 
without strict feasibility, the primal and dual optimal values may not even 
be equal, and even if they are equal, the optimal values may be unattained.
 \begin{example}[Infinite gap]
	 \label{ex:lackattain}
Consider the following primal \SDP in \eqref{eq:basicpd}:
\[
  0=v_p=\min\limits_{x\in {\bf S}^2}
\left\{
  2X_{12}\ :\ 
        X_{11}=0,\ X  \succeq 0
\right\}.
\]
The corresponding dual \SDP is the infeasible problem
\[
   -\infty=v_d=\max\limits_{y\in\R}\;
\left\{
  0y :\
  \begin{pmatrix}
y&0\cr 0&0
\end{pmatrix}
 \preceq 
\begin{pmatrix}
0&1\cr 1&0
\end{pmatrix}
\right\}.
\]
Both the primal and the dual fail strict feasibility in this example.
\end{example}
 \begin{example}[Positive duality gap]
	 \label{ex:posdualgap}
Consider the following primal \SDP in \eqref{eq:basicpd}:
\[
  v_p=\min\limits_{X\in {\bf S}^3}
\left\{
  X_{22}\ :\ 
        X_{33}=0,\ 
  X_{22}+ 2 X_{13}=1,\ 
 X  \succeq 0
\right\}.
\]
The constraint $X\succeq 0$ with $X_{33}=0$ implies equality $X_{13}=0$, and hence $X_{22}=1$.
Therefore, $v_p=1$ and the matrix $e_2e_2^T$ is optimal.

The corresponding dual \SDP is
\begin{equation}
	\label{eq:dualsdpdualgap}
   v_d=\max\limits_{y\in \R^2}\;
\left\{
  y_2\ :\
  \begin{pmatrix}
0&0&y_2\cr 0&y_2&0\cr y_2 &0&y_1
\end{pmatrix}
 \preceq 
\begin{pmatrix}
0&0&0\cr 0&1&0\cr 0&0&0
\end{pmatrix}
\right\}.
\end{equation}
This time the \SDP constraint implies $y_2=0$. We deduce that $(y_1,y_2)=(0,0)$ is optimal for the dual and hence $v_d=0<v_p=1$.
There is a finite duality gap between the primal and dual problems. The culprit again is that both the primal and the dual fail strict feasibility.

	 \begin{example}[Zero duality gap, but no attainment]
		 \label{ex:zerodualgap}
		 Consider the dual \SDP
\[
	\begin{array}{rclcc}
		v_d &=& \sup  & y \\
                     && \text{s.t.} &
		     y\begin{bmatrix}
			     0 & 1  \cr
			     1 & 0  \cr
		     \end{bmatrix} \preceq
		     \begin{bmatrix}
			     1 & 0  \cr
			     0 & 0  
		     \end{bmatrix}.
	\end{array}
\]
The only feasible point is $y=0$. Thus the optimal value is $v_d=0$ and
is \emph{attained}. The primal \SDP is
\[
	\begin{array}{rclcc}
		v_p &=& \inf  & X_{11} \\
		     && \text{s.t.} &  2X_{12}=1 \\
		     &&             &    X \in \Ss^2_+.
	\end{array}
\]
Notice $X_{11}> 0$ for all feasible $X$.
On the other hand, the sequence $X^k= \begin{bmatrix} 1/k & 1/2  \cr 1/2 & k
\end{bmatrix}$ is feasible and satisfies $X^k_{11}\to 0$. Thus there is
		 no duality gap, meaning $0=v_p=v_d$, but the primal
		 optimal value is {\em not attained}. The culprit is
		 that the dual \SDP  is not strictly feasible.

	 \end{example}

\begin{example}[Convergence to dual optimal value]
Numerical solutions of problems inevitably suffer from some
perturbations of the data, that is a perturbed problem is in fact solved.
	Moreover, often it is tempting to explicitly perturb a constraint in the problem, so that strict feasibility
	holds. This example shows that this latter strategy results in the dual of the problem
	being solved, as opposed to the problem under consideration.
	
	We consider the primal-dual \SDP pair in Example
	\ref{ex:posdualgap}. In particular, suppose first that we want to
	solve the dual problem. We canonically perturb the right-hand
	side of the dual in \eqref{eq:dualsdpdualgap}
	\[
	v_d(\epsilon):=\sup\limits_{y\in \R^2}\;
	\left\{
	y_2\ :\
	\begin{pmatrix}
	0&0&y_2\cr 0&y_2&0\cr y_2 &0&y_1
	\end{pmatrix}
	\preceq 
	\begin{pmatrix}
	0&0&0\cr 0&1&0\cr 0&0&0
	\end{pmatrix} + \epsilon P \right\}, 
	\]
	for some matrix $P\succ 0$ and real $\epsilon >0$. 
Strict feasibility now holds and we \emph{hope} that the optimal values of the
\textdef{perturbed problems} converge to that of the original one $v_d(0)$.
We can rewrite feasibility for the perturbed problem as
\begin{equation}
\label{eq:pertfeasdual}
	\begin{pmatrix}
	0&0&-y_2\cr 0&1-y_2&0\cr -y_2 &0&-y_1
	\end{pmatrix}+\epsilon P \succeq 0.
\end{equation}
A triple $(y_1,y_2,y_3)$ is strictly feasible
if, and only if, the leading principal minors $M_{11},M_{12},M_{123}$ 
of the left-hand side matrix in \eqref{eq:pertfeasdual}
are all positive. We have $M_{11}=\epsilon P_{11} > 0$.
The second leading principal minor as a function of $y_2$ is
\[
M_{12}(y_2)=
\epsilon\left( P_{11}(1-y_2)+\epsilon(P_{11}P_{22}-P_{12}^2)\right).
\] 
In particular, rearranging we have $M_{12}(y_2)>0$ whenever
\begin{equation}
\label{eq:minortwopos}
y_2 <  \frac 
{P_{11} + \epsilon(P_{11}P_{22}-P_{12}^2)}{P_{11}}.
\end{equation}
The last minor $M_{123}$ is positive for sufficiently negative 
$y_1$ by the Schur
complement. Consequently the perturbed problems satisfy
$$v_d(\epsilon)=1+
  \epsilon \frac{P_{11}P_{22}-P_{12}^2}{P_{11}}$$
and therefore
$$0=v_d<\lim_{\epsilon \to 0} v_d(\epsilon)=1=v_p.$$
That 
	is the primal optimal value is obtained in the limit rather than the dual optimal
	value that is sought.
	
	Let us look at an analogous perturbation to the primal problem. Let $A$ be the linear operator $A(X)=(X_{33},X_{22}+2X_{13})$ and set $b=(0,1)$.
	Consider the perturbed problems 
	$$v_p(\epsilon)=\min_{X\in {\bf S}^3} \{X_{22}: AX=b+\epsilon AP\}$$ for some fixed real $\epsilon>0$ and a matrix $P\succ 0$. Each such problem is strictly feasible, since the positive definite matrix $\widehat{X}+\epsilon P$ is feasible for any matrix $\widehat{X}$ that is feasible for the original primal problem.
	
	 In long form, the perturbed primal problems are
	\[
	v_p(\epsilon)=\min\limits_{X\succeq 0}
	\left\{
	X_{22}\ :\ 
	X_{33}=\epsilon P_{33},\ 
	X_{22}+ 2 X_{13}=1 + \epsilon(P_{22}+2P_{13})
	\right\}.
	\]
	Consider the matrix
	$$X=\begin{bmatrix}
	 X_{11} & 0 & 1/2+\epsilon (P_{22}+2P_{13})/2\\
	 0 & 0& 0\\
	 1/2+\epsilon (P_{22}+2P_{13})/2 & 0 & \epsilon P_{33}
	\end{bmatrix}.
	$$
	This matrix satisfies the linear system $AX=b$ by construction and is positive semi-definite for all sufficiently large $X_{11}$.  
We deduce	$v_p(\epsilon)=0=v_d$ for all $\epsilon >0$. Again, as $\epsilon$ tends to zero we obtain the dual optimal value
	rather than the sought after primal optimal value.
\end{example}


\end{example}

\section{Commentary}
\label{sect:relworkconvgeom}
We follow here well-established notation in convex optimization, as
illustrated for example in the monographs of Barvinok~\cite{barv_book},
Ben-Tal-Nemirovski~\cite{MR1857264}, Borwein-Lewis~\cite{BoLe:00}, and
Rockafellar~\cite{con:70}. The handbook of \SDP~\cite{SaVaWo:97} and
online lecture notes~\cite{laurentvallentin:16} are other excellent sources in
the context of semi-definite programming. These include 
discussion on the facial structure. 
The relevant results stated in the text can all be found for instance in Rockafellar~\cite{con:70}. The example \ref{ex:posdualgap} is a modification of the example in~\cite{Ram:93}.
In addition, note that the three examples \ref{ex:lackattain}, 
		 \ref{ex:posdualgap}, \ref{ex:zerodualgap}
have matrices with the special \textdef{perdiagonal} structure. The
universality of such special structure in ``ill-posed'' \SDPp s has recently been investigated in great length by Pataki~\cite{Gpat:11}.

\chapter{Virtues of strict feasibility}
\label{chap:virtuesstrfeas}

We have already seen in Theorem~\ref{thm:str_dual} that strict feasibility is essential to guarantee dual attainment and therefore for making the primal-dual optimality conditions \eqref{eq:optcondone} meaningful. In this section, we continue discussing the impact of strict feasibility on numerical stability. We begin with the theorems of the alternative, akin to the Farkas' Lemma in linear programming, which quantify the extent to which strict feasibility holds. We then show how such systems appear naturally in stability measures of the underlying problem.

\section{Theorem of the alternative}
The definition we have given of strict feasibility (Slater) is {\em qualitative} in nature, that is it involves no measurements. A convenient way to measure the extent to which strict feasibility holds (i.e. its strength) arises from dual characterizations of the property. We will see that strict feasibility corresponds to inconsistency of a certain {\em auxiliary system}. 
Measures of how close the auxiliary system is to being consistent yield estimates of ``stability'' of the problem. 

The aforementioned dual characterizations stem from the basic hyperplane separation theorem for convex sets.

\begin{theorem}[Hyperplane separation theorem]\label{hype:sep_thm}
Let $Q_1$ and $Q_2$ be two disjoint nonempty convex sets. Then there exists a nonzero vector $v$ and a real number $c$ satisfying
$$\langle v,x\rangle\leq c\leq \langle v,y\rangle\qquad \textrm{for all }x\in Q_1 \textrm{ and }y\in Q_2.$$
\end{theorem}

When one of the sets is a cone, the separation theorem takes the following ``homogeneous'' form.

\begin{theorem}[Homogeneous separation]\label{thm:sep_1}
Consider a nonempty closed convex set $Q$ and a closed convex cone $\K$ with nonempty interior. Then exactly one of the following alternatives holds.
\begin{enumerate}
\item\label{it:slat1} The set $Q$ intersects the interior of $\K$.
\item\label{it:slat2} There exists a vector $0\neq v\in \K^*$ satisfying $\langle v,x\rangle\leq 0$ for all $x\in Q$. 
\end{enumerate}
Moreover, for any vector $v$ satisfying the alternative \ref{it:slat2}, the region $Q\cap \K$ is contained in the proper face $v^{\perp}\cap \K$. 
\end{theorem}
\begin{proof}
Suppose that $Q$ does not intersect the interior of $\K$. Then the convex cone generated by $Q$, denoted by $\cone Q$, does not intersect $\inter\K$ either.
The hyperplane separation theorem (Theorem~\ref{hype:sep_thm}) shows that there is a nonzero vector $v$ and a real number $c$ satisfying
$$\langle v,x\rangle\leq c\leq \langle v,y\rangle \quad \textrm{ for all }x\in \cone Q\textrm{ and } y\in  \inter \K.$$
Setting $x=y=0$, we deduce $c= 0$. Hence $v$ lies in $\K^*$ and \ref{it:slat2} holds.

Conversely, suppose that \ref{it:slat2} holds. Then the inequalities
$0\leq \langle v,x\rangle \leq 0$  hold for all $x\in Q\cap \KK$. Thus we deduce that the intersection $Q\cap \KK$ lies in the proper face $v^{\perp}\cap \KK$. Hence the alternative  \ref{it:slat1} can not hold.
\end{proof}

Let us now specialize the previous theorem to the primal problem \Pprob, by letting $Q$ be the affine space $\{x: Ax=b\}$. Indeed, this is the main result of this subsection and it will be used extensively in what follows.

	\begin{theorem}[Theorem of the alternative for the primal]
	 \label{thm:primalthmalt}
Suppose that $\KK$ is a closed convex cone with nonempty interior. Then 
	exactly one of the following alternatives holds.
\begin{enumerate}
\item 
	\label{item:altaxibxgt}
	The primal \Pprobp is strictly feasible.
\item 
	\label{item:altya}
	The auxiliary system is consistent:
\begin{equation}
	\label{eq:auxsysprim}
	0\neq A^*y \succeq_{\KK^*} 0 \qquad\textrm{and}\qquad \langle b,y\rangle\leq 0.
\end{equation}
\end{enumerate}
Suppose that the primal \Pprobp is feasible. Then the auxiliary system \eqref{eq:auxsysprim} is equivalent to the system
\begin{equation}
\label{eq:auxsysprim_feas}
0\neq A^*y \succeq_{\KK^*} 0 \qquad\textrm{and}\qquad \langle b,y\rangle= 0.
\end{equation}
Moreover, then  any vector $v$ satisfying either of the equivalent systems, \eqref{eq:auxsysprim} and \eqref{eq:auxsysprim_feas}, yields a proper face $(A^*y)^{\perp}\cap \KK$ containing the primal  feasible region
 $\FF_p$.
	\end{theorem}
	\begin{proof}
Set $Q:=\{x: Ax=b\}$. 
Clearly strict feasibility of \Pprobp is equivalent to alternative \ref{it:slat1} of Theorem~\ref{thm:sep_1}. Thus it suffices to show that the auxiliary system \eqref{eq:auxsysprim} is equivalent to the alternative  \ref{it:slat2} of Theorem~\ref{thm:sep_1}. To this end, note that  for any vector $y$ satisfying  \eqref{eq:auxsysprim}, the vector $v:=A^*y$  satisfies the alternative  \ref{it:slat2} of Theorem~\ref{thm:sep_1}. Conversely, consider a vector $0\neq v\in \KK^*$ satisfying $\langle v,x\rangle\leq 0$ for all $x\in Q$. Fix a point $\hat x\in Q$ and observe the equality $Q=\hat x+\Null(A)$. An easy argument then shows that $v$ is orthogonal to $\Null(A)$, and therefore can be written as $v=A^*y$ for some vector $y$. We deduce $\langle b,y\rangle=\langle A\hat x,y \rangle=\langle \hat x,v\rangle\leq 0$, and therefore $y$ satisfies \eqref{eq:auxsysprim}.

Next, assume that \Pprobp is feasible.
Suppose $y$ satisfies \eqref{eq:auxsysprim}. Then for any feasible point $x$ of \Pprob, we deduce $0\leq \langle x,A^*y\rangle=\langle b,y\rangle\leq 0$.  Thus $y$ satisfies the system \eqref{eq:auxsysprim_feas}.  It follows that the two systems  \eqref{eq:auxsysprim} and \eqref{eq:auxsysprim_feas} are equivalent and that the proper face $(A^*y)^{\perp}\cap \KK$ contains the primal feasible region $\FF_p$, as claimed.
	\end{proof}

Suppose that $\mathcal{F}_p$ is nonempty. Then if strict feasibility
fails, there always exists a ``witness'' (or ``short certificate'') $y$
satisfying the auxiliary system \eqref{eq:auxsysprim_feas}. Indeed,
given such a vector $y$, one immediately deduces, as in the proof, that
$\mathcal{F}_p$ is contained in the proper face $(A^*y)^{\perp}\cap \KK$
of $\KK$. Such certificates will in a later section be used
constructively to regularize the conic problem through the \FR procedure.

The analogue of Theorem~\ref{thm:primalthmalt} for the dual \Dprobp quickly follows.

\begin{theorem}[Theorem of the alternative for the dual]
	\label{thm:dualthmalt}
	Suppose that $\KK^*$ has nonempty interior. Then
	exactly one of the following alternatives holds.
	\begin{enumerate}
		\item 
		\label{item:altaxibxgt_dual}
		The dual \Dprobp is strictly feasible.
		\item 
		\label{item:altya_dual}
		The auxiliary system is consistent:
		\begin{equation}
		\label{eq:auxsysprim_dual}
		0\neq x \succeq_{\KK} 0,\quad Ax=0,  \quad\textrm{and}\quad \langle c,x\rangle\leq 0.
		\end{equation}
	\end{enumerate}
	Suppose that the dual \Dprobp is feasible. Then the auxiliary system \eqref{eq:auxsysprim_dual} is equivalent to the system
	\begin{equation}
	\label{eq:auxsysprim_feas_dual}
	0\neq x \succeq_{\KK} 0,\quad Ax=0,  \quad\textrm{and}\quad \langle c,x\rangle= 0.
	\end{equation}
	Moreover, for any vector $x$ satisfying either of the equivalent systems, \eqref{eq:auxsysprim_dual} and \eqref{eq:auxsysprim_feas_dual}, yields a proper face $x^{\perp}\cap \KK^*$ containing the feasible slacks $\{c-A^*y: y\in\FF_d\}$.	
\end{theorem}
\begin{proof}
	Apply Theorem~\ref{thm:primalthmalt} to the equivalent formulation \eqref{eqn:dual_primal} of the dual \Dprob.
\end{proof}

\section{Stability of the solution}
In this section, we explain the impact of strict feasibility on stability of the conic optimization problem through quantities naturally arising from the auxiliary system \eqref{eq:auxsysprim}. For simplicity we focus on the primal problem \Pprob, though an entirely analogous development is possible for the dual, for example by introducing slack variables. 

We begin with a basic question: at what rate does the optimal value of the primal problem \Pprobp change relative to small perturbations of the right-hand-side of the linear equality constraints? To formalize this question, define the {\em value function} 
\begin{align*}
v(\Delta):=\quad\inf~~ &\langle c,x\rangle   \\
\textrm{ s.t.}~~  &Ax=b+\Delta      \\
&x\succeq_{\K} 0.		
\end{align*}
The value function $v\colon\F\to [-\infty,+\infty]$  
thus defined is convex, meaning that its epigraph
$$\epi v:=\{(x,r)\in \F\times\R: v(x)\leq r\}$$
is a convex set. Seeking to understand stability of the primal  \Pprobp under perturbation of the right-hand-side $b$, it is natural to examine the variational behavior of the value function. 
There is an immediate obstruction, however. If $A$ is not surjective, then there are arbitrarily small perturbations of $b$ making $v$ take on infinite values. As a result, in conjunction with strict feasibility, we will often make the mild assumption that $A$ is surjective. These two properties taken together are refereed to as the  {\em Mangasarian-Fromovitz Constraint Qualification (MFCQ)}.

\begin{definition}[Mangasarian-Fromovitz \CQp]
We say that the \textdef{Mangasarian-Fromovitz Constraint Qualification
(\MFCQp)} holds for \Pprobp if $A$ is surjective and \Pprobp is strictly feasible.
\end{definition}

The following result describes directional derivatives of the value function.

\begin{theorem}[Directional derivative of the value function]
Suppose the primal problem \Pprobp is feasible and its optimal value is finite. Let
$\rm{Sol}\Dprob$ be the set of optimal solutions of the dual \Dprob.
Then $\rm{Sol}\Dprob$ is nonempty and bounded if, and only if, \MFCQ holds.
Moreover, under \MFCQp, the directional derivative $v'(0;w)$ of the value function at $\Delta=0$ in direction $w$ admits the representation
$$v'(0;w)=\sup \{\langle w,y\rangle: y\in \rm{Sol}\Dprob\}.$$
\end{theorem}

In particular, in the notation of the above theorem, the local Lipschitz constant of $v$ at the origin,
$$ \limsup_{\Delta_1,\Delta_2 \to 0}~ \frac{|v(\Delta_1)-v(\Delta_2)|}{\|\Delta_1-\Delta_2\|},$$
coincides with the norm of the maximal-norm dual optimal solution, and
is finite if, and only if, \MFCQ holds. Is there then an upper-bound on
the latter that we can easily write down? Clearly, such a quantity must
measure the strength of \MFCQp, and is therefore intimately related to the auxiliary system \eqref{eq:auxsysprim}. To this end, let us define the {\em condition number}
$$\con \Pprob:=\min_{y: \|y\|=1} \max\{\dist_{\K^*}(A^*y), \langle b,y\rangle \}.$$ 
This number is a quantitative measure of MFCQ, and will appear in latter sections as well. Some thought shows that it is in essence measuring how close the auxiliary system \eqref{eq:auxsysprim} is to being consistent.

\begin{lemma}[Condition number and \MFCQp]
The condition number \con \Pprob\, is nonzero if, and only if, \MFCQ holds.	
\end{lemma}
\begin{proof}
	Suppose  \con \Pprob\, is nonzero. Then clearly $A$ is surjective, since otherwise we could find a unit vector $y$ with $A^*y=0$ and $\langle b,y\rangle \leq 0$. Moreover, the auxiliary system 
	\eqref{eq:auxsysprim} is clearly inconsistent, and therefore \Pprob\, is strictly feasible. 
	
	Conversely, suppose \MFCQ holds. Assume for the sake of contradiction $\con \Pprob\,=0$. Then there exists a unit vector $y$ satisfying $0\neq A^*y\in \KK^*$ and $\langle b,y\rangle\leq 0$. Thus \eqref{eq:auxsysprim} is consistent, a contradiction.  
%
%
\end{proof}

\begin{theorem}[Boundedness of the dual solution set]
Suppose the problem \Pprob\, is feasible with a finite optimal value $\val$. If  the condition number $\con \Pprob$ is nonzero, then the inequality 
$$\|y\|\leq \frac{\max\{\|c\|,-\val\}}{\con \Pprob}$$ holds for all dual optimal solutions $y$. 
\end{theorem}
\begin{proof}
	Consider an optimal solution $y$ of the dual \Dprob.	The inclusion $c-A^*y\in \KK^*$ implies $\dist_{\KK^*}(-A^*y/\|y\|)\leq \|c\|/\|y\|$. Moreover, we have $\langle b,\frac{-y}{\|y\|}\rangle=\frac{-\val}{\|y\|}$. We deduce   $\con \Pprob\leq\frac{ \max\{\|c\|,-\val\}}{\|y\|}$ and the result follows.
\end{proof}

Thus the Lipschitz constant of the value function depends on the extent
to which \MFCQ holds through the condition number. What about stability of the solution set itself? The following theorem, whose proof we omit, answers this question.

\begin{theorem}[Stability of the solution set]
	Suppose \Pprob\, satisfies \MFCQp. Let $\mathcal{F}_p(g)$ be the solution set of the perturbed system 
	$$Ax= b+g,\quad x\succeq_{\mathcal{K}} 0.$$
	Fix a putative solution $\bar x\in \mathcal{F}_p$. Then there exist constants $c>0$ and $\epsilon>0$ so that the inequality
	$$\dist(x;\mathcal{F}_p(g))\leq c\cdot\|(Ax-b)-g\|$$
	holds for any $x\in \KK\cap B_{\epsilon}(\bar x)$ and $g\in B_{\epsilon}(0)$. 
	The infimal value of $c$ over all choices of $\epsilon>0$ so that the above inequalities hold is exactly
	\begin{equation}\label{eqn:MFCQ_meas_reg}
	\sup_{y:\|y\|= 1} \frac{1}{\dist(-A^*y;\face(\bar x,\K)^{\triangle})}.
	\end{equation}
\end{theorem}

\bigskip
In particular, under \MFCQp, we can be sure that for any point $\bar x\in \mathcal{F}_p$, there exist $c$ and $\epsilon >0$ satisfying 
$$\dist(\bar x;\mathcal{F}_p(g))\leq c\cdot \|g\|\qquad \textrm{ for all }g\in B_{\epsilon}(0).$$
In other words, the distance $\dist(\bar x;\mathcal{F}_p(g))$, which measures the how far $\mathcal{F}_p(g)$ has moved relative to $\bar x$, is bounded by a multiple of the perturbation parameter $\|g\|$.
The proportionality constant $c$ is fully governed by the strength of \MFCQp, as measured by the quantity \eqref{eqn:MFCQ_meas_reg}. 

\section{Distance to infeasibility}
In numerical analysis, the notion of stability is closely related to the ``distance to infeasibility'' -- the smallest perturbation needed to make the problem infeasible. A simple example is the problem of solving an equation $Lx=b$ for an invertible matrix $L\colon\R^n\to\R^n$. Then the Eckart-Young theorem shows equality
$$\min_{G\in \R^{n\times n}} \{\|G\|: L+G \textrm{ is singular}\}=\frac{1}{\|L^{-1}\|}.$$
Here $\|G\|$ denotes the operator norm of $G$.
The left-hand-side measures the smallest perturbation $G$ needed to make the system $(L+G)x=b$ singular, while the right-hand-side measures the Lipschitz dependence of the solution to the linear system $Lx=b$ relative to perturbations in $b$, and yet the two quantities are equal.
An entirely analogous situation holds in conic optimization, with \MFCQ playing the role of invertibility.

\begin{definition}[Distance to infeasibility]
{\rm The {\em distance to infeasibility} of \Pprobp is the infimum of the quantity $\max\{\|G\|,\|g\|\}$ over linear mappings $G$ and vectors $g$ such that the system 
$$(A+G)x= b+g,\quad x\succeq_{\mathcal{K}} 0\qquad\qquad \textrm{ is infeasible}.$$
}
\end{definition}

This quantity does not change if instead of the loss of feasibility, we
consider the loss of strict feasibility. The following fundamental
result equates the condition number (measuring the strength of \MFCQp) and the distance to infeasibility.

\begin{theorem}[Strict feasibility and distance to infeasibility]\label{thm:radius}
The following exact equation is always true:
$$\textrm{distance to infeasibility}=\cond \Pprob.$$
\end{theorem}

\section{Commentary}
\label{sect:virtuescommentary}
The classical theorem of the alternative is Farkas Lemma that appears in
proofs of duality in linear programming, as well as in more general nonlinear
programming, after linearizations. This and more general theorems of 
the alternative are given in e.g.,~ the 1969 book by Mangasarian
\cite{Mang:69} and in the 1969 survey paper by Ben-Israel~\cite{MR0249108}.
The specific theorems of the alternative that we use are similar
to the one used in the \FR development in~\cite{bw1,bw2,bw3}.

The Mangasarian-Fromovitz \CQ was introduced in
\cite{MR34:7263}. This condition and its equivalence to stability with respect
to perturbations in the data and compactness of the multiplier set
 has been the center of extensive research,
e.g.,~\cite{MR0441352}. The analogous conditions for general nonlinear convex constraint systems is the \textdef{Robinson regularity condition},
e.g.,\cite{rob75a,rob76a}. The notion of distance to infeasibility and relations to condition numbers was initiated by Renegar e.g.,~\cite{Ren:93,Ren:94,MR96c:90048,MR1757553}. 
The relation with Slater's condition is clear. Theorem~\ref{thm:radius}, as stated, appears in~\cite{radius}, though in essence it is present in Renegar's work  \cite{Ren:93,Ren:94}.

\chapter{Facial reduction}
\label{chap:FR}
Theorems~\ref{thm:primalthmalt} and \ref{thm:dualthmalt} have already set
the stage for the ``Facial Reduction'' procedure, used for regularizing
degenerate conic optimization problems by restricting the problem to
smaller and smaller dimensional faces of the cone $\KK$. In this
section, we formalize this viewpoint, empathizing semi-definite programming. Before we proceed with a detailed description, it is instructive to look at the simplest example of Linear Programming. In this case, a single iteration of the facial reduction procedure corresponds to finding redundant variables (in the primal) and implicit equality constraints (in the dual).

\section{Preprocessing in linear programming}
Improvements in the solution methods for large-scale linear programming problems have
been dramatic since the late 1980's. A technique that has become
essential in commercial software is a
\textdef{preprocessing step} for the linear program before sending it to the
solver. The preprocessing has many essential features. For example, it
removes redundant variables (in the primal) and implicit equality constraints 
(in the dual) thus potentially dramatically
reducing the size of the problem while simultaneously improving the
stability of the model. 
These steps in linear programming are examples of the Facial Reduction procedure, which we will formalize shortly.


\begin{example}[primal facial reduction]
	Consider the problem
	\[
	\begin{array}{cl}
	\min &  
	\begin{pmatrix} 2 & 6 & -1 & -2 & 7 \end{pmatrix} x  \\
	\text{s.t.} & 
	\begin{bmatrix} 
	1 & 1 & 1 & 1 & 0  \cr 
	1 & -1 & -1 & 0 & 1  \cr 
	\end{bmatrix} x = 
	\begin{pmatrix} 1\cr -1 \end{pmatrix} \\
	&   x \geq 0.
	\end{array}
	\]
	If we sum the two constraints we see 
	\[
	x\geq 0 \quad \textrm{and}\quad 2x_1+x_4 + x_5 = 0\quad \implies\quad  x_1=x_4=x_5=0.
	\]
	Thus the coordinates $x_1$, $x_4$, and $x_5$ are identically zero on the entire feasible set. In other words, the feasible region is contained in the proper face $\{x\geq 0: x_1=x_4=x_5=0\}$ of the cone $\R^n_+$. The zero coordinates can easily be eliminated and the corresponding columns discarded, yielding the equivalent simplified problem in the smaller face:
	\[
	\begin{array}{cl}
	\min &  
	\begin{pmatrix}  6 & -1  \end{pmatrix} 
	\begin{pmatrix}  x_2 \cr x_3  \end{pmatrix} \\
	\text{s.t.} & 
	\begin{bmatrix} 
	1 & 1 \\
	-1 &-1
	\end{bmatrix} 
	\begin{pmatrix}  x_2 \cr x_3  \end{pmatrix} = \begin{pmatrix} 1 \\ -1\end{pmatrix}  \\
	& x_2 , x_3   \geq 0,\quad x_1 = x_4= x_5=0.
	\end{array}
	\]
The second equality can now also be discarded as it is is equivalent to the first.  
\end{example}

How can such implicit zero coordinates be discovered systematically? Not surprisingly, the auxiliary system \eqref{eq:auxsysprim_feas} provides the answer: 
$$0\neq A^Ty \geq 0 ,\qquad  b^Ty=0.$$ Suppose $y$ is feasible for this auxiliary system. Then for any $x$ feasible for the problem, we deduce
$0\leq  \sum_i x_i(A^Ty)_i=x^T(A^Ty)=b^Ty=0$.
Thus all the coordinates 
$x_i$, for which the strict inequality $(A^Ty)_i>0$ holds, must be zero.

\begin{example}[dual facial reduction]
	A similar procedure applies to the dual. Consider the problem
	\[
	\begin{array}{cl}
	\max &  
	\begin{pmatrix} 2 & 6 \end{pmatrix} y  \\
	\text{s.t.} & 
	\begin{bmatrix} 
	-1 & -1   \cr 
	1 & 1   \cr 
	1 & -1   \cr 
	-2 & 2   \cr 
	\end{bmatrix} y \leq 
	\begin{pmatrix} 1\cr 2 \cr 1 \cr -2  \end{pmatrix}.
	\end{array}
	\]
	Twice the third row plus the fourth row sums to zero. We conclude that
	the last two constraints are implicit equality constraints.
	Thus after 
	substituting $\begin{pmatrix}y_1 \cr y_2 \end{pmatrix} = 
	\begin{pmatrix}1 \cr 0  \end{pmatrix} +
	t \begin{pmatrix}1 \cr 1  \end{pmatrix}$, we obtain a simple univariate problem.
Again, this discovery of implicit equality constraints can be done systematically by considering the auxiliary system~\eqref{eq:auxsysprim_feas_dual}:
$$0\neq x\geq 0,\quad Ax=0,\quad \textrm{and}\quad c^\top x=0.$$ Suppose we find such a vector $x$. Then for any feasible vector $y$ we deduce
$0\leq \sum_{i} x_i(c-A^Ty)_i= x^T(c-A^Ty)=0$.
Thus for each positive component $x_i>0$, the corresponding inequality 
$(c-A^Ty)_i\geq 0$ is fulfilled with equality along the entire feasible region.
\end{example}

\section{Facial reduction in conic optimization}
Keeping in mind the example of Linear Programming, we now formally
describe the Facial Reduction procedure. To do this, consider the primal problem \Pprobp failing Slater's condition. Our goal is to find an equivalent  problem that does satisfy Slater's condition. To this end, suppose that we had a description of
 $\face(\mathcal{F}_p,\KK)$ -- the minimal face of $\KK$ containing the feasible region $\mathcal{F}_p$. 
Then we could replace $\KK$ with $\KK':=\face(\mathcal{F}_p,\KK)$, $\E$ with $\E':=\spann  \KK'$, and $A$ with its restriction to $\E'$. The resulting smaller dimensional primal problem would automatically satisfy Slater's condition, since $\mathcal{F}_p$ intersects the relative interior of $\KK'$. 

The Facial Reduction procedure is a conceptual method that at termination discovers  $\face(\mathcal{F}_p,\KK)$. Suppose that $\K$ has nonempty interior. In the first iteration, the scheme
determines any vector $y$ satisfying the auxiliary system \eqref{eq:auxsysprim_feas}. If no such vector exists, Slater's condition holds and the method terminates. Else, Theorem~\ref{thm:primalthmalt} guarantees that $\mathcal{F}_p$ lies in proper face $\K':=(A^*y)^{\perp}\cap \K$. Treating $\K'$ as a subset of the ambient Euclidean space $\E':=\spann\K'$ yields the smaller dimensional reformulation of $\Pprobp$:
\begin{equation}
\begin{array}{rcc}
\displaystyle \min_{x\in \E'} &\langle c,x\rangle \\
\textrm{s.t.} & Ax=b  \\
& x\in \KK'. 
\end{array}
\end{equation}
We can now repeat the process on this problem. Since the dimension of the problem decreases with each facial reduction iteration, the procedure will terminate after at most $\dim \E$ steps. 

\begin{definition}[Singularity degree]
The {\em singularity degree} of \Pprob, denoted $\sing\Pprob$, 
\index{singularity degree of \Pprob, $\sing\Pprob$} 
\index{$\sing\Pprob$, singularity degree of \Pprob}
is the minimal number of iterations that are necessary for the Facial Reduction to terminate, over all possible choices of certificates generated by the auxiliary systems in each iteration.
\end{definition}
	
The singularity degree of linear programs is at most one, as we will see
shortly. More generally, such as for semi-definite programming problems, the singularity degree can be much higher.

The Facial Reduction procedure applies to the dual problem \Dprobp by using the equivalent primal form \eqref{eqn:dual_primal} and using Theorem~\ref{thm:dualthmalt}. We leave the details to the reader.

\section{Facial reduction in semi-definite programming}
Before discussing further properties of the Facial Reduction algorithm
in conic optimization, let us illustrate the procedure in semi-definite programming. To this end, consider the primal problem \Pprobp with $\K={\bf S}^n_+$. Suppose that we have available a vector $y$ feasible for the auxiliary system \eqref{eq:auxsysprim_feas}. Form now an eigenvalue decomposition 
$$A^*y=\begin{bmatrix}
U & V
\end{bmatrix}
\begin{bmatrix}
\Lambda & 0\\
0 & 0
\end{bmatrix}
\begin{bmatrix}
U & V
\end{bmatrix}^T,
$$
where $\begin{bmatrix}
	U & V
\end{bmatrix}\in \R^{n\times n}$ is an orthogonal matrix and $\Lambda\in {\bf S}^r_{++}$ is a diagonal matrix. Then as we have seen in Example~\ref{exa:face_snp}, the matrix $A^*y$ exposes the face $V{\bf S}^{n-r}V^T$ of ${\bf S}^n_{+}$.
Consequently, defining the linear map $\widetilde A(Z)=A(VZV^T)$, the
primal problem \Pprobp is equivalent to the smaller dimensional \SDP
\begin{equation}
\begin{array}{rcc}
\displaystyle \min_{Z} &\langle V^TCV,Z\rangle \\
\textrm{s.t.} & \widehat{A}(Z)=b  \\
& Z\in {\bf S}^{n-r}_+.
\end{array}
\end{equation}
Thus one step of facial reduction is complete. Similarly let us look at the dual problem \Dprob. Suppose that $X$ is feasible for the auxiliary system 
\eqref{eq:auxsysprim_feas_dual}. Let us form an eigenvalue decomposition 
$$X=\begin{bmatrix}
U & V
\end{bmatrix}
\begin{bmatrix}
\Lambda & 0\\
0 & 0
\end{bmatrix}
\begin{bmatrix}
U & V
\end{bmatrix}^T,
$$
where $\begin{bmatrix}
U & V
\end{bmatrix}\in \R^{n\times n}$ is an orthogonal matrix and $\Lambda\in {\bf S}^r_{++}$ is a diagonal matrix. The face exposed by $X$, namely 
$V{\bf S}^{n-r}V^T$, contains all the feasible slacks $\{C-A^*y: y\in \mathcal{F}_d\}$ by Theorem~\ref{thm:dualthmalt}. Thus defining the linear map $L(y):=V^T(A^*y)V$, the dual \Dprobp is equivalent to the problem
\begin{equation}
\begin{array}{rcc}
\displaystyle \max_{y} &\langle y,b\rangle \\
\textrm{s.t.} & L(y)\preceq_{{\bf S}^{n-r}} V^TCV.
\end{array}
\end{equation}
Thus one step of Facial Reduction is complete and the process can continue.

To drive the point home, the following simple example shows that for \SDP the
singularity degree can indeed be strictly larger than one.
\begin{example}[Singularity degree larger than one]\label{exa:basic_fail_sing}
	Consider the primal \SDP feasible region
	$$\mathcal{F}_p=\{X\in {\bf S}^3_+: X_{11}=1, X_{12}+X_{33}=0, X_{22}=0\}.$$
	Notice the equality $X_{22}=0$ forces the second row and column of $X$ to be zero, i.e. they are redundant. Let us see how this will be discovered by Facial Reduction.
	
The linear map $A\colon{\bf S^3}\to\R^3$ has the form 
$$AX=(\tr(A_1X),\tr(A_2X),\tr(A_3X))^T$$ 
for the matrices	
$$A_1=\begin{bmatrix}
1 & 0 & 0\\
0 & 0& 0\\
0 & 0& 0
\end{bmatrix},\quad 
A_2=\begin{bmatrix}
0 & 1/2 & 0\\
1/2 & 0& 0\\
0 & 0& 1
\end{bmatrix},\quad \textrm{and}\quad 
A_3=\begin{bmatrix}
0 & 0 & 0\\
0 & 1& 0\\
0 & 0& 0
\end{bmatrix}.
$$
Notice that $\mathcal{F}_p$ is nonempty since it contains the rank $1$ matrix $e_1e_1^T$. The auxiliary system \eqref{eq:auxsysprim_feas} then reads
$$0\neq 
\begin{bmatrix}
v_1 & v_2/2 & 0\\
v_2/2 & v_3& 0\\
0 & 0& v_2
\end{bmatrix}\succeq 0  \quad \textrm{and}\quad v_1=0.$$
Looking at the second principal minor, we see $v_3=0$. Thus all feasible
$v$ are positive multiples of the vector $e_3=\begin{pmatrix}
0&0&1\end{pmatrix}^T$. One step of Facial Reduction using the exposing vector 
$A^*e_3=e_2e_2^T$ yields the equivalent reduced region
$$\left\{\begin{bmatrix} X_{11} & X_{13}\\
X_{13} & X_{33}
\end{bmatrix}\succeq 0: X_{11}=1, X_{33}=0\right\}.$$
This reduced problem clearly fails Slater's condition, and Facial Reduction can continue. Thus the singularity degree of this problem is exactly two.
\end{example}

The pathological Example~\ref{exa:basic_fail_sing} can be generalized to higher dimensional space with $n=m$, by nesting, leading to problems with singularity degree $n-1$; the construction is explained in Tun{\c{c}}el~\cite[page 43]{MR2724357}.

%
%
%

\section{What facial reduction actually does}
There is a direct and enlightening connection between Facial Reduction
and the geometry of the image set $A\KK$. To elucidate this relationship, we first note the following equivalent characterization of Slater's condition.

\begin{proposition}[Range space characterization of Slater]\label{prop:range}
The primal problem \Pprobp is strictly feasible if, and only if, the vector $b$ lies in the relative interior of  $A(\K)$.
\end{proposition}

The following is the central result of this section.
\begin{theorem}[Fundamental description of the minimal face]
\label{thm:face_red} {\hfill \\ }
Assume that the primal \Pprobp is feasible.
Then a vector $v$ exposes a proper face of
$A(\K)$ containing $b$ if, and only if, $y$ satisfies the auxiliary system \eqref{eq:auxsysprim_feas}.
Defining for notational convenience $N:=\face(b, A(\K))$, the following are true.
\\ (I) $\quad$ 
	We always have:
\[
		\face(\mathcal{F}_p, \K)=\K \cap A^{-1}N.
\]
(II) $\quad$ For any vector $y\in \F$ the following equivalence holds:
\[
	y \textrm{ exposes } N \quad\Longleftrightarrow\quad 
	A^*y \textrm{ exposes } \face(\mathcal{F}_p, \K).
\]
In particular, the inequality $\sing\Pprobp\leq 1$ holds if, and only
if, $\face(b, A(\K))$ is an exposed face of $A(\K)$.
\end{theorem}

Some commentary is in order. First, as noted in Proposition
\ref{prop:range}, the primal \Pprobp is strictly feasible if, and only
if, the right-hand-side $b$ lies in $\ri A(\KK)$. Thus when strict feasibility fails, the set $N=\face (b, A(\K))$
is a proper face of the image $A(\K)$. The theorem above yields the
exact description $\face(\mathcal{F}_p, \K)=\K \cap A^{-1}N$ of the
object we are after. On the other hand, determining a facial description of $A(\K)$ is a difficult proposition. Indeed, even when $\K$ is a simple cone, the image $A(\K)$ can be a highly nontrivial object. For instance, the image $A({\bf S}^n_+)$ may fail to be facially exposed or even closed; examples are forthcoming. 

Seeking to obtain a description of $N$, one can instead try to find ``certificates'' $y$ exposing a proper face of $A(\K)$ containing $b$. Such vectors $y$ are precisely those satisfying the auxiliary system \eqref{eq:auxsysprim_feas}. In particular, 
part II of Theorem~\ref{thm:face_red} yields a direct obstruction to
having low singularity degree: $\sing\Pprobp\leq 1$ if, and only if, $\face(b, A(\K))$ is an exposed face of $A(\K)$.
Thus the lack of facial exposedness of the image $A(\mathcal{K})$ can become an obstruction. For the cone $\K=\R^n_+$, the image $A(\mathcal{K})$ is polyhedral and is therefore facially exposed. On the other hand, linear images of the cone $\K={\bf S}^n_+$ can easily fail to be facially exposed. This in essence is the reason why preprocessing for general conic optimization is much more difficult than its linear programming counterpart (having singularity degree at most one).

The following two examples illustrate the possibly complex geometry of image sets $A({\bf S}^n)$.

\begin{example}[Linear image not closed]
Define the linear map $A\colon{\bf S}^2\to\R^2$ by
\[
	A(X):= \begin{pmatrix} 2X_{12}\cr X_{22}
       \end{pmatrix}
\]
Then the image $A({\bf S}_+^2)$ is not closed, since
\[
A\left(\begin{bmatrix} k & 1 \cr 1 & \frac 1k \end{bmatrix}\right) \rightarrow
\begin{pmatrix} 2\cr 0 \end{pmatrix} \notin A\left({\bf S}^2_+\right).
\]
More broadly, it is easy to see the equality
\[
	A\left({\bf S}^2_+\right) = \{(0,0)\} \cup (\R\times \R_{++}).
\]
\end{example}

\begin{example}[Linear image that is not facially exposed]\label{exa:lin_im_not_exp}
Consider the feasible region in Example~\ref{exa:basic_fail_sing}. There we showed that the singularity degree is equal to two. Consequently, by Theorem~\ref{thm:face_red} we know that the minimal face of $A(\K)$ containing 
	$b=\begin{pmatrix} 1&0&0\end{pmatrix}^T$ must be nonexposed. 

Let us verify this directly. To this end, we can without loss of generality treat $A$ as mapping into ${\bf S}^2$ via $$A(X)=\begin{bmatrix}
X_{11}&X_{12}+X_{33}\\
X_{12}+X_{33}& X_{22}
\end{bmatrix}
$$ 
and identify $b$ with $e_1e_1^T$.
Then the image $A({\bf S}^3_+)$ is simply the sum, 
$$A({\bf S}^3_+)={\bf S}^{2}_+ +\cone\begin{bmatrix}
0&1\\
1&0\\
\end{bmatrix}.$$
 See Figure~\ref{fig:rotatedcone}.
 \begin{figure}[!ht]
 	\centering
 	\includegraphics[scale=0.5]{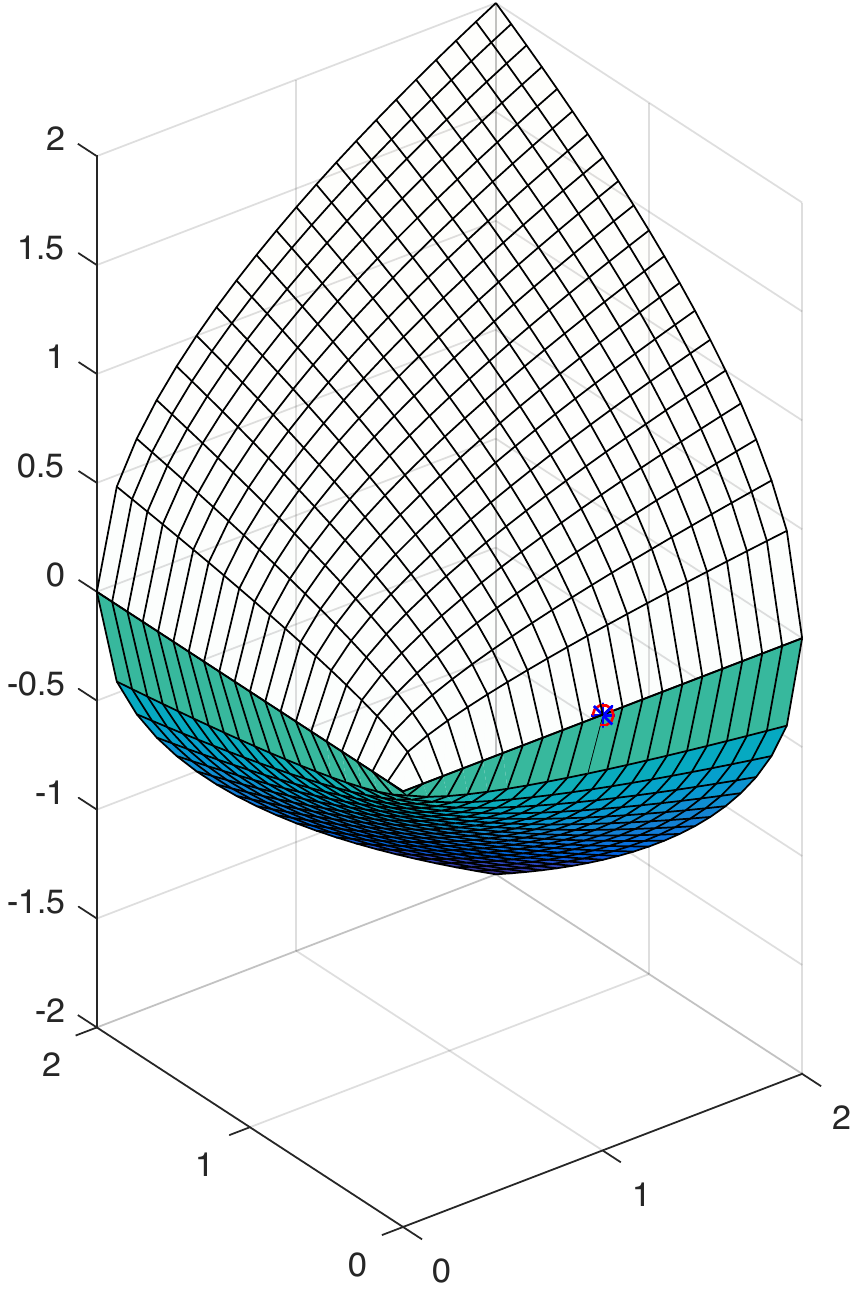}
 	\caption{Nonexposed face of the image set}
 	\label{fig:rotatedcone}
 \end{figure}
 
  Consider the set $$G:=\cone \begin{bmatrix}
 1 & 0\\
 0 & 0
 \end{bmatrix}.$$ We claim that $G$ is a face of $A({\bf S}^3_+)$ and is therefore the minimal face containing $e_1e_1^T$. Indeed, suppose we may write 
$$\begin{bmatrix}
1 & 0\\
0 & 0
\end{bmatrix}
=\begin{bmatrix}
X_{11}&X_{12}+X_{33}\\
X_{12}+X_{33}& X_{22}
\end{bmatrix}+\begin{bmatrix}
X_{11}'&X_{12}'+X_{33}'\\
X_{12}'+X_{33}'& X_{22}'
\end{bmatrix},$$
for some matrices $X,X'\in {\bf S}^3_+$.
Comparing the $2,2$-entries, we deduce $X_{22}=X_{22}'=0$ and consequently $X_{12}=X_{12}'=0$.
Comparing the 1,2-entries yields $X_{33}=X_{33}'=0$. Thus both summands lie in $G$; therefore $G$ is a face of the image $A({\bf S}^3_+)$. Next, using Lemma~\ref{lem:dual_sum}, observe 
$$A({\bf S}^3_+)^*=\left({\bf S}^{2}_++\cone\begin{bmatrix}
0&1\\
1&0\\
\end{bmatrix}\right)^*=\{Y\in {\bf S}^2_+: Y_{12}\geq 0\}.$$
Consequently, any matrix exposing $G$ must lie in the set
$$A({\bf S}^3)^*\cap G^{\perp}=\cone (e_2e_2^T).$$
On the other hand, the set $$(e_2e_2^T)^{\perp}\cap A({\bf S}^3_+)=\cone \begin{bmatrix}
1 & 0\\
0 & 0
\end{bmatrix}+\cone\begin{bmatrix}
0 & 1\\
1 & 0
\end{bmatrix}$$
is clearly strictly larger than $G$. Hence $G$ is not an exposed face.

\end{example}

\section{Singularity degree and the \textdef{H\"older error bound} in \SDPp}
For semi-definite programming, the singularity degree plays an especially 
important role, controlling the \emph{H\"olderian stability of the feasible 
region}. Consider two sets $Q_1$ and $Q_2$ in $\E$. A convenient way to understand the regularity of the intersection $Q_1\cap Q_2$ is to determine the extent to which the computable residuals, $\dist(x,Q_1)$ and $\dist(x,Q_2)$, bound the 
\emph{error} $\dist(x,Q_1\cap Q_2)$. Relationships of this sort are commonly 
called {\em error bounds} of the intersection and play an important role for 
convergence and stability of algorithms. Of particular importance are 
H\"olderian error bounds -- those asserting the inequalities
$$
\dist(x,Q_1\cap Q_2)\leq \mathcal{O}
\left(\dist^{q}(x,Q_1)+\dist^{q}(x,Q_2)\right)
$$
on compact sets, for some powers $q \geq 0$. For semi-definite programming, the singularity degree precisely dictates the H\"older exponent $q$.

\begin{theorem}[{H\"olderian
	error bounds from the singularity degree}]\label{thm:Holder}
Consider a feasible primal \SDP problem \Pprobp and define the affine space 
$$\mathcal{V}:=\{X:A(X)=b\}.$$ Set $d:=\sing\Pprob$.
Then for any compact set $U$, there is a real $c> 0$ so that  
$$\dist_{\mathcal{F}_p}(X)\leq
 c\left(\dist^{2^{-d}}(X,\Snp)+\dist^{2^{-d}}(X,\mathcal{V})\right),
\quad \textrm{ for all }x\in U.$$	
\end{theorem}

What is remarkable about this result is that neither the dimension of the matrices $n$, the number of inequalities $m$, nor the rank of the matrices in the region $\mathcal{F}_p$ determines the error bound. Instead, it is only the single quantity, the singularity degree, that drives this regularity concept.

\begin{example}[Worst-case example]\label{exa:worst}
Consider the \SDP feasible region
$$\mathcal{F}_p:=\{X\in {\bf S}^n_+: X_{22}=0\textrm{ and }X_{k+1,k+1}=X_{1,k} \textrm{ for } k=2,3,\ldots,n-1\}.$$
For any feasible $X$, the constraint $X_{22}=0$ forces $0=X_{12}=X_{33}$. By an inductive argument, then we deduce $X_{1k}=0$ and $X_{k,k}=0$ for all $k=2,\ldots,n$. Thus the feasible region $\mathcal{F}_p$ coincides with the ray $\cone (e_1e_1^T)$.

 Given $\epsilon>0$, define the matrix
$$X(\epsilon)=\begin{bmatrix}
n & \epsilon^{2^{-1}} & \epsilon^{2^{-2}} &\hdots& \epsilon^{2^{-(n-1)}}\\
\epsilon^{2^{-1}} & \epsilon & 0 &\hdots& 0\\
\epsilon^{2^{-2}} & 0 &\epsilon^{2^{-1}} & & \\
\vdots & \vdots & & \ddots & \\
\epsilon^{2^{n-1}} & 0 & & & \epsilon^{2^{-(n-2)}}
\end{bmatrix}$$
Observe that $X(\epsilon)$ violates the linear constraints only in the 
requirement $X_{22}=0$. Consequently, the distance of $X(\epsilon)$ to the linear space $\mathcal{V}:=\{X\in {\bf S}^n: A(X)=b\}$
is on the order of $\epsilon$. On the other hand, the distance of $X(\epsilon)$ to the solution set is at least on the order of $\epsilon^{2^{-(n-1)}}$. This example shows that the H\"older exponent in this case is at least $2^{-(n-1)}$. Combined with Theorem~\ref {thm:Holder} and the fact that the feasible region contains rank one matrices, we deduce $\sing\Pprob= n-1$ and the H\"older exponent guaranteed by the theorem is sharp.
\end{example}

\section{Towards computation}
The Facial Reduction procedure is conceptual. To implement it, since the error compounds along the iterations, one must be able to either solve the auxiliary systems \eqref{eq:auxsysprim_feas} (resp. \eqref{eq:auxsysprim_feas_dual}) to machine precision in each iteration or certify that the systems are inconsistent. On the other hand, in general, there is no reason to believe that solving a single auxiliary system is any easier than solving the original problem \Pprob. 

One computational approach for facial reduction in \SDPp, explored by Permenter-Parrilo~\cite{perm}, is to {\em relax} the auxiliary problems to ones that are solvable. Instead of considering  \eqref{eq:auxsysprim_feas}, one can choose a convex cone $\widehat\K\subseteq {\bf S}^n_+$ so that
consistency of the system can be checked:
\begin{equation*}
0\neq A^*y \in \widehat \K \qquad\textrm{and}\qquad \langle b,y\rangle= 0.
\end{equation*}
If a vector $y$ satisfying the system is found, then one can perform one
step of facial reduction. If not, the scheme quits, possibly without
having successfully deduced that Slater's condition holds. Simple examples 
of $\widehat\K$ are the sets
$\Diag(\R^n_+)$ and the cone dual to 
$\{X \in \Sn: \textrm{every }2\times 2 \textrm{ minor of } X  
\textrm{is \PSDp}\}$, where \PSD denote positive semi-definite. 
\index{\PSDp, positive semi-definite} 
\index{positive semi-definite, \PSDp} 
The above feasibility problem is then an instance of linear programming in the first case and of second-order  cone programming in the second. More details are provided in~\cite{perm}. Readers may be skeptical of this strategy since this technique will work only for special types of degeneracy. For example, the first relaxation $\Diag(\R^n_+)$ can only detect that some diagonal elements of $X$ are identically zero on the feasible region $\mathcal{F}_p$. On the other hand, it does appear that degeneracy typically arising in applications {\em is} highly structured, and promising numerical results has been reported in~\cite{perm}.

There exist other influential techniques for regularizing conic
optimization problems that are different from the facial reduction
procedure. Two notable examples are Ramana's extended dual~\cite{Ram:93}
and the homogeneous self-dual embedding 
e.g.,~\cite{int:deklerk7,permfribergandersen}. The latter, in particular, 
is used by MOSEK~\cite{MR1748773} and SeDuMi~\cite{MR1778433}. A dual
approach, called the \textdef{conic expansion approach} is discussed 
at length in~\cite{MR3063940}, see 
also~\cite{GPatakiLiu:15,P00,LuoStZh:97,Sturmphd97}

We do not discuss these techniques here. Instead, we focus on the most promising class of conic optimization problems -- those having singularity degree at most one.  In the rest of the manuscript, we provide a series of influential examples, where the structure of the problem enables one to obtain  feasible points of the auxiliary systems without having to invoke any solvers. Numerical illustrations show that the resulting reduced subproblems are often much smaller and more stable than the original.

\section{Commentary}
Preprocessing is essential in making \LP algorithms efficient. A main
ingredient is identifying primal and/or dual slack variables that are
identically zero on the feasible set. This is equivalent to facial
reduction that reduces the problem to faces of the nonnegative orthant,
e.g.,~\cite{Huang04preprocessingand}
The facial reduction procedure for
general conic optimization started in~\cite{bw1,bw2,bw3}. The procedure provides a primal-dual pair of conic optimization problems that are \emph{proper} in the
sense that the dual of the dual yields the primal. Example~\ref{exa:basic_fail_sing} and extensions can be found in Tun{\c{c}}el~\cite[page 43]{MR2724357}. The notation of singularity degree and its connection to error bounds (Theorem~\ref{thm:Holder}) was discovered by Sturm \cite[Sect. 4]{S98lmi}; Example~\ref{exa:worst} appears in this paper as well.
Example~\ref{exa:lin_im_not_exp} is motivated by Example 1 in \cite{Pnice}. Theorem~\ref{thm:face_red}  appeared in \cite{DrPaWo:14}.
As mentioned previously, there are many approaches to
``regularization'' of conic optimization problem, aside from facial reduction, including the self-dual embedding and the approximation approaches in
\cite{perm,permfribergandersen,2016arXiv160802090P}.
An alternate view of obtaining a dual without 
a constraint qualification was given in~\cite{Ram:93,Ram:95}, though  
a relationship to facial reduction was later explained in~\cite{RaTuWo:95}.

\part{Applications and illustrations}
\label{part:applic}

In this chapter, we discuss a number of diverse and important computational 
problems. This includes various matrix completion problems and discrete optimization problems such as the quadratic assignment problem, graph partitioning,  and the strengthened relaxation of the maximum cut problem. 
In the final section, we also discuss sum of squares relaxations for polynomial optimization problems. In each case, 
we use the structure of the problem to determine a face of the 
positive semi-definite cone containing the entire feasible region. One exception is  
the matrix completion problem, where we instead determine a face
containing the \textdef{optimal face}, as opposed to the entire feasible region.
Numerical illustrations illustrate the efficacy of the approach.

\chapter{Matrix completions}
\label{chap:matrixcompl}
We begin with various matrix completion problems. Broadly speaking, the goal is to 
complete a partially specified matrix, while taking into account a priori known structural properties such as a given
rank or sparsity. There is a great variety of references on matrix completion problems; see for example \cite{MR0247866,Floudas:2001aa,MR1607310,MR1059486}.

\section{Positive semi-definite matrix completion}
We begin with a classical problem of completing a \PSD matrix from 
partially observed entries. To model this problem,  consider an 
undirected graph $G=(V,E)$ with a vertex set $V=\{1,\ldots,n\}$ and 
\textdef{graph}
edge set $E\subseteq\{ij: 1\leq i\leq j\leq n\}$. 
The symbols $ij$ and $ji$ always refer to the same edge.
 Notice we allow self-loops $ii\in E$. For simplicity, we in fact 
assume that $E$ contains $ii$ for each node $i\in V$. Elements $\omega$ of $\R^{E}$ are called {\em partial matrices}, as they specify entries of a partially observed symmetric $n\times n$ matrix.
Given  a partial matrix $\omega\in \R^{E}$, the \textdef{\PSD completion problem}
asks to determine, if possible, a matrix $X$ in the set 
$$\textdef{$\mathcal{F}_p$}:=\{X\in {\bf S}^n_+: X_{ij}=\omega_{ij}\quad \textrm{for all indices }ij\in E\}.$$
That is, we seek to complete the partial matrix $\omega$ to an $n\times
n$ positive semi-definite matrix. When do such \PSD completions exist,
that is, when is $\mathcal{F}_p$ nonempty? Clearly, a necessary
condition is that $\omega$ is a {\em partial \PSD matrix}, meaning that
its restriction to any specified principal submatrix is \PSDp. This condition, however, is not always sufficient. 

A  graph $G=(V,E)$ is called {\em \PSD completable} if every partial
\PSD matrix $\omega\in \R^E$ is completable to a \PSD matrix. It turns
out that \PSD completable graphs are precisely those that are chordal.
Recall that a graph $G$ is called \textdef{chordal} if any cycle of four or more nodes has a chord -- an edge joining any two nodes that are not adjacent in the cycle.

\begin{theorem}[\PSD completable and chordal graphs]\label{thm:partial_comp}
The graph $G$ is \PSD completable if, and only if, $G$ is
chordal.\footnote{If all the self-loops are not included
in $E$, one needs to add that the two subgraphs with self-loops and
without are disconnected from each other, see e.g.,~\cite{DrPaWo:14}.}
\end{theorem}

Chordal graphs, or equivalently those that are \PSD completable, play a
fundamental role for the \PSD completion problem. For example, on such graphs, the completion problem admits an efficient combinatorial algorithm~\cite{GrJoSaWo:84,Laurent:00,poly_solve_chord}.

Next, we turn to Slater's condition. Consider the completion problem
\begin{equation}\label{eqn:compl_exa}
\begin{bmatrix}
1 & 1 & ? & ? \\
1 & 1 & 1 & ? \\
? & 1 & 1 & -1 \\
? & ? & -1 & 2 \\
\end{bmatrix}.
\end{equation}
The question marks ? denote the unknown entries. The underlying graph on
four nodes is a path and is therefore chordal. The known entries make up
a partial \PSD matrix since the three specified principal minors are
\PSDp.
Thus by Theorem~\ref{thm:partial_comp}, the completion problem is
solvable. Does Slater's condition hold? The answer is no. The first
leading principal minor is singular, and therefore any \PSD completion must be singular. 

By the same logic, any singular specified principal minor of a partial
matrix $\omega\in \R^E$ certifies that strict feasibility fails. Much
more is true, however. We now show how any singular specified principal minor of a partial matrix $\omega\in \R^E$ yields a face of ${\bf S}^n_+$ containing the entire feasible region, allowing one to reduce the dimension of the problem.

\index{coordinate projection, $\mathcal{P}_{E}$}
\index{$\mathcal{P}_{E}$, coordinate projection}
To see how this can be done, let us introduce some notation.
Define the {\em coordinate projection} map
 $\mathcal{P}_{E}\colon{\bf S}^n\to \R^{E}$ by setting
$$\mathcal{P}_{E}(X)=(X_{ij})_{ij\in E}.$$
In this notation, we can write
$$\mathcal{F}_p:=\{X\in {\bf S}^{n}_+: \PP_{E}(X)=\omega\}.$$
We will now see how the geometry of the image set $\PP_{E}({\bf S}^n_+)$, along with Theorem~\ref{thm:face_red}, helps us discover a face of ${\bf S}^n_+$ containing the feasible region. We note in passing that the image $\PP_{E}({\bf S}^n_+)$ is always closed.\footnote{If some elements $ii$ do not lie in $E$, contrary to our simplifying assumption, then the image $\PP_{E}({\bf S}^n_+)$ can fail to be closed. A precise characterization is given in~\cite{DrPaWo:14}.}

\begin{proposition}[Closure of the image]\label{prop:psd_im_closed}
The image $\PP_{E}({\bf S}^n_+)$ is closed.
\end{proposition}

The reader can check that the adjoint $\mathcal{P}_{E}^*\colon\R^{E}\to {\bf S}^n$ simply pads partial matrices in $\R^E$ with zeros:
\begin{equation}
\left(\PP_E^*(y)\right)_{ij}= \left\{
\begin{array}{ll}
   y_{ij}, & \text{if } ij \in E \text{ or } ji \in E\\
   0, & \text{otherwise}.
\end{array}
  \right.
\end{equation}
  For any subset of vertices $\alpha\subseteq V$, we let 
\index{$E[\alpha]:=\{ij\in E: i,j\in \alpha\}$}
$E[\alpha]:=\{ij\in E: i,j\in \alpha\}$ be the edge set induced by $G$
on $\alpha$ and we set \textdef{$\omega[\alpha]$} to be the restriction of $\omega$ to $E[\alpha]$.
Define the {\em relaxed region}
\begin{equation}\label{eqn:feas_restric}
\textdef{$\mathcal{F}_p(\alpha)$}:=\{X\in {\bf S}^{n}_+: P_{E[\alpha]}(X)=\omega[\alpha]\}.
\end{equation}
Clearly $\alpha \subseteq V$ means we have fewer constraints and
\[
\mathcal{F}_p= \mathcal{F}_{p}(V) \subseteq \mathcal{F}_{p}(\alpha),
\qquad \textrm{ for all }\alpha \subseteq V.
\]

A subset $\alpha\subseteq V$ is called a 
\textdef{clique} if for any two nodes $i,j\in\alpha$ the edge $ij$ lies in $E$. 
Specified principal minors of $\omega$ correspond precisely to cliques
in the graph $G$. We can moreover clearly identify
\textdef{$\R^{E[\alpha]}$} with
the matrix space ${\bf S}^k$. Suppose now that $\omega[\alpha]$ has
rank $r<k$, i.e.,~the principal submatrix of $\omega$ indexed by
$\alpha$ is singular. Then the right-hand-side $\omega[\alpha]$ of
\eqref{eqn:feas_restric} lies in the boundary of the image set
$\PP_{E[\alpha]}({\bf S}^n_+)\subseteq {\bf S}^k_+$. Let $V_{\alpha}$ be an exposing vector of  $\face(\omega[\alpha],{\bf S}^k_+)$. Then by Theorem~\ref{thm:face_red}, we can be sure that $W_{\alpha}:=\PP_{\alpha}^*(V_{\alpha})$ exposes the minimal face of ${\bf S}^n_+$ containing the entire region $\mathcal{F}_{p}(\alpha)$. Given a collection of  cliques $\alpha_1,\alpha_2, \ldots,\alpha_l$, we can perform the same procedure and deduce 
that the entire feasible region  $\mathcal{F}_{p}$ lies in the face 
$$(W_{\alpha_1}^{\perp}\cap {\bf S}^n_+)\cap  (W_{\alpha_2}^{\perp}\cap {\bf S}^n_+)\cap \ldots \cap (W_{\alpha_l}^{\perp}\cap {\bf S}^n_+),$$
which by Proposition~\ref{prop:inter_face_exp} admits the equivalent description
$$(W_{\alpha_1}+W_{\alpha_2}+\ldots+W_{\alpha_l})^{\perp}\cap {\bf S}^n_+.$$
The following example will clarify the strategy.

\begin{example}[Reducing the \PSD completion problem]\label{exa:exam_psd comple}
Let $\FF_p$ consist of all matrices $X\in{\bf S}^4_+$ solving the \PSD completion problem \eqref{eqn:compl_exa}.
There are three nontrivial cliques in the graph, all of size $2$. 
The minimal face of ${\bf S}^2_+$ containing the matrix 
$$\begin{bmatrix}
1 & 1\\
1 & 1\\
\end{bmatrix}
= \begin{bmatrix}
-\frac{1}{2} & \frac{1}{2}\\
\frac{1}{2} & \frac{1}{2}\\
\end{bmatrix} \begin{bmatrix}
0 & 0\\
0 & 4\\
\end{bmatrix} \begin{bmatrix}
-\frac{1}{2} & \frac{1}{2}\\
\frac{1}{2} & \frac{1}{2}\\
\end{bmatrix} 
$$
is exposed by  
$$\begin{bmatrix}
-\frac{1}{2} & \frac{1}{2}\\
\frac{1}{2} & \frac{1}{2}\\
\end{bmatrix} \begin{bmatrix}
4 & 0\\
0 & 0\\
\end{bmatrix} \begin{bmatrix}
-\frac{1}{2} & \frac{1}{2}\\
\frac{1}{2} & \frac{1}{2}\\
\end{bmatrix} 
= \begin{bmatrix}
1 & -1\\
-1 & 1\\
\end{bmatrix}
.$$
Moreover, the matrix $\begin{bmatrix}
	1 & -1\\
	-1 & \,\,\,2\\
\end{bmatrix}$
is definite and hence the minimal face of ${\bf S}^2_+$ containing this matrix is exposed by the all-zero matrix.

The intersection of exposed faces is exposed by the sum of
their exposing vectors. We deduce that $\FF_p$ is contained in the face of ${\bf S}^4_+$ exposed by the sum
$$
\begin{bmatrix}
1 & -1 & 0 &0\\
-1 & 1 & 0 &0\\
0 & 0 & 0 &0\\
0 & 0 & 0 &0\\
\end{bmatrix}
+\begin{bmatrix}
0 & 0 & 0 &0\\
0 & 1 & -1 & 0 \\
0 & -1 & 1 & 0 \\
0 & 0 & 0 &0\\
\end{bmatrix}=\begin{bmatrix}
1 & -1 & 0 & 0 \\
-1 & 2 & -1 & 0 \\
0 & -1 & 1 &0\\
0 & 0 & 0 &0\\
\end{bmatrix}.
$$
After finding the nullspace of this matrix, we deduce 
$$\FF_p\subseteq \begin{bmatrix}
0 & 1  \\
0 & 1  \\
0 & 1  \\
1 & 0  \\
\end{bmatrix} 
{\bf S}^2_+
\begin{bmatrix}
0 & 1 \\
0 & 1  \\
0 & 1  \\
1 & 0  \\
\end{bmatrix}^T.
$$
\end{example}

The following lemma is another nice consequence of the procedure described in the above example.

\begin{lemma}[Completion of banded all ones matrices]\label{lem:allones}
The matrix of all ones is the unique positive semi-definite matrix satisfying $X_{ij} = 1$ for all indices with $|i-j|\leq 1$.
\end{lemma}
\begin{proof}
Consider the edge set $E=\{ij: |i-j|\leq 1\}$ and let $\omega\in \R^{E}$
	be a partial matrix of all ones. Observe $\omega$ has $n$
	specified $2\times 2$-principal submatrices, each having rank 1. By the same logic as in Example~\ref{exa:exam_psd comple}, it follows that the feasible region $\FF_p$ is zero-dimensional, as claimed. 
\end{proof}

The strategy outlined above suggests an algorithm for finding the minimal
face based on exploiting cliques in the graph. This strategy is well-founded at least for chordal graphs.

\begin{theorem}[Finding the minimal face on chordal graphs]
	\label{thm:cliquesuffPSD}
	Suppose that $G$ is chordal and consider a
	partial \PSD matrix $\omega\in \R^{E}$.
	Then the equality $$\face(\FF_{p}, \Snp)=\bigcap_{\alpha\in \Theta} \face(\FF_p({\alpha}),\Snp)\qquad \textrm{ holds},$$
	where $\Theta$ denotes the set of all maximal cliques in  $G$.
\end{theorem}

On the other hand, it is important to realize that when the graph is not
chordal, the minimal face $\face(\FF_p,{\bf S}^n_+)$ is not always
guaranteed to be found from cliques alone. The following example shows a
\PSD completion problem that fails Slater's condition but where all the faces arising from cliques are trivial.

\begin{example}[Slater condition \& nonchordal graphs, \cite{DrPaWo:14}]
\label{ex:n4complbndr}
Let $G=(V,E)$ be  the graph with $V=\{1,2,3,4\}$ and 
$E=\{12,23,34,14\}\cup\{11,22,33,44\}$.
Define the corresponding \PSD completion problems $C(\epsilon)$, 
parametrized by $\epsilon\geq 0$:
\begin{equation*}
C(\epsilon):\qquad\qquad
\begin{bmatrix}
1+\epsilon& 1 & ? & -1 \cr
1 & 1+\epsilon & 1 & ? \cr
? & 1 & 1 +\epsilon& 1 \cr
-1 & ? & 1 & 1 +\epsilon\cr
\end{bmatrix}.
\end{equation*}
Let $\omega(\epsilon)\in\R^E$ denote the corresponding partial matrices.
From Lemma \ref{lem:allones},
the \PSD completion problem $C(0)$ is infeasible, that is 
$\omega(0)$ lies outside of $\PP_{E}({\bf S}^4_+)$. On the other hand, for all sufficiently 
large $\epsilon$, the partial matrices 
$\omega(\epsilon)$ lie in $\Int \PP({\bf S}^4_+)$ by diagonal dominance. 
Since $\PP_E({\bf S}^4_+)$ is closed by Proposition~\ref{prop:psd_im_closed}, 
we deduce that there exists $\hat{\epsilon} >0$, so that $\omega(\hat{\epsilon})$ 
lies on the boundary of $\PP_{E}({\bf S}^4_+)$, that is Slater's condition 
\underline{fails} for the completion problem $C(\hat{\epsilon})$. 
In fact, it can be shown that the smallest such
$\epsilon$ is $\hat \epsilon = \sqrt 2 -1$ with completion values of $0$
in the $?$ positions in $C(\hat \epsilon)$.
On the other hand, all specified principal matrices of $\omega(\epsilon)$ for $\epsilon >0$ are clearly positive definite, and therefore all the corresponding faces  are trivial.  
Thus we have found a partial matrix $\omega(\hat \epsilon)$ that has a
singular completion but the minimal face cannot be found from an
intersection using the cliques of the graph.
\end{example}

Given the importance of singularity degree, the following question
arises naturally. Which graphs $G=(V,E)$ have the property that the cone
$\PP_{E}({\bf S}^n_+)$ is facially exposed? Equivalently, on which
graphs $G=(V,E)$ does every feasible \PSD completion problem have
singularity degree at most one? Let us make the following definition. The
\textdef{singularity degree of a graph} $G=(V,E)$ is the maximal singularity
degree among all completion problems with
\PSD completable partial matrices $\omega\in \R^E$.

Chordal graphs have singularity degree one~\cite{DrPaWo:14}, and surprisingly these are the only graphs with this property~\cite{tanigawa}.

\begin{corollary}[Singularity degree of chordal completions]
	\label{cor:sing_psd}{\hfill \\ }
	The graph $G$ has singularity degree one if, and only if, $G$ is chordal.
\end{corollary}

\section{Euclidean distance matrix completion, \EDMCp}
\index{Euclidean distance matrix completion, \EDMCp}
\index{\EDMCp, Euclidean distance matrix completion}
In this section, we discuss a problem that is closely related to the \PSD
completion problem of the previous section, namely the Euclidean
distance matrix completion, \EDMCp, problem.  As we will see, the \EDMC
problem inherently fails Slater's condition, and facial reduction once
again becomes applicable by analyzing certain cliques in a graph.

Setting the stage, fix an undirected graph $G=(V,E)$ on a vertex set $V=\{1,2,\ldots,n\}$ with an edge set $E\subseteq\{ij: 1\leq i< j\leq n\}$. Given a partial matrix $d\in \R^E$, the {\em Euclidean distance matrix completion problem} asks to determine if possible an integer $k$ and a collection of points $p_1,\ldots,p_n\in \R^k$ satisfying
$$d_{ij}=\|p_i-p_j\|^2\quad \textrm{ for all } ij\in E.$$
See figure~\ref{fig:EDMC} for an illustration.

 \begin{figure}[!ht]
 	\centering
 	\includegraphics[scale=0.65]{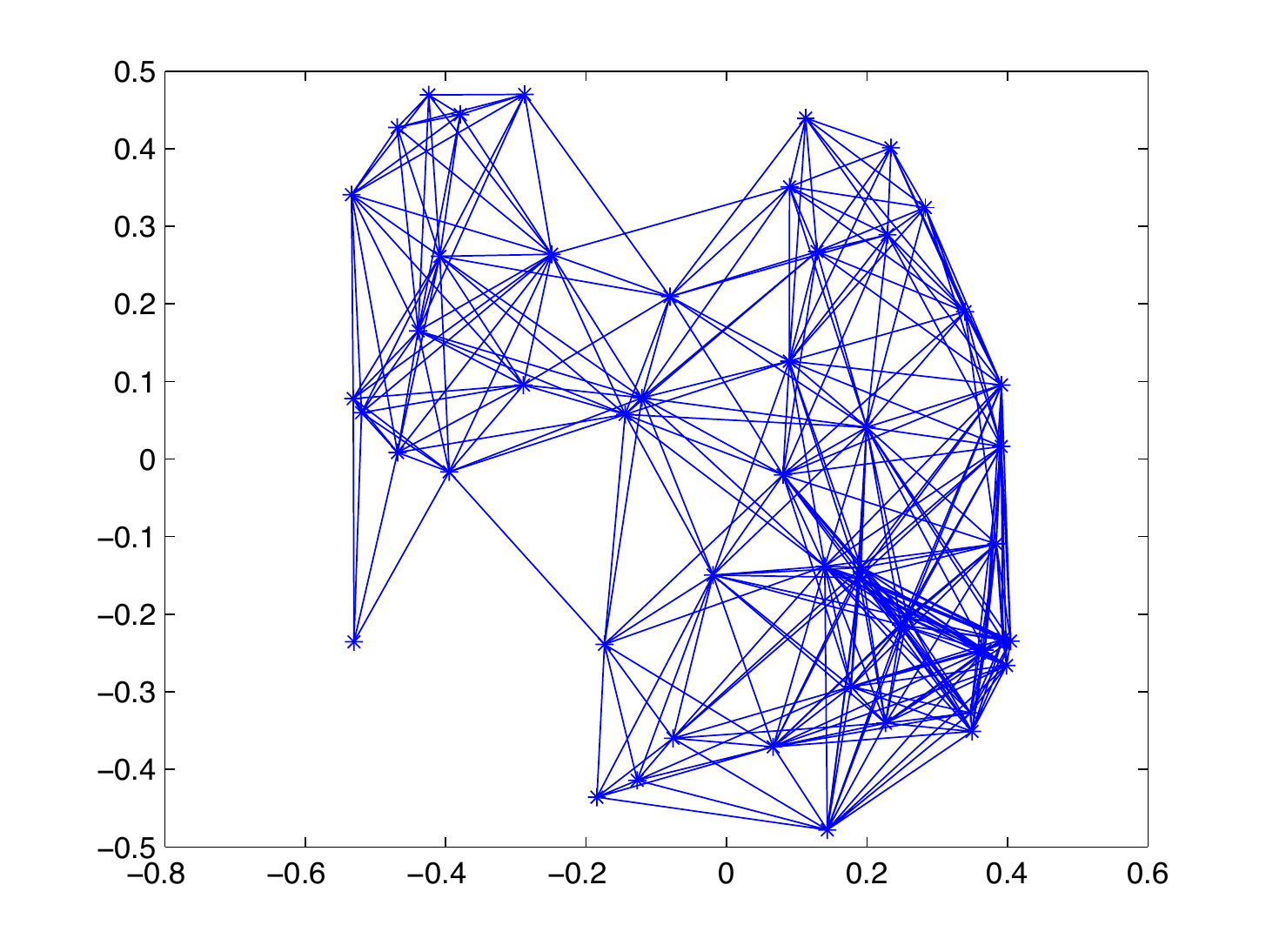}
 	\caption{Instance of EDMC}
 	\label{fig:EDMC}
 \end{figure}

We now see how this problem can be modeled as an \SDPp. To this end, 
let us introduce the following notation. A matrix $D\in {\bf S}^n$ is called a 
\textdef{Euclidean distance matrix, \EDMp}, if there exists an integer $k$ 
\index{\EDMp, Euclidean distance matrix}
and points $p_1,\ldots,p_n\in \R^k$ satisfying
$$D_{ij}=\|p_i-p_j\|^2\quad \textrm{ for all } i,j=1,\ldots,n.$$
 Such points $p_1,\ldots,p_n$ are said to {\em realize} $D$ in $\R^k$.
The smallest integer $k$ such that there exist points in $\R^k$ realizing $D$ is called the 
 \textdef{embedding dimension of $D$}, and is denoted by $\embdim D$. 
\index{$\embdim D$, embedding dimension of $D$}
We let $\En$ denote the set of all $n\times n$ \EDM matrices. In this
language the \EDM completion problem reads: given a partial matrix $d\in \R^E$ determine a matrix in the set 
$$\{D\in \En: D_{ij}=d_{ij} \quad \textrm{for all } ij\in E\}.$$
Thus the \EDM completion problem is a conic feasibility problem. Since we are interested in facial reduction, the facial structure of $\En$ is central.

Notice that $\En$ has empty interior since it is contained in the space of 
{\em hollow matrices}
\index{${\bf S}^n_H$, hollow matrices}
\index{hollow matrices, ${\bf S}^n_H$}
$${\bf S}^n_H:= \{D \in \Sn: \diag(D)=0\}.$$
A fundamental fact often used in the literature is that $\En$ is linearly isomorphic to ${\bf S}^{n-1}$.
More precisely, consider the mapping
$$\KK\colon \Sn \rightarrow \Sn$$ 
defined by  
\begin{equation}
\label{eq:defK}
\KK(X)_{ij}:=X_{ii}+X_{jj}-2X_{ij}.
\end{equation}
Clearly $\KK$ maps into the space of hollow matrices ${\bf S}^n_H$. One can quickly verify that
the adjoint is given by 
\begin{equation*}
\KK^*(D)=2(\Diag(De)-D).
\end{equation*}
Moreover, the range of the adjoint $\KK^*$ is the space of {\em centered matrices}
\index{${\bf S}^n_c$, centered matrices}
\index{centered matrices, ${\bf S}^n_c$}
$${\bf S}^n_c:= \{X \in \Sn: Xe=0\}.$$
The following result is fundamental.

\begin{theorem}[Parametrization of the \EDM cone]
The map $\KK\colon {\bf S}^n_c\to {\bf S}^n_H$ is a linear isomorphism carrying
${\bf S}^n_c\cap {\bf S}^n_+$ onto $\En$. The inverse $\KK^{\dagger}\colon 
{\bf S}^n_H\to {\bf S}^n_c$ is the map $\KK^{\dagger}(D)=-\frac 12 
	JD J$, where $J:=I-\frac 1n ee^T$ is the orthogonal projection onto
	$e^\perp$.\footnote{In fact, we can consider
the map as $\KK\colon {\bf S}^n\to {\bf S}^n$. Then we still have
$\KK({\bf S}^n_+) = \En$ and the \textdef{Moore-Penrose
pseudoinverse} $\KK^{\dagger}(D)=-\frac 12 
J\offDiag\left(D\right)J$, where $\offDiag$ zeros out the diagonal.}
\end{theorem}

In particular, the cone $\En$ is linearly isomorphic to the cone ${\bf S}^n_c\cap {\bf S}^n_+$. On the other hand, observe for any matrix $X\in{\bf S}^n_+$ the equivalence
$$ Xe=0~\Leftrightarrow~ e^TXe=0~\Leftrightarrow~\tr(Xee^T)=0 ~\Leftrightarrow~ X\in (ee^T)^{\perp}. $$
Thus $\En$ is linearly isomorphic to the
face $${\bf S}^n_c\cap {\bf S}^n_+=(ee^T)^{\perp}\cap{\bf S}^n_+$$ as claimed. More explicitly, forming  an $n\times n$ orthogonal matrix $\begin{bmatrix}
\frac{1}{\sqrt n}e & U  \cr
\end{bmatrix}$ yields the equality 
${\bf S}^n_c\cap \Snp= U{\bf S}^{n-1}_+U$.

Thus the \EDM completion problem amounts to finding a matrix $X$ in the set
\begin{equation}
	\label{eq:EDMC}
\mathcal{F}_{p}=\{X\in {\bf S}^n_c\cap {\bf S}^n_+: \mathcal{P}_{E}\circ\KK(X)=d\}.
\end{equation}
To see how to recover the realizing points $p_i$ of $d\in \R^E$, consider a matrix $X\in \mathcal{F}_{p}$ and form a factorization 
$X=PP^T$ for some matrix $P\in \R^{n\times k}$ with $k=\rank X$. Let $p_1,\ldots,p_n\in \R^k$ be the rows of $P$. Then $X$ lying in ${\bf S}^n_c $ implies $\sum^n_{i=1} p_i=0$, that is, the points $p_i$ are centered around the origin, while the constraint $\mathcal{P}_{E}\circ\KK(X)=d$ implies 
$$d_{ij}=X_{ii}+X_{jj}-2X_{ij}=\|p_i\|^2+\|p_j\|^2-2p_i^Tp_j=\|p_i-p_j\|^2,
$$
for all $ij\in E$. Hence the points $p_1,\ldots, p_n$ solve the \EDM completion problem.

Let us turn now to understanding (strict) feasibility of $\mathcal{F}_p$.
A vector $d\in \R^E$ is called a {\em partial \EDM} if the restriction of
$d$ to every specified principal submatrix is an \EDMp. 
The graph $G$ is \textdef{\EDM completable} if every partial \EDM $d\in
\R^E$ is {\em \EDM completable}.
The following result, proved in~\cite{MR1321802}, is a direct analogue
of Theorem~\ref{thm:partial_comp} for the \PSD completion problem.
\begin{theorem}[\EDM completability \& chordal graphs]
\label{thm:compledm}
The graph $G$ is \EDM completable if, and only if, $G$ is chordal.
\end{theorem}

We also mention in passing the following observation from~\cite{DrPaWo:14}.
\begin{theorem}[Closedness of the projected \EDM cone]
	\label{thm:close_EDM_fixed} 
	The projected image $\PP(\En)$ is always closed.
	\qed
\end{theorem}

Given a clique $\alpha$ in $G$, we let $\EE^{\alpha}$ denote the set of $|\alpha|\times |\alpha|$ Euclidean distance matrices indexed by $\alpha$. In what follows, given a partial matrix $d\in\R^E$, the restriction $d[{\alpha}]$ can then be thought of either as a vector in $\R^{E(\alpha)}$  or as a hollow matrix in $\Sc^{\alpha}$.
We also use the symbol $\KK_{\alpha}\colon \Sc^{\alpha}\to\Sc^{\alpha}$ to indicate the mapping $\KK$ acting on $\Sc^{\alpha}$. The following recipe provides a simple way to discover faces of the PSD cone containing the feasible region from specified cliques in the graph. 

\begin{theorem}[Clique facial reduction for \EDM completions]
	\label{thm:EDM_face}
Let $\alpha$ be any $k$-clique in the graph $G$.
Let $d\in\R^{E}$ be a partial Euclidean distance matrix and define the 
relaxation 
$$
\begin{array}{rcl}
\FF_p(\alpha) &:= &
\{X\in \Snp\cap {\bf S}^n_c:  P_{E[\alpha]}\circ \KK(X)=d[\alpha]\}.
\end{array}
$$
Then for any matrix $V_\alpha$ exposing $\face\big(\KK^{\dag}_{\alpha}(d[\alpha]),{\bf S}^{\alpha}_+\cap {\bf S}^{\alpha}_c\big)$, the matrix
$$
\PP^{*}_{E[\alpha]}V_{\alpha}\quad \textrm{ exposes }\quad \face(\FF_p(\alpha),\Snp \cap {\bf S}^n_c).
$$
\end{theorem}

In other words, the recipe is as follows. Given a clique $\alpha$ in $G$, consider the matrix $\mathcal{K}^{\dag}_{\alpha}(d[\alpha])\in {\bf S}^{\alpha}_+\cap {\bf S}^{\alpha}_c$. Let $V_{\alpha}\in {\bf S}^{\alpha}_+\cap {\bf S}^{\alpha}_c$ be an exposing vector of $\face\big(\KK^{\dag}_{\alpha}(d[\alpha]),{\bf S}^{\alpha}_+\cap {\bf S}^n_c\big)$. Then $\PP^{*}_{E[\alpha]}V_{\alpha}$ is an extension of $V_{\alpha}$ to ${\bf S}^{n}$ obtained by padding $V_{\alpha}$ with zeroes. The above theorem guarantees that the entire feasible region of \eqref{eq:EDMC} is contained in the face of $\Ss_c^n\cap \Snp$ exposed by $\PP^{*}_{E[\alpha]}V_{\alpha}$.

In direct analogy with Theorem~\ref{thm:cliquesuffPSD} for PSD completions, the minimal face is sure to be discovered in this way for chordal graphs.

\begin{theorem}[Clique facial reduction for \EDM is sufficient]
\label{thm:EDM_facesuff}
Suppose that $G$ is chordal, and consider a partial Euclidean distance matrix $d\in \R^{E}$ and the region
\[
F:=\{X\in \Ss_c^n\cap \Snp : [\KK(X)]_{ij} = d_{ij} \textrm{ for all } ij\in E\}.
\]
Let $\Theta$ denote the set of all maximal cliques in $G$, and for each $\alpha \in \Theta$ define
\[
F_{\alpha}:=\{X\in \Ss_c^n\cap \Snp : [\KK(X)]_{ij} = d_{ij} \textrm{ for all } ij\in E[\alpha]\}.
\]
Then the equality $$\face(F, \Ss_c^n\cap\Snp)=\bigcap_{\alpha\in \Theta}
\face(F_{\alpha},\Ss_c^n\cap\Snp)\qquad \textrm{ holds}.
$$ 
\end{theorem}

\begin{corollary}[Singularity degree of chordal completions]
	\label{cor:sing}
If the graph $G=(V,E)$ is chordal, then the \EDM completion problem has singularity degree at most one, when feasible.  	
\end{corollary}

Finally, in analogy with Corollary~\ref{cor:sing_psd}, the following is true.  Define the
\textdef{singularity degree of a graph} $G=(V,E)$ to be the maximal singularity
degree among all \EDM completion problems with
\EDM completable partial matrices $d\in \R^E$.
\begin{corollary}[Singularity degree of chordal completions]
	\label{cor:sing_psd}{\hfill \\ }
	The graph $G$ has singularity degree one if, and only if, $G$ is chordal.
\end{corollary}

\subsection{\EDM and \SNL with exact data}
\label{sect:edmsnlexact}
The material above explains in part the surprising success of the algorithm in 
\cite{kriswolk:09} for the \textdef{sensor network localization problem, \SNLp}.
\index{\SNLp, sensor network localization problem}
The \SNL  problem differs from the \EDM completion problem only in
that some of the points or \textdef{sensors} $p_i$ that define the problem are
in fact \textdef{anchors} and their positions are known. 
The algorithm proceeds by iteratively finding faces of the PSD cone from cliques and intersecting them two at a time, thereby decreasing the dimension of the problem in each step. In practice, this procedure often terminates with a unique solution of the problem.
We should mention that the anchors are a \emph{red herring}. Indeed, they should only be treated differently than the other sensors after \emph{all} the
sensors have been localized. In the post-processing step, a so-called Procrustes problem is solved to bring the putative anchors as close as possible to their original (known) positions and thus rotating the sensor positions appropriately. 
Another important point in applications is that the distances for
sensors that are close enough to each other are often known. This suggests that there are
often many local cliques in the graph. This means that the resulting
\SDP relaxation is highly degenerate but this degeneracy can be
exploited as we have seen above.

Some numerical results from the year 2010 in~\cite{kriswolk:09} appear in 
Table~\ref{table:SNL}. 
		\begin{table}[ht!]
        \begin{center}
\begin{tabular}{|ccc|c|c|}
\hline
\# sensors & \# anchors & radio range & RMSD & Time
\\ \hline
$20000$ & 9 & $.025$ & $5e{-16}$ & $25$s \\
$40000$ & 9 & $.02$ & $8e{-16}$ & $1$m $23$s \\
$60000$ & 9 & $.015$ & $5e{-16}$ & $3$m $13$s \\
\hline
$100000$ & 9 & $.01$ & $6e{-16}$ & $9$m $8$s \\
\hline
\end{tabular}
	\caption{Empirical results for \SNL}
	\label{table:SNL}
        \end{center}
        \end{table}
These results are on random, noiseless problems
using a 2.16 GHz Intel Core 2 Duo, 2 GB of RAM. The embedding
dimension is $r=2$ and the sensors are in a square region 
$[0,1] \times [0,1]$ with $m = 9$ anchors.
We use the  Root Mean Square Deviation to measure the quality of the solution:
                \[   
                        \mbox{RMSD} = 
                        \left( \frac{1}{n} \sum_{i=1}^n \| p_i - p_i^{\mbox{\tiny true}} \|^2 \right)^{1/2}
                \]
The \emph{huge} expected 
number of
constraints and variables in the four problems in Table~\ref{table:SNL} are
\[
	\begin{array}{rcl}
	M  &=&  \begin{pmatrix}3,078,915  & 12,315,351  & 27,709,309 &
76,969,790\end{pmatrix}
\\N  &=& 
10^9 \begin{pmatrix} 0.2000  & 0.8000  & 1.8000  & 5.0000 \end{pmatrix},
\end{array}
\]
respectively\footnote{The 2016 tarfile with MATLAB codes is available:
\begin{quote}
\tiny{
}
\end{quote}
}.


\subsection{Extensions to noisy \EDM and \SNL problems}
\label{sect:edmsnlnoisy}
When there is noise in the distance measurements -- the much more realistic setting -- the approach requires an intriguing modification. Let us see what goes wrong, in the standard approach.
Given a clique $\alpha$, let us form $\KK^{\dag}_{\alpha}(d[\alpha])$ as in Theorem~\ref{thm:EDM_face}. The difficulty is that this matrix is no longer PSD. On the other hand, it is simple enough to find the nearest matrix $W$  of ${\bf S}^{\alpha}_+\cap {\bf S}^n_c$ to $\KK^{\dag}_{\alpha}(d[\alpha])$. Let then $V_{\alpha}$ be a vector exposing $\face(W,{\bf S}^{\alpha}_+\cap {\bf S}^n_c)$. Letting $\Theta$ be the collection of cliques under consideration, we thus obtain faces $F_{\alpha}$ exposed by $\PP^{*}_{E[\alpha]}V_{\alpha}$ for $\alpha\in \Theta$. In the noiseless regime, the entire feasible region is contained in the intersection $\bigcap_{\alpha\in \Theta}F_{\alpha}$. In the noisy regime, this intersection likely consists only of the origin for the simple reason that randomly perturbed faces typically intersect only at the origin. Here is an elementary fix that makes the algorithm robust to noise.
Form the sum 
$$V:=\sum_{\alpha\in\Theta}\PP^{*}_{E[\alpha]}V_{\alpha}.$$
Again in the noiseless regime, Proposition~\ref{prop:inter_face_exp} implies that $V$ exposes precisely the intersection  $\bigcap_{\alpha\in \Theta}F_{\alpha}$. When noise is present, the matrix $V$ will likely have only one zero eigenvalue corresponding to the vector of all ones $e$ and the rest of the eigenvalues will be strictly positive. Suppose we know that the realization of the graph should lie in $r$-dimensional space. Then we can find a rank $n-{r+1}$ best PSD approximation of $V$ and use it to expose a face of the PSD cone. 
Under appropriate conditions, this procedure is indeed provably robust to noise and extremely effective in practice.
A detailed description of such a scheme is presented in \cite{ChDrWo:14}.

\section{Low-rank matrix completions}
In this section, we consider another example inspired by facial reduction. 
We will be considering matrices $Z\in \Rmn$; for convenience, we will index the rows of $Z$ by $ i\in \{1,\ldots, m\}$  and  the columns using $j \in \{m+1,\ldots,m+n\}$.
Consider two vertex sets $V_1=\{1,\ldots,m\}$ and $V_2:=\{m+1,\ldots,m+n\}$ and a bipartite graph $G=(V_1\cup V_2,E)$. 

Given a partial matrix $z\in \R^E$,
the \textdef{low-rank matrix completion problem, \LRMCp},
\index{\LRMCp, low-rank matrix completion problem}
aims to find a rank $r$ matrix $Z$ from the partially observed
elements $z$. 
A common approach (with statistical guarantees) is to instead solve the convex problem:
\begin{equation}\label{eqn:min_comp_nuc}
\min_{Z\in  \Rmn} \|Z\|_* \quad \textrm{ subject to } P_{E}(Z)=z,
\end{equation}
where $\|Z\|_*$ is the \textdef{nuclear
norm} -- the sum of the singular values of $Z$.
Throughout the section, we will make the following assumption: the solution of the convex problem \eqref{eqn:min_comp_nuc} coincides with the rank $r$ matrix $Z$ that we seek. There are standard statistical assumptions that one makes in order to guarantee this to be the case \cite{Fazel:02,Rechtparrilofazel,MR2565240}.

 It is known that this problem \eqref{eqn:min_comp_nuc}
\index{$\|Z\|_*$, nuclear norm}
can be solved 
efficiently using \SDPp. At first glance
it appears that this does not fit into our framework for problems where
strict feasibility fails; indeed strict feasibility holds under the appropriate reformulation below. We will see, however, that 
one can exploit the special structure at the \underline{\emph{optimum}} and
discover a face of the PSD cone containing an \underline{optimal solution}, thereby decreasing the dimension of the problem. Even though, this is not facial reduction exactly, the ideas behind facial reduction play the main role. 

Let us first show that the problem \eqref{eqn:min_comp_nuc} can be written equivalently as the SDP:
\begin{equation}
\label{sdpnuclear}
\begin{array}{rlll}
	&\min & \frac 12\trace(Y) \\
	&\text{s.t.}& Y=\begin{bmatrix}
		A & Z \cr Z^T & B
	\end{bmatrix} & \\
	  &         &  \PP_{E}(Z) = z \\
	&		  & Y \in \Ss_+^{m+n}
\end{array}.
\end{equation}
To see this, we recall a classical fact that the operator norm $\|\cdot\|_2$ of the matrix is dual to the nuclear norm  $\|\cdot\|_*$, that is 
$$\|Z\|_*=\sup\{\trace(Z^TX): \|X\|_{2}\leq 1\}.$$
Note the equivalence
$$\|X\|_2\leq 1~\Longleftrightarrow~ I-XX^T\succeq 0~\Longleftrightarrow~ \begin{bmatrix} I & X\\ X^T & I\end{bmatrix} \succeq 0.$$
Thus we may represent the nuclear norm through an \SDPp:
$$\|Z\|_*=\sup_X\left\{\trace(Z^TX):\begin{bmatrix} I & X\\ X^T &
I\end{bmatrix} \succeq 0\right\}.$$ The dual of this \SDP is 
\begin{equation*}
\begin{array}{rlll}
	&\min & \frac 12\trace(Y) \\
	&\text{s.t.}& Y=\begin{bmatrix}
		A & Z \cr Z^T & B
	\end{bmatrix} & \\
	&		  & Y \in \Ss_+^{m+n}
\end{array}.
\end{equation*}
Thus the problems \eqref{eqn:min_comp_nuc} and \eqref{sdpnuclear} are indeed equivalent. Let us moreover make the following important observation. Suppose that $Z$ is optimal for \eqref{eqn:min_comp_nuc}. Let $Z=U\Sigma V^T$ be a compact SVD of $Z$ and set $A:=U\Sigma U^T$ and $B:=V\Sigma V^T$. Then the triple $(Z,A,B)$ is feasible for \eqref{sdpnuclear}  since 
$$Y=\begin{bmatrix}
		A & Z \cr Z^T & B
	\end{bmatrix} = \begin{bmatrix} U\\ V\end{bmatrix}\Sigma\begin{bmatrix} U\\ V\end{bmatrix}^T\succeq 0.$$ 
Moreover $\rank(Y)=\rank(Z)$ and $\frac{1}{2}\trace(Y)=\frac{1}{2}(\trace(A)+\trace(B))=\trace(\Sigma)=\|Z\|_*$. Thus $Y$ is optimal for \eqref{sdpnuclear}.


Let us see now how we can exploit the structure and target rank $r$ of the 
problem to find an \emph{exposing vector} of a face containing an 
optimal solution
of the \SDPp. Fix two numbers $p,q> r$ and let $\alpha$ be any $p\times q$ 
complete bipartite subgraph  of 
$G$. Let also $z[\alpha]$ be the restriction of $z$ to $\alpha$. Thus   $z[\alpha]$ corresponds to a fully specified submatrix.



For almost any\footnote{This is in the sense of Lebesgue measure on the factors $P\in\R^{m\times r}$, $Q\in\R^{n\times r}$ satisfying $Z=PQ^T$.} rank $r$ underlying matrix $Z$, it will be the case that $\rank(z[\alpha]) =r$.
\index{specified submatrix, $z \in \R^{p \times q}$}
 Without loss of generality,
 after row and column permutations if needed, we can assume that $\alpha$ encodes the bottom left corner of $Z$: 
\[
	Z=
\begin{bmatrix}  Z_1  & Z_2     \cr 
		       z[\alpha] & Z_3 \end{bmatrix},
\]
that is
 $\alpha=\{m-p+1,\ldots, m\}\times\{m+1,\ldots,m+q\}$.
Form now the factorization $z[\alpha]=\bar P\bar D\bar Q^T$ obtained using
the compact SVD. Both $\bar P,\bar Q$ have rank $r$.

Let $Z=U\Sigma V^T$ be a compact SVD of $Z$ and define
$$Y=\begin{bmatrix} U\\ V\end{bmatrix}\Sigma\begin{bmatrix} U\\ V\end{bmatrix}^T.$$
As we saw previously, $Y$ is optimal for the SDP \eqref{sdpnuclear}.
Subdividing $U$ and $V$ into two blocks each, we deduce
\begin{equation}
	\label{eq:Ypartit}
0\preceq Y 
= 
\begin{bmatrix} U_1 \cr U_2 \cr V_1  \cr V_2 \end{bmatrix}
	\Sigma
\begin{bmatrix} U_1 \cr U_2 \cr V_1  \cr V_2 \end{bmatrix}^T
= 
\left[
\begin{array}{c|cc|c}  
	U_1\Sigma U_1 & U_1 \Sigma U_2^T& U_1 \Sigma V_1^T& U_1\Sigma V_2^T     \cr 
	\hline
	U_2\Sigma U_1^T & U_2\Sigma U_2^T& U_2\Sigma V_1^T& U_2\Sigma V_2^T     \cr 
	V_1\Sigma U_1^T & V_1\Sigma U_2^T& V_1\Sigma V_1^T& V_1\Sigma V_2^T     \cr 
	\hline
	V_2\Sigma U_1^T & V_2\Sigma U_2^T& V_2\Sigma V_1^T& V_2\Sigma V_2^T     \cr 
		       \end{array}
\right].
\end{equation}
Therefore, we conclude that $z[\alpha]=U_2\Sigma V_1^T = \bar P\bar D\bar Q^T$. Taking into account that $z[\alpha]$ has rank $r$ and the matrices $U_2,V_1,\bar P, \bar Q$ have exactly $r$ columns we deduce
\begin{equation}
	\label{eq:PQbar}
\begin{array}{c}
\Range(z[\alpha]) = \Range(U_2) = \Range(\bar{P}), \\
\Range(z[\alpha]^T) = \Range(V_1) = \Range(\bar{Q}).
\end{array}
\end{equation}

We can now use the \textdef{exposing vector} form of \FR
formed from $\bar P$ and/or $\bar Q$. Using the calculated $\bar P,\bar
Q$, let $\bar E\in \R^{p\times (p-r)}$ and $\bar F\in \R^{q\times (q-r)}$ satisfy  $\Range \bar E=(\Range \bar P)^{\perp}$ and $ \Range \bar F=(\Range \bar Q)^{\perp}$.
Define then the PSD matrix 
\[
\overline W=\left[
\begin{array}{c|cc|c}  
	0 & 0 & 0 & 0      \cr 
	\hline
	0  & \bar E \bar E^T & 0 & 0      \cr 
	0  & 0 & \bar F \bar F^T & 0      \cr 
	\hline
	0  &0 & 0 & 0      \cr 
		       \end{array}
\right] \succeq 0.
\]
By construction $\overline W Y=0$. Hence $\overline W$ exposes a face of the PSD cone containing the optimal $Y$.

Performing this procedure for many specified submatrices $z\in \R^E$, can yield a dramatic decrease in the dimension of the final SDP that needs to be solved. When noise is present, the strategy can be made robust in exactly the same way as for the EDM problem in Section~\ref{sect:edmsnlnoisy}.

We include one of the
tables of numerics from~\cite{HuangWolkXYe:16} in 
Table~\ref{table:noiseless4}, page \pageref{table:noiseless4}.
Results are for the average of five instances. We have recovered the
correct rank $4$ each time without calling an SDP solver at all. Note that the largest matrices recovered
have $(2,500)\times (20,000)=50,000,000$ elements.
	\begin{table*}[!h]
	\centering
	\caption{noiseless: $r=4$;
		$m \times n$ size; density $p$.}
\vspace{-.1in}
		\label{table:noiseless4}
		\begin{tabular}{|ccc||c|c|c|} \hline
\multicolumn{3}{|c||}{Specifications} & \multirow{2}{*}{Time (s)} & \multirow{2}{*}{Rank} &\multirow{2}{*}{Residual (\%$Z$)}\cr\cline{1-3}
  $m$ &   $n$ & mean($p$) &        &         &           \cr\hline
  700 &  2000 &  0.36 &  12.80 &     4.0 & 1.5217e-12 \cr\hline
 1000 &  5000 &  0.36 &  49.66 &     4.0 & 1.0910e-12 \cr\hline
 1400 &  9000 &  0.36 & 131.53 &     4.0 & 6.0304e-13 \cr\hline
 1900 & 14000 &  0.36 & 291.22 &     4.0 & 3.4847e-11 \cr\hline
 2500 & 20000 &  0.36 & 798.70 &     4.0 & 7.2256e-08 \cr\hline
\end{tabular}

\end{table*}

\section{Commentary}
The work using chordal graphs for \PSD completions was done in
\cite{GrJoSaWo:84} and extended the special case of banded structure
in~\cite{DymGoh:81}. Surveys for matrix completion are given in
e.g.,~\cite{MR1059481,MR1059486,AlWo:99,MR2014037,MR2487961,MR2508310,MR1044604,joh:90}.
A survey specifically related to chordality is given in
\cite{VandenbergeAndersen:15}.
More details on early algorithms for \PSD completion
are in e.g.,~\cite{JoKrWo:95,AlWo:99}.

The origin of distance geometry problems can be traced back to the work
of Grassmann in 1896~\cite{MR1747519}. More recent work appeared in
e.g.,~\cite{MR86j:62133,MR83a:51024,MR975025,MR783122,MR89c:51020}. 
Many of
these papers emphasized the relationships with molecular conformation.
Chordality and relations with positive definiteness are studied in
\cite{MR1340702} and more recently in~\cite{Laur:97b,MR99c:05135}.
The book~\cite{MR2807419} has a chapter on matrix completions with
the connections to \EDM completions, see also~\cite{MR1366579} for the
relations with faces. An excellent online reference for \EDM is the book by Dattorro
\cite{dattorroonline:05}. In addition, there are many excellent survey articles,
e.g.,~\cite{KrislockWolk:10,DBLP:journals/corr/DokmanicPRV15,surv,AlfakihAnjosKPW:08,MR3059590}.
The survey~\cite{2016arXiv161000652L} contains many open problems in 
\EDMC and references for application areas.

Early work using \SDP interior point algorithms for  \EDM completion problems  is given in~\cite{AlKaWo:97}. 
Exploiting the clique structure for \SNL type problems is done in
e.g.,~\cite{DiKrQiWo:08,krislock:2010,kriswolk:09}. The improved robust
algorithm based on averaging approximate exposing vectors was developed
in \cite{ChDrWo:14}, while a parallel viewpoint based on rigidity theory
was developed in \cite{face_sing}. In fact, a parallel view on facial
reduction is based on rigidity theory,
e.g.,~\cite{GortlerThurston:14,MR3341582,Alfak:00,MR3216675}. 
The facial structure for the \EDM cone is studied in
e.g.,~\cite{MR2209243,MR2166851}.
Applications of the technique to
molecular conformation are in \cite{MR3042029}. 

The \LRMC problem has parallels in the compressed sensing 
framework that is currently of great interest. 
The renewed interest followed the work
in~\cite{fazelhindiboyd:01,Fazel:02,MR2565240,Rechtparrilofazel} that
used the nuclear norm as a convex relaxation of the rank function.
Exploiting the structure of the optimal face using \FR is introduced recently
in~\cite{HuangWolkXYe:16}.
An alternative approach, which applies much more broadly, is described in~\cite{2016arXiv160802090P}.


\chapter{Hard combinatorial problems}
\label{chap:hardcomb}

\section{Quadratic assignment problem, \QAPp}
\label{sect:qap}
The \textdef{quadratic assignment problem, \QAPp}, is arguably the hardest of the
\index{\QAPp, quadratic assignment problem}
so-called NP-hard combinatorial optimization problems. The problem can
best be described in terms of facility location. We have $n$ given facilities
that need to be located among $n$ specified locations. 
As input data, we have information on the
distances $D_{ij}$ between pairs of locations $i,j$ and the flow values
(weights) $F_{st}$ between pairs of facilities $s,t$. 
The (quadratic) cost of a possible location is the flow between each pair 
of facilities multiplied by the distance between their assigned locations. Surprisingly,
problems of size $n\geq 30$ are still considered hard to solve.
As well, we can have a (linear)
cost $C_{kl}$ of locating facility $k$ in location $l$. 
The unknown variable that decides which facility goes into which location is
an $n\times n$ \textdef{permutation matrix $X=(X_{kl})\in \Pi$} with
\[
	X_{kl}=\left\{ \begin{array}{rl} 
1 & \text{if facility  } k \text{ is assigned to location  } l \\
		0 & \text{otherwise}
	\end{array} \right..
\]
This problem has the elegant trace formulation
\[
\min_{X\in \Pi}\, \trace(FXDX^T) + \trace (CX^T).
\]
Notice that the objective is a quadratic function, and typically the
quadratic form, $\trace (FXDX^T)$, is indefinite.\footnote{One can
perturb the objective function by exploiting the structure of
the permutation matrices and obtain positive definiteness of the
quadratic form. However, this
can result in deterioration of the bounds from any relaxations.}

Notice also that the feasible region consists of permutation matrices, a discrete set.
There is a standard strategy for forming a semi-definite programming relaxation for such a problem. 
Consider the \textdef{vectorization $x:=\kvec(X)\in \R^{n^2}$} and define the 
\textdef{lifting} to the rank one \emph{block matrix} 
$$
Y
=
\begin{pmatrix} 1 \\	x \end{pmatrix}
\begin{pmatrix} 1 \\	x \end{pmatrix}^T
=
\begin{bmatrix}
1& x^T \\	
x& xx^T
\end{bmatrix} 
=
\begin{bmatrix} 
	1 & x^T \cr
	x & \begin{bmatrix}
	    X_{:,i} X_{:,j}^T
            \end{bmatrix}
\end{bmatrix} 
\in \Sntop,
$$
where the matrix consists of $n^2, n\times n$ blocks 
$X_{:,i} X_{:,j}^T$ beginning in row and column $2$.
The idea is then to reformulate the objective and a relaxation of the feasible region linearly in terms of $Y$, and then simply insist that $Y$ is PSD, though not necessarily rank one.
In particular, the objective function can easily be rewritten as a {\em linear function} of $Y$, namely $\trace(LY)$, where 
$$
L:=\begin{bmatrix}
	0 & \frac 12 \kvec(C)^T \cr \frac 12\kvec(C) & D\otimes F
\end{bmatrix},
$$
and we denote the \textdef{Kronecker product, $D\otimes F$}.
\index{$D\otimes F$, Kronecker product}

Next we turn to the constraints. We seek to replace the set of permutation matrices by more favorable constraints that permutation matrices satisfy.
For example, observe that the permutation matrices are doubly stochastic
and hence the row sums and column sums are one, yielding the following
\textdef{linear assignment constraints}
\begin{equation}
	\label{eq:Xee}
	Xe=X^Te=e.
\end{equation}
There are of course many more possible constraints one can utilize; the
greater their number, even if redundant, the tighter the \SDP relaxation in general. Some 
prominent ones, including the ones above, are 
\begin{subequations}
		\label{eq:consrqap}
	\begin{align}
		Xe=X^Te&=e \label{eqn:ones}\\
		X_{ij}^2 - X_{ij} &=0, \label{eqn:zeroone}\\
		XX^T = X^TX&=I, \label{eqn:orthog}\\
		X(:,i)\circ X(:,j)&=0, 
	X(k,:)\circ X(l,:)=0, ~ \textrm{ for all } i\neq j, k\neq l,
		 \label{eqn:gang}
\end{align}
\end{subequations}
where \textdef{$\circ$} denotes the \textdef{Hadamard (elementwise) product}.
Note that including both equivalent orthogonality constraints $XX^T = X^TX=I$ 
is not redundant in the relaxations.

Let us see how to reformulate the constraints linearly in $Y$. We first consider the $n$ linear row sum
constraints $(Xe-e)_i=0$ in \eqref{eqn:ones}. To this end, observe
\[
\begin{array}{rcl}
0
&=&
(Xe-e)_i
\\&=&
e_i^T(Xe-e) 
\\&=&
\trace (X^Te_ie^T)-1

\\&=&
\begin{pmatrix} 1 \cr x \end{pmatrix}^T
\begin{pmatrix} -1 \cr \kvec(e_ie^T) \end{pmatrix}.
\end{array}
\]
We obtain 
\[
	0=
\begin{pmatrix} 1 \cr x \end{pmatrix}
\begin{pmatrix} 1 \cr x \end{pmatrix}^T
\begin{pmatrix} -1 \cr \kvec(e_ie^T) \end{pmatrix}
\begin{pmatrix} -1 \cr \kvec(e_ie^T) \end{pmatrix}.
\]
Defining now the matrix 
$$E_r=\begin{bmatrix}
-1 & -1  &\ldots & -1\\
\kvec(e_1e^T) &\kvec(e_2e^T) &\ldots & \kvec(e_ne^T)
\end{bmatrix},
$$
we obtain the equivalent linear homogeneous constraint $\trace(YE_rE_r^T)=0$. 
Similarly the $n$ linear column sum
constraints $(X^Te-e)_i=0$ amount to the equality
$\trace(YE_cE_c^T)=0$, where 
$$E_c=\begin{bmatrix}
-1 & -1  &\ldots & -1\\
\kvec(ee_1^T) &\kvec(ee_2^T) &\ldots & \kvec(ee_n^T)
\end{bmatrix}.
$$
Thus the feasible region of the \SDP relaxation lies in the face of $\Sntop$ exposed by $D_0:=E_rE_r^T+E_cE_c^T$. Henceforth, let \textdef{$\widehat V$} be full column rank and satisfying 
$\range (\widehat V) =\nul(D_0)$. 


The other three constraints in \eqref{eq:consrqap} can be rephrased linearly in $Y$ as well: 
\eqref{eqn:zeroone} results in the so-called
\textdef{arrow constraint} (the first row (column) and the diagonal of $Y$ are equal);
the constraints \eqref{eqn:orthog} yield the 
\textdef{block-diagonal} constraint (diagonal blocks sum to the identity matrix) and the 
\textdef{off-diagonal} contraint (the traces of the off-diagonal blocks are zero); and the
Hadamard orthogonality constraints \eqref{eqn:gang} are called the \textdef{gangster
constraints} and guarantee that the diagonal blocks are diagonal
matrices and the diagonals of the off-diagonal blocks are zero.
We omit further details\footnote{See more details in~\cite{KaReWoZh:94}.}
but denote the resulting constraints with the additional
$Y_{00}=1$ in the form $\A(Y)=b$. 
We note that the transformation $\A$ without the Hadamard orthogonality
constraints is onto while $\A$ is not. We numerically test both settings with and
without the gangster constraints and together with and
without facial reduction below in this section. 

Now the standard relaxation of the problem is obtained by letting $Y$ be a 
positive semi-definite matrix with no constraint on its rank: 
\begin{equation}
	\label{eq:firstmainqaprelax}
	\begin{array}{rl}
\min~ &\trace(LY)\\
\textrm{s.t. }~& \mathcal{A}(Y)=b\\
&Y\in \Sc^{n^2+1}_+.
\end{array}
\end{equation}
All in all, the number of linear constraints is 
\[
m_A=1+n+(n(n-1)/2)+n^2+2(n(n(n-1)/2)) +1,
\]
i.e.,
\[
	\A: \Sntop \rightarrow  \R^{\scriptsize{
	(n^3+\frac {n^2}2 + \frac n2 +2)}}.
\]

As discussed above, the matrix $D_0$ certifies that this relaxation fails strict feasibility. 
Indeed the entire feasible region
lies in the face of $\Sc^{n^2+1}_+$ exposed by $D_0$.
Surprisingly, after restricting to the face exposed by  $D_0$, the constraints $\mathcal{A}(Y)=b$ simplify dramatically. The resulting equivalent formulation becomes
\begin{equation}
	\label{eq:gangsterqap}
\begin{array}{rl}
\min  & \trace (\widehat V^TL\widehat V R) \\
\text{s.t.} & \GG(\widehat V R \widehat V^T) = e_0e_0^T \\
				 & R \in \Sc_+^{(n-1)^2+1},
\end{array}
\end{equation}
where $e_0$ is the first unit vector, as we start indexing at $0$, and
\[
\GG(Y)_{ij}=\left\{
\begin{array}{ll}
1 & \text{if  } ij \in \bar \JJ\\
0 & \text{otherwise  }
\end{array}
\right.
\]
and $\bar \JJ$ is an appropriately defined index set; see~\cite{KaReWoZh:94}.
Roughly speaking, this index set guarantees that the diagonal blocks of 
$Y=\widehat V R \widehat V^T$ are diagonal matrices and the diagonal
elements of the off-diagonal blocks of $Y$ are all zero. 
In particular, one can show that the resulting linear constraint is surjective.

In fact, the \textdef{gangster operator} and gangster constraints
guarantee that most of the
Hadamard product constraints in \eqref{eqn:gang} hold.
And the constraints corresponding to the linear constraints in
\eqref{eqn:ones}, the arrow constraint in \eqref{eqn:zeroone}, 
the block-diagonal and off-diagonal constraints in \eqref{eqn:orthog} 
and some of the gangster constraints in \eqref{eqn:gang} have all
become redundant, thereby illuminating the strength of the facial
reduction together with the gangster constraints \eqref{eqn:gang}.

Moreover, we can rewrite the linear constraints in
\eqref{eq:gangsterqap} as
\[
\trace \left(Y\left(
\widehat V_{j:}^T\widehat V_{i:}+ 
\widehat V_{i:}^T\widehat V_{j:}\right)\right)=0, \quad \forall ij
	\in \bar \JJ.
\]
We see that these \emph{low rank}\footnote{These constraints are rank
	two. Low rank constraints can be exploited in several of the
current software packages for \SDPp.} constraints are linearly independent
and the number of constraints has been
reduced from $m_A=n^3+\frac {n^2}2 + \frac n2 +2$ to
\[
	|\bar \JJ|=
	1+  n(n(n-1)/2) + n(n(n-1)/2 - (n-1) -1)
	=n^3-2n^2+1,
\]
i.e.,~the number of constraints is still $O(n^3)$ but has decreased by
${1+\frac {5n^2+n}2}$.

Finally, we should mention that the singularity degree of the \SDP relaxation
of the \QAP is $d=1$. The problem \eqref{eq:gangsterqap} has a strictly
feasible point $\hat R$. Moreover, one can show that the dual of \eqref{eq:gangsterqap}
also has a strictly feasible point,
see~\cite{KaReWoZh:94,OliveiraWolkXu:15}.

Let us illustrate empirically the improvement in accuracy and cputime
for the facially reduced \SDP relaxation of the QAP.
We use the model in
\eqref{eq:firstmainqaprelax} and compare it to the simplified 
facially reduced model in \eqref{eq:gangsterqap}. See Figure
\ref{fig:cputimes}, page \pageref{fig:cputimes}, and Figure \ref{fig:acc},
page \pageref{fig:acc}. The improvement in accuracy and cputime is
evident.
\begin{figure}[!ht]
\centering
    \includegraphics[scale=0.67]{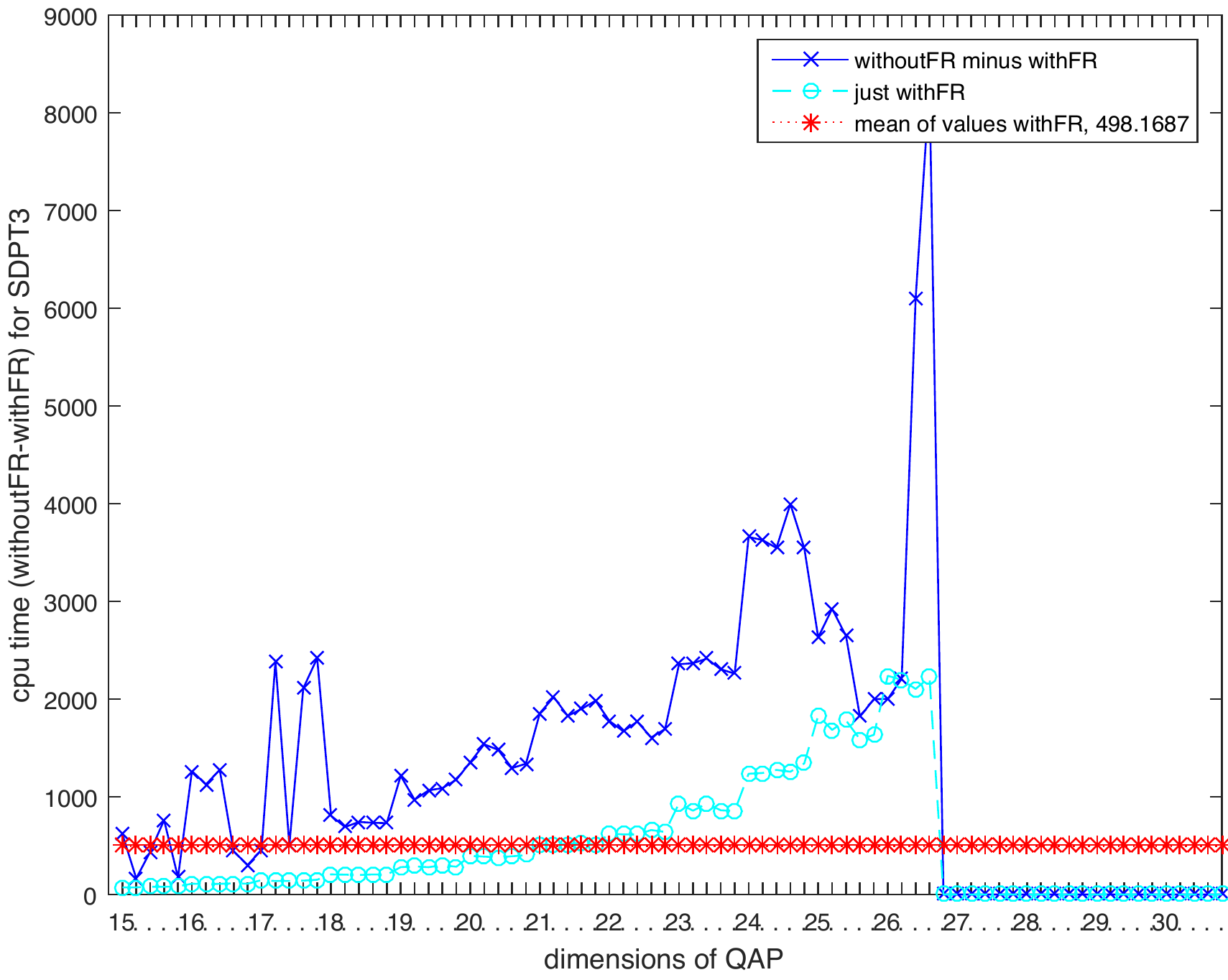}
    \caption{Difference in cpu seconds 
    (without \FRp $-$ with \FRp)}
    \label{fig:cputimes}
\end{figure}
\begin{figure}[!ht]
\centering
    \includegraphics[scale=0.67]{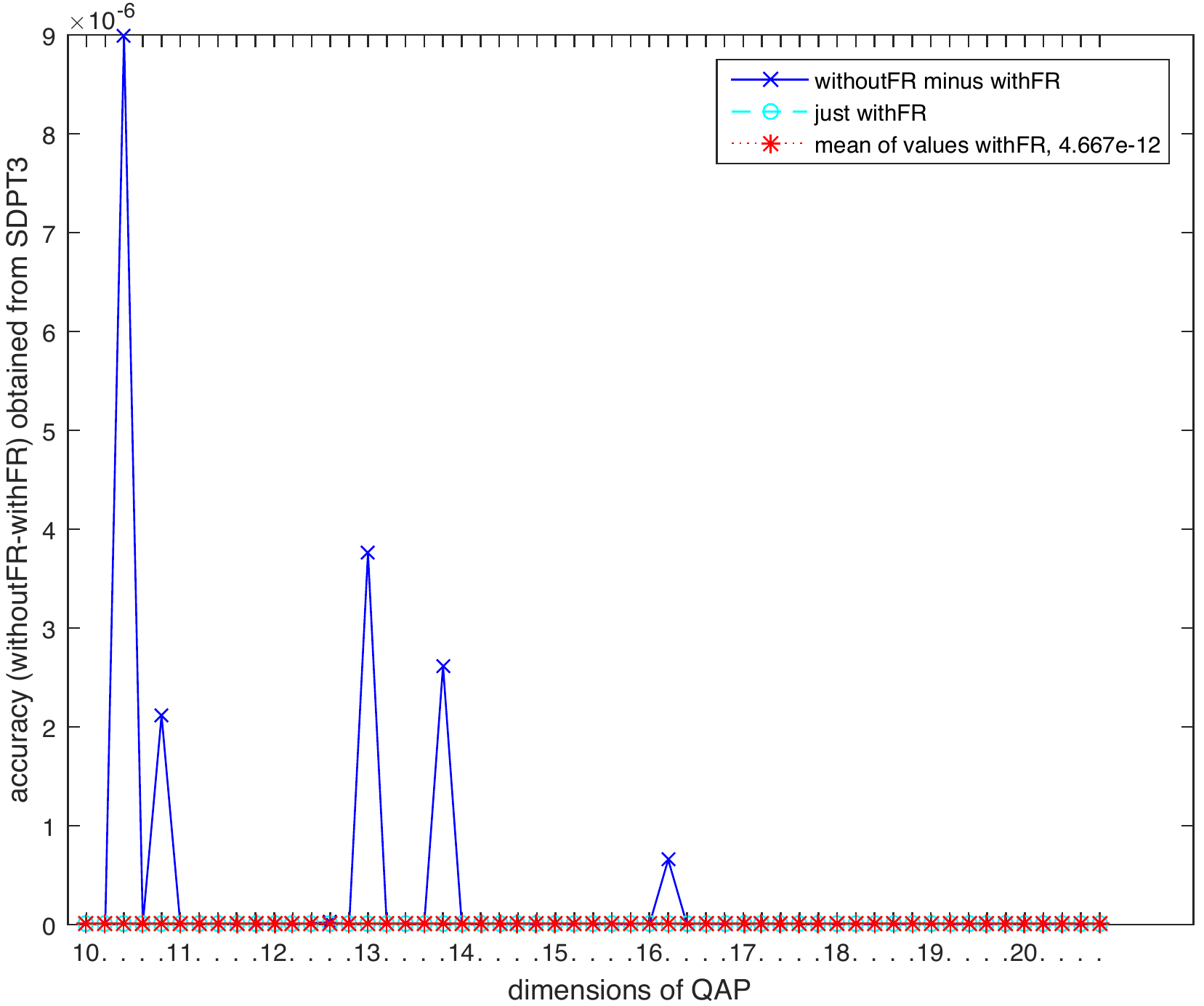}
    \caption{Difference in accuracy values (without \FRp $-$ with
    \FRp)}
    \label{fig:acc}
\end{figure}

\section{Second lift of Max-Cut}
\label{sect:secliftMC}

Recall that for a given weighted undirected graph $G=(V,E,W)$, 
the maximum cut problem is to determine a vertex set $S$ such that the total 
weight of the edges between $S$ and its complement $S^c$ is as 
large as possible. Thus enumerating the vertices $V=\{1,\ldots, n\}$, 
we are interested in the problem
\[
	(\MCp) \qquad \begin{array}{rl}
\max &\frac{1}{2} \sum_{i <j} w_{ij} (1-x_ix_j)\\
\textrm{s.t. } & x_i\in \{\pm 1\}, \qquad\textrm{ for } i=1,\ldots,n.
\end{array}
\]
Here, we have $x_i=1$ for $i\in S$ and $x_i=-1$ for $i\notin S$. Notice
the constraints $x_i\in \{\pm 1\}$ can equivalently be written with the
quadratic constraint 
\begin{equation}
	\label{eq:diagq}
	\diag(xx^T)=e.
\end{equation}
Relaxing $xx^T$ to a positive
semi-definite matrix $X$, we arrive at the celebrated 
SDP relaxation of Max-Cut:
\begin{equation}\label{eqn:orign_relax}
\begin{array}{rl}
\max & \frac{1}{4}\trace(LX)\\
\textrm{s.t. } & \diag(X)=e\\
& X\succeq 0
\end{array}
\end{equation}
Here $L$ denotes the weighted Laplacian matrix of the graph, which will
not play a role in our discussion. This \SDP is clearly strictly feasible. 

Another idea now to improve the accuracy of the relaxation is to ``extend the lifting''. Namely, with the goal of
tightening the approximation to the original Max-Cut problem, we can
certainly \emph{add} the following quadratic constraints to the \SDP relaxation:
\begin{equation}\label{eqn:nonlin_constr}
X_{ij}X_{jk}=X_{ik}, \qquad \textrm{ for all } i,j,k.
\end{equation}
Let us see how to form a relaxation with these nonlinear constraints.

For $X\in \Sn$, let \textdef{$\svec(X)$} denote the vector formed from the upper
triangular part of $X$ taken columnwise with the strict upper triangular
part multiplied by $\sqrt 2$.
By abuse of notation, we let $x=\svec(X)$ and define the matrix $Y=
\begin{pmatrix} 1 \\	x \end{pmatrix}
\begin{pmatrix} 1 \\	x \end{pmatrix}^T$.
We can now form a new \SDP relaxation by insisting that $Y$ is PSD (though not rank one) and rewriting the constraints linearly in $Y$. The nonlinear constraints 
\eqref{eqn:nonlin_constr} can indeed be written linearly in $Y$; we omit the details.  On the other hand, note that the $i$-th constraint in the original \SDP relaxation \eqref{eqn:orign_relax} is equivalent to
\[
	\begin{array}{rcl}
	0 
	&=& 
	  \langle e_i,\diag(X)-e \rangle 
	\\&=& 
	  \langle \Diag(e_i),X \rangle -1
	  \\&=& x^T \svec(\Diag(e_i)) -1.
  \end{array}
\]
Exactly, the same way as in Section \ref{sect:qap}, we can define the matrix
$$E:=\begin{bmatrix}
-1 & -1  &\ldots & -1\\
	\svec(\Diag(e_1)) &\svec(\Diag(e_2)) &\ldots & \svec(\Diag(e_n))
\end{bmatrix},
$$
which certifies that strict feasibility fails and that the entire
feasible region lies in the face of 
$\Sc^{\frac{n(n+1)}{2}+1}_+$ exposed by $EE^T$. It turns out that this
second lift of max-cut, in practice, provides much tighter bounds than
the original \SDP relaxation \eqref{eqn:orign_relax}, and the elementary facial reduction step using $EE^T$ serves to stabilize the problem.

\section{General semi-definite lifts of combinatorial problems}
\label{sect:sdplifts}
Let us next look at a general recipe often used for obtaining \SDP relaxations of NP-hard problems; elements of this technique were already used in the previous sections. Consider a nonconvex feasible region of the form
$$\mathcal{F}:=\left\{x\in\R^n: \mathcal{A}\begin{pmatrix} 
1 & x^T\\
x & xx^T
\end{pmatrix}=0\right\}$$
where $\mathcal{A}\colon\Sc^n\to \R^m$ is a linear transformation. An \SDP relaxation of this region is the set
$$\widehat{\mathcal{F}}= \left\{X\in\Sc^n_+: \mathcal{A}(X)=0,\, X_{11}=1 \right\}.$$
Indeed, $\mathcal{F}$ is the image of a linear projection of the intersection of 
$\widehat{\mathcal{F}}$ with rank one matrices. For this reason
$\widehat{\mathcal{F}}$ is often called an {\em \SDP lift} of $\mathcal{F}$. 

In applications, such as the ones in the previous sections, the affine hull of $\mathcal{F}$ may not be full dimensional. For example, the affine hull of the set of permutation matrices (used for QAP) has empty interior. To this end, suppose that the affine hull of $\mathcal{F}$ is given by $\{x:Lx=l\}$, where $L$ is a linear transformation and $l$ a vector. Define the matrix 
$\widehat{L}=\begin{bmatrix} 
-l & L\\
\end{bmatrix}$.
Then clearly there is no harm in including the redundant constraint 
\begin{equation}\label{eqn:redund}
\Big\langle\widehat{L}^T\widehat{L},\begin{pmatrix} 
1 & x^T\\
x & xx^T
\end{pmatrix}\Big\rangle=0
\end{equation}
in the very definition of $\mathcal{F}$. Notice then
$\widehat{\mathcal{F}}$ is clearly contained in the face of $\Sc^n_+$ exposed by $\widehat{L}^T\widehat{L}$.
Indeed, this is the minimal face of $\Sc^n_+$ containing $\widehat{\mathcal{F}}$. To see this, suppose that the affine span of $\mathcal{F}$ has dimension $d$, and consider any affinely independent vectors $x_1,\ldots, x_{d+1}\in\mathcal{F}$. Then the vectors 
$\begin{pmatrix} 
1 \\
x_1
\end{pmatrix}, \begin{pmatrix} 
1 \\
x_2
\end{pmatrix}, \ldots, \begin{pmatrix} 
1 \\
x_{d+1}
\end{pmatrix}$ are linearly independent, and therefore the barycenter
$$\frac{1}{d+1}\sum^{d+1}_{i=1}  \begin{pmatrix} 
1 & x_i^T\\
x_i & x_ix_i^T
\end{pmatrix}$$ is a rank $d+1$ matrix lying in $\widehat{\mathcal{F}}$.
On the other hand, it is immediate that the face of $\Sc^n_+$ exposed by
$\widehat{L}^T\widehat{L}$ also has dimension $d+1$. The claimed
minimality follows. It is curious to note that if the constraint
\eqref{eqn:redund} were not explicitly included in the definition
$\mathcal{F}$, then the \SDP lift $\widehat{\mathcal{F}}$ could nevertheless be strictly feasible, and hence unnecessarily large. 

\begin{example}[Strictly feasible \SDP lifts]
Consider the region:
	$$\left\{(x,y,z) \in\R^3: x^2=y^2=z^2=1,\quad xy+yz=0\right\}.$$
There are only four feasible points, namely $\left\{\pm\begin{pmatrix} 
1 \\
1\\
-1
\end{pmatrix}, \pm\begin{pmatrix} 
1 \\
-1\\
-1
\end{pmatrix}\right\}$, and they affinely span the two dimensional subspace perpendicular to the vector
	$ \begin{pmatrix} 1&0&1\end{pmatrix}^T $. If this constraint is not included explicitly, then the \SDP lift is given by 
$$\{X\in\Sc^4_+: X_{23}=-X_{34}, ~X_{ii}=1 \textrm{ for } i=1,\ldots,4\}.$$
In particular, the identity matrix is feasible. 
\end{example}



\section{Elimination method for sparse SOS polynomials}
Checking whether a polynomial is always nonnegative is a ubiquitous task
in computational mathematics. This problem is NP-hard, as it for example
encompasses a great variety of hard combinatorial problems. Instead a
common approach utilizes sum of squares formulations. Indeed, checking
whether a polynomial is a sum of squares of polynomials can be modeled
as an \SDP. A certain hierarchy of sum of squares problems
\cite{Lasserre01globaloptimization,PP_hiarchy} can then be used to
determine the nonnegativity of the original polynomial. The size of the
\SDP arising from a sum of squares problem depends on the number of
monomials that must be used in the formulation. In this section, we show
how facial reduction iterations on the cone of sums of squares
polynomials can be used to eliminate monomials yielding a smaller and
better conditioned equivalent \SDP formulation.
A rigorous explanation of the material in this section requires some heavier notation; therefore we only outline the techniques.

Let $\R[x]_{n,2d}$ denote the vector space of polynomials in $n$ variable with real coefficients of degree at most $2d$.  We will write a polynomial $f\in \R[x]_{n,2d}$ using multi-index notation
$$f(x)=\sum_{\alpha\in N} c_{\alpha}x^{\alpha},$$
where $N$ is some subset of $\mathbb{N}^n$, we set $x^{\alpha}=x_1^{\alpha_1}\cdots x_n^{\alpha_n}$, and $c_{\alpha}$ are some real coefficients. We will think of $\R[x]_{n,2d}$ as a Euclidean space with the inner product being the usual dot product between coefficient vectors.
 Let  $\Sigma_{n,2d}\subseteq \R[x]_{n,2d}$ be the set of polynomials $f\in \R[x]_{n,2d}$ that are sums of squares, meaning that $f$ can be written as $\sum_i f_i^2$ for some polynomials $f_i$.  
 Clearly $\Sigma_{n,2d}\subseteq \R[x]_{n,2d}$ is a closed convex cone, often called the SOS cone. 
 
 A fundamental fact is that membership in the SOS cone $\Sigma_{n,2d}$
can be checked by solving an \SDP. 
  \begin{theorem}
 Fix a set of monomials $M\subset \mathbb{N}^n$. 
Then a polynomial $f\in \R[x]_{n,2d}$ is a sum of squares of polynomials over the monomial set $M$ if and only if there exists a matrix $Q\succeq 0$ so that $f(x)=[x]_{M}^TQ[x]_M$, where $[x]_M$ is a vector of monomials in $M$.
\end{theorem}
\begin{proof}
 If $f$ is a sum of squares $f=\sum_{i} f_i^2$, then we can form a matrix $P$ whose rows are the coefficient vectors of $\{f_i\}_i$. Then $Q=P^TP$ is the PSD matrix we seek. Conversely, given a PSD matrix $Q$ satisfying $f(x)=[x]_{M}^TQ[x]_M$, we can form a factorization $Q=P^TP$, and read off the coefficients of each polynomial $f_i$ from the rows of $P$. 
 \end{proof}
 
 Notice that  the relation $f=[x]_{M}^TQ[x]_M$ can be easily rewritten as a linear relation on $Q$ by matching coefficient of the left and right-hand-sides. The size of $Q$ is completely dictated by the number of monomials.
 
 More generally, instead of certifying whether a polynomial is SOS, we mught be interested in minimizing a linear functional over an affine slice of the SOS cone. More precisely, consider a problem of the form:
\begin{equation}\label{eqn:SOS}\begin{aligned}
\min_{u\in\R^m}&~ \sum_{i=1}^m w_iu_i\\
\textrm{s.t. } &f=g_0+\sum_{i=1}^m u_ig_i\\
&f\in \Sigma_{n,2d}
\end{aligned}
\end{equation}
where $u\in \R^m$ is the decision variable, $g_i\in \R[x]_{n,2d}$ are specified polynomials and $w\in\R^m$ is a fixed vector.
Clearly this problem can be converted to an \SDPp. The size of the
decision matrix $X$ is determined by the number of monomials. Parsing
algorithms attempt to choose a (small) set of monomials $M$ so that {\em
every} feasible $f$ for \eqref{eqn:SOS}  can be written as a sum of
squares over the monomial set $M$, thereby decreasing the size of the
\SDPp. Not surprisingly, some parsing strategies can be interpreted as facial reduction iterations on \eqref{eqn:SOS}.

We next outline such a strategy closely following \cite{perm_sos}. To this end, we must first explain which faces of the SOS cone $\Sigma_{n,2d}$ correspond to eliminating monomials. Indeed, there are faces of $\Sigma_{n,2d}$ that do not have such a description.

To answer this question, we will need extra notation. 
Henceforth, fix a set of monomials $M\subseteq\mathbb{N}^n$ and set $d:=\max\{\sum_{i=1}^n z_i: z\in M\}$ to be the maximal degree of monomials in $M$. Let 
$\Sigma(M)$ be the set of all polynomials that can be written as sums of squares over the monomial set $M$. Finally, set $M^+$ to be the set of points in $M$ that are not midpoints of any points in $M$, namely
$$M^+:=M\setminus\left\{\frac{\alpha+\beta}{2}: \alpha,\beta\in M \textrm{ and }\alpha\neq\beta\right\}.$$

Let us now look two types of faces that arise from elimination of monomials. 
\begin{theorem}[Type I face]
If equality, $$\conv(M)\cap \mathbb{N}^n=M$$ holds, then  $\Sigma(M)$ is a face of $\Sigma_{n,2d}$.
\end{theorem}
In other words, if the convex hull $\conv(M)$ contains no grid points other than those already in $M$, then $\Sigma(M)$ is a face of $\Sigma_{n,2d}$.

\begin{theorem}[Type II face]
If $\Sigma(M)$ is a face of $\Sigma_{n,2d}$, then $\Sigma(M\setminus \beta)$ is a face of $\Sigma_{n,2d}$ for any $\beta\in M^+$. 
\end{theorem}
Thus given a face $\Sigma(M)$, we can recursively make the face smaller by deleting any $\beta\in M^+$. 

Let us now turn to facial reduction. In the first step of facial reduction for \eqref{eqn:SOS}, we must find an exposing vector  $v\in\Sigma_{n,2d}^{*}$ that is orthogonal to all the affine constraints. Doing so in full generality is a difficult proposition. Instead, let us try to replace 
$\Sigma_{n,2d}^{*}$ by a polyhedral inner approximation. Then the search for $v$ is a linear program.
\begin{theorem}
The polyhedral set 
$$\Lambda= \Sigma(M)^\perp+\left\{\sum_{\alpha\in M^+} \lambda_{2\alpha}e_{2\alpha}: \lambda_{2\alpha}\geq 0\right\}$$
satisfies $\Lambda\subseteq \Sigma(M)^*.$
\end{theorem}
Thus if we can find $v\in \Lambda$ that is orthogonal to the affine constraints \eqref{eqn:SOS}, then we can use $v$ to expose a face of $\Sigma_{n,2d}$ containing the feasible region. Remarkably, this face can indeed be represented by eliminating monomials from $M$.

\begin{theorem}
Consider a vector $$v=p+\sum_{\alpha\in M^+} \lambda_{2\alpha}e_{2\alpha},$$ for some $p\in \Sigma(M)^\perp$ and nonnegative numbers $\lambda_{2\alpha}\geq 0$. Define the monomial set $\mathcal{I}:=\{\alpha\in M^+: \lambda_{2\alpha}>0\}$. Then the face 
$\Sigma(M)\cap v^{\perp}$ coincides with $\Sigma(M\setminus \mathcal{I})$.
\end{theorem}

Thus we can inductively use this procedure to eliminate monomials. At
the end, one would hope that we would be left with a small dimensional \SDP to solve. Promising numerical results and further explanations of methods of this type can be found in \cite{perm, perm_sos, coj_spars,waki_mur_sparse,WakiKimKojimaMura:06}.

\section{Commentary}
\subsection*{Quadratic assignment problem, \QAPp}
Many survey articles and books have appeared on the \QAPp,
e.g.,~\cite{prw:93,PardWolk:94,MR1490831,MR2267435}.
More recent work on implementation of \SDP relaxations include
\cite{MR3347884,OliveiraWolkXu:15,RongZhu:07}.
That the quadratic assignment problem is
NP-hard is shown in~\cite{SahGon:76}.
The elegant trace formulation we used was introduced in~\cite{Edw77}.

The classic \emph{Nugent test set} for QAP is given in
\cite{Nug:68}.\footnote{It is
maintained within QAPLIB~\cite{BurKarRen91} currently online.}
These problems have proven to be extremely hard to solve to optimality,
see e.g.,~\cite{MR1490831}.
The difficulty of these problems is illustrated in the fact that many of
them were not solved for $30$ odd years, see e.g.,~\cite{MR2004391}.

The semi-definite relaxation described here was introduced in
\cite{KaReWoZh:94}. It was derived by using the Lagrangian relaxation
after modelling the permutation matrix constraint by various
quadratic constraints. The
semi-definite relaxation is then the dual of the Lagrangian relaxation,
i.e.,~\emph{the dual of the dual}.
Application of \FR then results in the surprisingly simplified 
\textdef{gangster operator} formulation.

This gangster formulation along with symmetry for certain \QAP models is
exploited in \cite{MR3347884,MR2546331} to significantly increase the
size of \QAP problems that can be solved.
Other relaxations of QAP based on e.g.,~eigenvalue bounds are studied in 
e.g.,~\cite{FiBuRe:87,MR878775,AnsBrix:99}.

\subsection*{Graph partitioning, \GPp}
The \textdef{graph partitioning, \GPp}, problem is very similar to the
\QAP in that it involves a trace quadratic objective with a $0,1$ matrix
variable $X$, i.e. the matrix whose columns are the incidence vectors of
the sets for the partition. A similar successful \SDP relaxation can be
found~\cite{WoZh:96}. More recently, successful bounding results have
been found in~\cite{HaoWangPongWolk:14,HaoSun:16,MR3060945}.

\subsection*{Second lift of Max-Cut}
The second lifting from Section \ref{sect:secliftMC}
is derived in~\cite{Anjosthesis,AnWo:00,MR1846167} but in a different
way, i.e. using the nullspace of the barycenter approach.
The bounds found were extremely successful and, in fact, found the
optimal solution of the \MC in almost all but very special cases.
The algorithm used for the \SDP relaxation was the spectral bundle
approach~\cite{HelmbergRendl:97} and only problems of limited size could be
solved. More recently an ADMM approach was much more successful in
solving larger problems in \cite{TangWolk:17}.

\subsection*{Lifts of combinatorial problems}
The \SDP lifting of combinatorial regions described in
Section~\ref{sect:sdplifts} is standard; see for \cite{lev_book} for
many examples and references. The material on the minimal face of the \SDP lift follows~\cite{Tun:01}, though our explanation here is stated in dual terms, i.e. using exposing vectors.

\subsection*{Monomial elimination from SOS problems}
The topic of eliminating monomials from sum of squares problems has a rich history. The section in the current text follows entirely the exposition in   \cite{perm_sos}. The technique of solving linear programs in order to approximate exposing vectors was extensively studied in  \cite{perm}. Important earlier references on monomial elimination include \cite{coj_spars,waki_mur_sparse,WakiKimKojimaMura:06}. For an exposition of how to use SOS hierarchies to solve polynomial optimization problems see the monograph \cite{JL_book}.

\backmatter  
\begin{acknowledgements}
\addcontentsline{toc}{chapter}{Acknowledgements} %
We would like to thank  Jiyoung (Haesol)  Im for her helpful comments and
help with proofreading the manuscript. Research of the first author was partially supported by the AFOSR YIP award FA9550-15-1-0237. Research of the second author was supported by The Natural Sciences and Engineering Research Council of Canada.
\end{acknowledgements}

\chapter*{}
\printindex
\label{ind:index}
\addcontentsline{toc}{chapter}{Index}

\bibliographystyle{plain}

\begin{thebibliography}{100}

\bibitem{Alfak:00}
A.~Alfakih.
\newblock Graph rigidity via {E}uclidean distance matrices.
\newblock {\em Linear Algebra Appl.}, 310(1-3):49--165, 2000.

\bibitem{AlfakihAnjosKPW:08}
A.~Alfakih, M.F. Anjos, V.~Piccialli, and H.~Wolkowicz.
\newblock Euclidean distance matrices, semidefinite programming, and sensor
  network localization.
\newblock {\em Portug. Math.}, 68(1):53--102, 2011.

\bibitem{AlKaWo:97}
A.~Alfakih, A.~Khandani, and H.~Wolkowicz.
\newblock Solving {E}uclidean distance matrix completion problems via
  semidefinite programming.
\newblock {\em Comput. Optim. Appl.}, 12(1-3):13--30, 1999.
\newblock A tribute to Olvi Mangasarian.

\bibitem{AlWo:99}
A.~Alfakih and H.~Wolkowicz.
\newblock Matrix completion problems.
\newblock In {\em Handbook of semidefinite programming}, volume~27 of {\em
  Internat. Ser. Oper. Res. Management Sci.}, pages 533--545. Kluwer Acad.
  Publ., Boston, MA, 2000.

\bibitem{MR2209243}
A.Y. Alfakih.
\newblock A remark on the faces of the cone of {E}uclidean distance matrices.
\newblock {\em Linear Algebra Appl.}, 414(1):266--270, 2006.

\bibitem{Alipanahi:2012}
B.~Alipanahi, N.~Krislock, A.~Ghodsi, H.~Wolkowicz, L.~Donaldson, and M.~Li.
\newblock Determining protein structures from {NOESY} distance constraints by
  semidefinite programming.
\newblock {\em J. Comput. Biol.}, 20(4):296--310, 2013.

\bibitem{MR3042029}
B.~Alipanahi, N.~Krislock, A.~Ghodsi, H.~Wolkowicz, L.~Donaldson, and M.~Li.
\newblock Determining protein structures from {NOESY} distance constraints by
  semidefinite programming.
\newblock {\em J. Comput. Biol.}, 20(4):296--310, 2013.

\bibitem{MR1748773}
E.D. Andersen and K.D. Andersen.
\newblock The {M}osek interior point optimizer for linear programming: an
  implementation of the homogeneous algorithm.
\newblock In {\em High performance optimization}, volume~33 of {\em Appl.
  Optim.}, pages 197--232. Kluwer Acad. Publ., Dordrecht, 2000.

\bibitem{AnjosLasserre:11}
A.F. Anjos and J.B. Lasserre, editors.
\newblock {\em Handbook on Semidefinite, Conic and Polynomial Optimization}.
\newblock International Series in Operations Research \& Management Science.
  Springer-Verlag, 2011.

\bibitem{MR1846167}
M.~F. Anjos and H.~Wolkowicz.
\newblock Strengthened semidefinite programming relaxations for the max-cut
  problem.
\newblock In {\em Advances in convex analysis and global optimization
  (Pythagorion, 2000)}, volume~54 of {\em Nonconvex Optim. Appl.}, pages
  409--420. Kluwer Acad. Publ., Dordrecht, 2001.

\bibitem{Anjosthesis}
M.F. Anjos.
\newblock {\em New Convex Relaxations for the Maximum Cut and VLSI Layout
  Problems}.
\newblock PhD thesis, University of Waterloo, 2001.

\bibitem{AnWo:00}
M.F. Anjos and H.~Wolkowicz.
\newblock Strengthened semidefinite relaxations via a second lifting for the
  {M}ax-{C}ut problem.
\newblock {\em Discrete Appl. Math.}, 119(1-2):79--106, 2002.
\newblock Foundations of heuristics in combinatorial optimization.

\bibitem{MR2004391}
K.M. Anstreicher.
\newblock Recent advances in the solution of quadratic assignment problems.
\newblock {\em Math. Program.}, 97(1-2, Ser. B):27--42, 2003.
\newblock ISMP, 2003 (Copenhagen).

\bibitem{AnsBrix:99}
K.M. Anstreicher and N.W. Brixius.
\newblock A new bound for the quadratic assignment problem based on convex
  quadratic programming.
\newblock {\em Math. Program.}, 89(3, Ser. A):341--357, 2001.

\bibitem{MR1321802}
M.~Bakonyi and C.R. Johnson.
\newblock The {E}uclidean distance matrix completion problem.
\newblock {\em SIAM J. Matrix Anal. Appl.}, 16(2):646--654, 1995.

\bibitem{MR2807419}
M.~Bakonyi and H.J. Woerdeman.
\newblock {\em Matrix completions, moments, and sums of {H}ermitian squares}.
\newblock Princeton University Press, Princeton, NJ, 2011.

\bibitem{barv_book}
A.~Barvinok.
\newblock {\em A course in convexity}, volume~54 of {\em Graduate Studies in
  Mathematics}.
\newblock American Mathematical Society, Providence, RI, 2002.

\bibitem{MR0249108}
A.~Ben-Israel.
\newblock Theorems of the alternative for complex linear inequalities.
\newblock {\em Israel J. Math.}, 7:129--136, 1969.

\bibitem{MR607673}
A.~Ben-Israel, A.~Ben-Tal, and S.~Zlobec.
\newblock {\em Optimality in nonlinear programming: a feasible directions
  approach}.
\newblock John Wiley \& Sons, Inc., New York, 1981.
\newblock A Wiley-Interscience Publication.

\bibitem{MR1857264}
A.~Ben-Tal and A.S. Nemirovski.
\newblock {\em Lectures on modern convex optimization}.
\newblock MPS/SIAM Series on Optimization. Society for Industrial and Applied
  Mathematics (SIAM), Philadelphia, PA, 2001.
\newblock Analysis, algorithms, and engineering applications.

\bibitem{BoLe:00}
J.M. Borwein and A.S. Lewis.
\newblock {\em Convex analysis and nonlinear optimization}.
\newblock Springer-Verlag, New York, 2000.
\newblock Theory and examples.

\bibitem{bw2}
J.M. Borwein and H.~Wolkowicz.
\newblock Characterization of optimality for the abstract convex program with
  finite-dimensional range.
\newblock {\em J. Austral. Math. Soc. Ser. A}, 30(4):390--411, 1980/81.

\bibitem{bw1}
J.M. Borwein and H.~Wolkowicz.
\newblock Facial reduction for a cone-convex programming problem.
\newblock {\em J. Austral. Math. Soc. Ser. A}, 30(3):369--380, 1980/81.

\bibitem{bw3}
J.M. Borwein and H.~Wolkowicz.
\newblock Regularizing the abstract convex program.
\newblock {\em J. Math. Anal. Appl.}, 83(2):495--530, 1981.

\bibitem{bw4}
J.M. Borwein and H.~Wolkowicz.
\newblock Characterizations of optimality without constraint qualification for
  the abstract convex program.
\newblock {\em Math. Programming Stud.}, 19:77--100, 1982.
\newblock Optimality and stability in mathematical programming.

\bibitem{BurKarRen91}
R.E. Burkard, S.~Karisch, and F.~Rendl.
\newblock {QAPLIB} -- a quadratic assignment problem library.
\newblock {\em European J. Oper. Res.}, 55:115--119, 1991.
\newblock anjos.mgi.polymtl.ca/qaplib/.

\bibitem{MR2565240}
E.J. Cand{\`e}s and B.~Recht.
\newblock Exact matrix completion via convex optimization.
\newblock {\em Found. Comput. Math.}, 9(6):717--772, 2009.

\bibitem{MR1490831}
E.~{\c{C}}ela.
\newblock {\em The quadratic assignment problem}, volume~1 of {\em
  Combinatorial Optimization}.
\newblock Kluwer Academic Publishers, Dordrecht, 1998.
\newblock Theory and algorithms.

\bibitem{Cheung:2013}
Y.-L. Cheung.
\newblock {\em Preprocessing and Reduction for Semidefinite Programming via
  Facial Reduction: Theory and Practice}.
\newblock PhD thesis, University of Waterloo, 2013.

\bibitem{ScTuWonumeric:07}
Y-L. Cheung, S.~Schurr, and H.~Wolkowicz.
\newblock Preprocessing and regularization for degenerate semidefinite
  programs.
\newblock In D.H. Bailey, H.H. Bauschke, P.~Borwein, F.~Garvan, M.~Thera,
  J.~Vanderwerff, and H.~Wolkowicz, editors, {\em Computational and
  {A}nalytical {M}athematics, {I}n {H}onor of {J}onathan {B}orwein's 60th
  {B}irthday}, volume~50 of {\em Springer Proceedings in Mathematics \&
  Statistics}, pages 225--276. Springer, 2013.

\bibitem{MR3341582}
R.~Connelly and S.J. Gortler.
\newblock Iterative universal rigidity.
\newblock {\em Discrete Comput. Geom.}, 53(4):847--877, 2015.

\bibitem{MR2487961}
G.~Cravo.
\newblock Recent progress on matrix completion problems.
\newblock {\em J. Math. Sci. Adv. Appl.}, 1(1):69--90, 2008.

\bibitem{MR2508310}
G.~Cravo.
\newblock Matrix completion problems.
\newblock {\em Linear Algebra Appl.}, 430(8-9):2511--2540, 2009.

\bibitem{MR975025}
G.M. Crippen and T.F. Havel.
\newblock {\em Distance geometry and molecular conformation}, volume~15 of {\em
  Chemometrics Series}.
\newblock Research Studies Press Ltd., Chichester, 1988.

\bibitem{Dant:63}
G.~Dantzig.
\newblock {\em Linear Programming and Extensions}.
\newblock Princeton University Press, Princeton, New Jersey, 1963.

\bibitem{Dant:90}
G.~Dantzig.
\newblock Linear programming.
\newblock In {\em History of Mathematical Programming: A Collection of Personal
  Reminiscences}. CWI North-Holland, Amsterdam, 1991.

\bibitem{dattorroonline:05}
J.~Dattorro.
\newblock Convex optimization \& {E}uclidean distance geometry.
\newblock \url{https://ccrma.stanford.edu/~dattorro/mybook.html}, Meboo
  Publishing, 2005.

\bibitem{MR3060945}
E.~de~Klerk, M.~E.-Nagy, and R.~Sotirov.
\newblock On semidefinite programming bounds for graph bandwidth.
\newblock {\em Optim. Methods Softw.}, 28(3):485--500, 2013.

\bibitem{int:deklerk7}
E.~de~Klerk, C.~Roos, and T.~Terlaky.
\newblock Infeasible--start semidefinite programming algorithms via self--dual
  embeddings.
\newblock In {\em Topics in Semidefinite and Interior-Point Methods}, volume~18
  of {\em The Fields Institute for Research in Mathematical Sciences,
  Communications Series}, pages 215--236. American Mathematical Society, 1998.

\bibitem{MR2546331}
E.~de~Klerk and R.~Sotirov.
\newblock Exploiting group symmetry in semidefinite programming relaxations of
  the quadratic assignment problem.
\newblock {\em Math. Program.}, 122(2, Ser. A):225--246, 2010.

\bibitem{MR3347884}
E.~de~Klerk, R.~Sotirov, and U.~Truetsch.
\newblock A new semidefinite programming relaxation for the quadratic
  assignment problem and its computational perspectives.
\newblock {\em INFORMS J. Comput.}, 27(2):378--391, 2015.

\bibitem{DiKrQiWo:08}
Y.~Ding, N.~Krislock, J.~Qian, and H.~Wolkowicz.
\newblock Sensor network localization, {E}uclidean distance matrix completions,
  and graph realization.
\newblock {\em Optim. Eng.}, 11(1):45--66, 2010.

\bibitem{DBLP:journals/corr/DokmanicPRV15}
I.~Dokmanic, R.~Parhizkar, J.~Ranieri, and M.~Vetterli.
\newblock Euclidean distance matrices: {A} short walk through theory,
  algorithms and applications.
\newblock {\em CoRR}, abs/1502.07541, 2015.

\bibitem{radius}
A.L. Dontchev, A.S. Lewis, and R.T. Rockafellar.
\newblock The radius of metric regularity.
\newblock {\em Trans. Amer. Math. Soc.}, 355(2):493--517 (electronic), 2003.

\bibitem{MR89c:51020}
A.W.M. Dress and T.F. Havel.
\newblock The fundamental theory of distance geometry.
\newblock In {\em Computer aided geometric reasoning, Vol.\ I, II
  (Sophia-Antipolis, 1987)}, pages 127--169. INRIA, Rocquencourt, 1987.

\bibitem{ChDrWo:14}
D.~Drusvyatskiy, N.~Krislock, Y-L.~Cheung Voronin, and H.~Wolkowicz.
\newblock Noisy sensor network localization: robust facial reduction and the
  {P}areto frontier.
\newblock {\em SIAM Journal on Optimization}, pages 1--30, 2017.
\newblock arXiv:1410.6852, 20 pages, to appear.

\bibitem{DrPaWo:14}
D.~Drusvyatskiy, G.~Pataki, and H.~Wolkowicz.
\newblock Coordinate shadows of semidefinite and {E}uclidean distance matrices.
\newblock {\em SIAM J. Optim.}, 25(2):1160--1178, 2015.

\bibitem{DymGoh:81}
H.~Dym and I.~Gohberg.
\newblock Extensions of band matrices with band inverses.
\newblock {\em Linear Algebra Appl.}, 36:1--24, 1981.

\bibitem{Edw77}
C.S. Edwards.
\newblock The derivation of a greedy approximator for the koopmans-beckmann
  quadratic assignment problem.
\newblock {\em Proc. CP77 Combinatorial Prog. Conf.}, pages 55--86, 1977.

\bibitem{Fazel:02}
M.~Fazel.
\newblock {\em Matrix Rank Minimization with Applications}.
\newblock PhD thesis, Stanford University, Stanford, CA, 2001.

\bibitem{fazelhindiboyd:01}
M.~Fazel, H.~Hindi, and S.P. Boyd.
\newblock A rank minimization heuristic with application to minimum order
  system approximation.
\newblock In {\em Proceedings American Control Conference}, pages 4734--4739,
  2001.

\bibitem{FiBuRe:87}
G.~Finke, R.E. Burkard, and F.~Rendl.
\newblock Quadratic assignment problems.
\newblock {\em Ann. Discrete Math.}, 31:61--82, 1987.

\bibitem{MR878775}
G.~Finke, R.E. Burkard, and F.~Rendl.
\newblock Quadratic assignment problems.
\newblock In {\em Surveys in combinatorial optimization (Rio de Janeiro,
  1985)}, volume 132 of {\em North-Holland Math. Stud.}, pages 61--82.
  North-Holland, Amsterdam, 1987.

\bibitem{MR0441352}
J.~Gauvin and J.W. Tolle.
\newblock Differential stability in nonlinear programming.
\newblock {\em SIAM J. Control Optimization}, 15(2):294--311, 1977.

\bibitem{GortlerThurston:14}
S.J. Gortler and D.P. Thurston.
\newblock Characterizing the universal rigidity of generic frameworks.
\newblock {\em Discrete Comput Geom}, 51:1017--1036, 2014.

\bibitem{MR3216675}
S.J. Gortler and D.P. Thurston.
\newblock Characterizing the universal rigidity of generic frameworks.
\newblock {\em Discrete Comput. Geom.}, 51(4):1017--1036, 2014.

\bibitem{MR45:6415}
F.J. Gould and J.W. Tolle.
\newblock Geometry of optimality conditions and constraint qualifications.
\newblock {\em Math. Programming}, 2(1):1--18, 1972.

\bibitem{MR83a:51024}
J.C. Gower.
\newblock {E}uclidean distance geometry.
\newblock {\em Math. Sci.}, 7(1):1--14, 1982.

\bibitem{MR86j:62133}
J.C. Gower.
\newblock Properties of {E}uclidean and non-{E}uclidean distance matrices.
\newblock {\em Linear Algebra Appl.}, 67:81--97, 1985.

\bibitem{MR1747519}
H.~Grassmann.
\newblock {\em Extension theory}, volume~19 of {\em History of Mathematics}.
\newblock American Mathematical Society, Providence, RI, 2000.
\newblock Translated from the 1896 German original and with a foreword,
  editorial notes and supplementary notes by Lloyd C. Kannenberg.

\bibitem{GrJoSaWo:84}
B.~Grone, C.R. Johnson, E.~Marques~de Sa, and H.~Wolkowicz.
\newblock Positive definite completions of partial {H}ermitian matrices.
\newblock {\em Linear Algebra Appl.}, 58:109--124, 1984.

\bibitem{Guig:69}
M.~Guignard.
\newblock Generalized {K}uhn-{T}ucker conditions for mathematical programming
  problems in a {B}anach space.
\newblock {\em SIAM J. of Control}, 7(2):232--241, 1969.

\bibitem{MR2014037}
K.J. Harrison.
\newblock Matrix completions and chordal graphs.
\newblock {\em Acta Math. Sin. (Engl. Ser.)}, 19(3):577--590, 2003.
\newblock International Workshop on Operator Algebra and Operator Theory
  (Linfen, 2001).

\bibitem{MR783122}
T.F. Havel, I.D. Kuntz, and B.~Crippen.
\newblock Errata: ``{T}he theory and practice of distance geometry'' [{B}ull.
  {M}ath. {B}iol. {\bf 45} (1983), no. 5, 665--720; {MR}0718540 (84k:92037)].
\newblock {\em Bull. Math. Biol.}, 47(1):157, 1985.

\bibitem{HelmbergRendl:97}
C.~Helmberg and F.~Rendl.
\newblock A spectral bundle method for semidefinite programming.
\newblock {\em SIAM J. Optim.}, 10(3):673 -- 696, 2000.

\bibitem{HuangWolkXYe:16}
S.~Huang and H.~Wolkowicz.
\newblock Low-rank matrix completion using nuclear norm\ with facial reduction.
\newblock {\em J. Global Optim.}, 2017.
\newblock 23 pages, to appear.

\bibitem{Huang04preprocessingand}
X.~Huang.
\newblock {\em Preprocessing and postprocessing in linear optimization}.
\newblock PhD thesis, McMaster University, 2004.

\bibitem{MR1059486}
C.R. Johnson.
\newblock Matrix completion problems: a survey.
\newblock In {\em Matrix theory and applications (Phoenix, AZ, 1989)}, pages
  171--198. Amer. Math. Soc., Providence, RI, 1990.

\bibitem{joh:90}
C.R. Johnson.
\newblock Matrix completion problems: a survey.
\newblock {\em Proceedings of Symposium in Applied Mathematics}, 40:171--198,
  1990.

\bibitem{MR1059481}
C.R. Johnson, editor.
\newblock {\em Matrix theory and applications}, volume~40 of {\em Proceedings
  of Symposia in Applied Mathematics}.
\newblock American Mathematical Society, Providence, RI, 1990.
\newblock Lecture notes prepared for the American Mathematical Society Short
  Course held in Phoenix, Arizona, January 10--11, 1989, AMS Short Course
  Lecture Notes.

\bibitem{MR1044604}
C.R. Johnson.
\newblock Positive definite completions: a guide to selected literature.
\newblock In {\em Signal processing, Part I}, volume~22 of {\em IMA Vol. Math.
  Appl.}, pages 169--188. Springer, New York, 1990.

\bibitem{JoKrWo:95}
C.R. Johnson, B.~Kroschel, and H.~Wolkowicz.
\newblock An interior-point method for approximate positive semidefinite
  completions.
\newblock {\em Comput. Optim. Appl.}, 9(2):175--190, 1998.

\bibitem{MR1340702}
C.R. Johnson and P.~Tarazaga.
\newblock Connections between the real positive semidefinite and distance
  matrix completion problems.
\newblock {\em Linear Algebra Appl.}, 223/224:375--391, 1995.
\newblock Special issue honoring Miroslav Fiedler and Vlastimil Pt{\'a}k.

\bibitem{coj_spars}
M.~Kojima, S.~Kim, and H.~Waki.
\newblock Sparsity in sums of squares of polynomials.
\newblock {\em Math. Program.}, 103(1, Ser. A):45--62, 2005.

\bibitem{krislock:2010}
N.~Krislock.
\newblock {\em Semidefinite Facial Reduction for Low-Rank Euclidean Distance
  Matrix Completion}.
\newblock PhD thesis, University of Waterloo, 2010.

\bibitem{kriswolk:09}
N.~Krislock and H.~Wolkowicz.
\newblock Explicit sensor network localization using semidefinite
  representations and facial reductions.
\newblock {\em SIAM Journal on Optimization}, 20(5):2679--2708, 2010.

\bibitem{KrislockWolk:10}
N.~Krislock and H.~Wolkowicz.
\newblock Euclidean distance matrices and applications.
\newblock In {\em Handbook on Semidefinite, Cone and Polynomial Optimization},
  number 2009-06 in International Series in Operations Research \& Management
  Science, pages 879--914. Springer-Verlag, 2011.

\bibitem{Lasserre01globaloptimization}
J.B. Lasserre.
\newblock Global optimization with polynomials and the problem of moments.
\newblock {\em SIAM J. Optim.}, 11(3):796--817 (electronic), 2000/01.

\bibitem{JL_book}
J.B. Lasserre.
\newblock {\em Moments, positive polynomials and their applications}, volume~1
  of {\em Imperial College Press Optimization Series}.
\newblock Imperial College Press, London, 2010.

\bibitem{Laur:97b}
M.~Laurent.
\newblock A connection between positive semidefinite and {E}uclidean distance
  matrix completion problems.
\newblock {\em Linear Algebra Appl.}, 273:9--22, 1998.

\bibitem{MR1607310}
M.~Laurent.
\newblock A tour d'horizon on positive semidefinite and {E}uclidean distance
  matrix completion problems.
\newblock In {\em Topics in semidefinite and interior-point methods ({T}oronto,
  {ON}, 1996)}, volume~18 of {\em Fields Inst. Commun.}, pages 51--76. Amer.
  Math. Soc., Providence, RI, 1998.

\bibitem{MR99c:05135}
M.~Laurent.
\newblock A tour d'horizon on positive semidefinite and {E}uclidean distance
  matrix completion problems.
\newblock In {\em Topics in semidefinite and interior-point methods (Toronto,
  ON, 1996)}, pages 51--76. Amer. Math. Soc., Providence, RI, 1998.

\bibitem{Laurent:00}
M.~Laurent.
\newblock Polynomial instances of the positive semidefinite and {E}uclidean
  distance matrix completion problems.
\newblock {\em SIAM J. Matrix Anal. Appl.}, 22:874--894, 2000.

\bibitem{Floudas:2001aa}
M.~Laurent.
\newblock Matrix completion problems.
\newblock In {\em Encyclopedia of Optimization}, pages 1311--1319. Springer US,
  2001.

\bibitem{laurentvallentin:16}
M.~Laurent and F.~Vallentin.
\newblock {Semidefinite {O}ptimization}.
\newblock \url{http://homepages.cwi.nl/~monique/master_SDP_2016.pdf}, 2016.
\newblock [Online; accessed 25-April-2016].

\bibitem{2016arXiv161000652L}
L.~{Liberti} and C.~{Lavor}.
\newblock {Open research areas in distance geometry}.
\newblock {\em ArXiv e-prints}, October 2016.

\bibitem{surv}
L.~Liberti, C.~Lavor, N.~Maculan, and A.~Mucherino.
\newblock Euclidean distance geometry and applications.
\newblock {\em SIAM Review}, 56(1):3--69, 2014.

\bibitem{GPatakiLiu:15}
M.~Liu and G.~Pataki.
\newblock Exact duals and short certificates of infeasibility and weak
  infeasibility in conic linear programming.
\newblock Technical report, Department of Statistics and Operations Research,
  University of North Carolina at Chapel Hill, 2015.

\bibitem{MR2267435}
E.M. Loiola, Nair~M. Maia~d A., P.O. Boaventura-Netto, P.~Hahn, and T.~Querido.
\newblock A survey for the quadratic assignment problem.
\newblock {\em European J. Oper. Res.}, 176(2):657--690, 2007.

\bibitem{LuoStZh:97}
Z-Q. Luo, J.F. Sturm, and S.~Zhang.
\newblock Duality results for conic convex programming.
\newblock Technical Report Report 9719/A, April, Erasmus University Rotterdam,
  Econometric Institute EUR, P.O. Box 1738, 3000 DR, The Netherlands, 1997.

\bibitem{MR34:7263}
O.~L. Mangasarian and S.~Fromovitz.
\newblock The {F}ritz {J}ohn necessary optimality conditions in the presence of
  equality and inequality constraints.
\newblock {\em J. Math. Anal. Appl.}, 17:37--47, 1967.

\bibitem{Mang:69}
O.L. Mangasarian.
\newblock {\em Nonlinear Programming}.
\newblock McGraw-Hill, New York, NY, 1969.

\bibitem{Massam:79}
H.~Massam.
\newblock Optimality conditions for a cone-convex programming problem.
\newblock {\em J. Austral. Math. Soc. Ser. A}, 27(2):141--162, 1979.

\bibitem{MR663327}
H.~Massam.
\newblock Optimality conditions using sum-convex approximations.
\newblock {\em J. Optim. Theory Appl.}, 35(4):475--495, 1981.

\bibitem{MR0247866}
J.~Maybee and J.~Quirk.
\newblock Qualitative problems in matrix theory.
\newblock {\em SIAM Rev.}, 11:30--51, 1969.

\bibitem{springerlink:10.1007/s00291-003-0130-x}
C.~Mészáros and U.H. Suhl.
\newblock Advanced preprocessing techniques for linear and quadratic
  programming.
\newblock {\em OR Spectrum}, 25:575--595, 2003.
\newblock 10.1007/s00291-003-0130-x.

\bibitem{MR3059590}
A.~Mucherino, C.~Lavor, L.~Liberti, and N.~Maculan, editors.
\newblock {\em Distance geometry}.
\newblock Springer, New York, 2013.
\newblock Theory, methods, and applications.

\bibitem{int:Nesterov5}
Y.E. Nesterov and A.S. Nemirovski.
\newblock {\em Interior-point polynomial algorithms in convex programming},
  volume~13 of {\em SIAM Studies in Applied Mathematics}.
\newblock Society for Industrial and Applied Mathematics (SIAM), Philadelphia,
  PA, 1994.

\bibitem{Nug:68}
C.E. Nugent, T.E. Vollman, and J.~Ruml.
\newblock An experimental comparison of techniques for the assignment of
  facilities to locations.
\newblock {\em Operations Research}, 16:150--173, 1968.

\bibitem{OliveiraWolkXu:15}
D.E. Oliveira, H.~Wolkowicz, and Y.~Xu.
\newblock {ADMM} for the {SDP} relaxation of the {QAP}.
\newblock Technical report, University of Waterloo, Waterloo, Ontario, 2015.
\newblock arXiv:1512.05448, under revision Oct. 2016, 12 pages.

\bibitem{prw:93}
P.~Pardalos, F.~Rendl, and H.~Wolkowicz.
\newblock The quadratic assignment problem: a survey and recent developments.
\newblock In P.M. Pardalos and H.~Wolkowicz, editors, {\em Quadratic assignment
  and related problems (New Brunswick, NJ, 1993)}, pages 1--42. Amer. Math.
  Soc., Providence, RI, 1994.

\bibitem{PardWolk:94}
P.~Pardalos and H.~Wolkowicz, editors.
\newblock {\em Quadratic assignment and related problems}.
\newblock American Mathematical Society, Providence, RI, 1994.
\newblock Papers from the workshop held at Rutgers University, New Brunswick,
  New Jersey, May 20--21, 1993.

\bibitem{PP_hiarchy}
P.A. Parrilo.
\newblock Semidefinite programming relaxations for semialgebraic problems.
\newblock {\em Math. Program.}, 96(2, Ser. B):293--320, 2003.
\newblock Algebraic and geometric methods in discrete optimization.

\bibitem{P00}
G.~Pataki.
\newblock A simple derivation of a facial reduction algorithm and extended dual
  systems.
\newblock Technical report, Columbia University, New York, 2000.

\bibitem{Gpat:11}
G.~{Pataki}.
\newblock {Bad semidefinite programs: they all look the same}.
\newblock {\em ArXiv e-prints}, December 2011.

\bibitem{Pnice}
G.~Pataki.
\newblock On the connection of facially exposed and nice cones.
\newblock {\em J. Math. Anal. Appl.}, 400(1):211--221, 2013.

\bibitem{MR3108446}
G.~Pataki.
\newblock Strong duality in conic linear programming: facial reduction and
  extended duals.
\newblock In David Bailey, Heinz~H. Bauschke, Frank Garvan, Michel Thera,
  Jon~D. Vanderwerff, and Henry Wolkowicz, editors, {\em Computational and
  analytical mathematics}, volume~50 of {\em Springer Proc. Math. Stat.}, pages
  613--634. Springer, New York, 2013.

\bibitem{MR1757553}
J.~Pe{\~n}a and J.~Renegar.
\newblock Computing approximate solutions for convex conic systems of
  constraints.
\newblock {\em Math. Program.}, 87(3, Ser. A):351--383, 2000.

\bibitem{permfribergandersen}
F.~Permenter, H.~Friberg, and E.~Andersen.
\newblock Solving conic optimization problems via self-dual embedding and
  facial reduction: a unified approach.
\newblock Technical report, MIT, Boston, MA, 2015.

\bibitem{perm}
F.~Permenter and P.~Parrilo.
\newblock Partial facial reduction: simplified, equivalent {SDP}s via
  approximations of the {PSD} cone.
\newblock Technical Report Preprint arXiv:1408.4685, MIT, Boston, MA, 2014.

\bibitem{perm_sos}
F.~Permenter and P.~A. Parrilo.
\newblock Basis selection for sos programs via facial reduction and polyhedral
  approximations.
\newblock In {\em 53rd IEEE Conference on Decision and Control}, pages
  6615--6620, Dec 2014.

\bibitem{2016arXiv160802090P}
F.~{Permenter} and P.~A. {Parrilo}.
\newblock {Dimension reduction for semidefinite programs via Jordan algebras}.
\newblock {\em ArXiv e-prints}, August 2016.

\bibitem{HaoWangPongWolk:14}
T.K. Pong, H.~Sun, N.~Wang, and H.~Wolkowicz.
\newblock Eigenvalue, quadratic programming, and semidefinite programming
  relaxations for a cut minimization problem.
\newblock {\em Comput. Optim. Appl.}, 63(2):333--364, 2016.

\bibitem{Ram:93}
M.V. Ramana.
\newblock {\em An Algorithmic Analysis of Multiquadratic and Semidefinite
  Programming Problems}.
\newblock PhD thesis, Johns Hopkins University, Baltimore, Md, 1993.

\bibitem{Ram:95}
M.V. Ramana.
\newblock An exact duality theory for semidefinite programming and its
  complexity implications.
\newblock {\em Math. Programming}, 77(2):129--162, 1997.

\bibitem{RaTuWo:95}
M.V. Ramana, L.~Tun{\c{c}}el, and H.~Wolkowicz.
\newblock Strong duality for semidefinite programming.
\newblock {\em SIAM J. Optim.}, 7(3):641--662, 1997.

\bibitem{Rechtparrilofazel}
B.~Recht, M.~Fazel, and P.~Parrilo.
\newblock Guaranteed minimum-rank solutions of linear matrix equations via
  nuclear norm minimization.
\newblock {\em SIAM Rev.}, 52(3):471--501, 2010.

\bibitem{Ren:94}
J.~Renegar.
\newblock Some perturbation theory for linear programming.
\newblock {\em Math. Programming}, 65(1, Ser. A):73--91, 1994.

\bibitem{MR96c:90048}
J.~Renegar.
\newblock Incorporating condition measures into the complexity theory of linear
  programming.
\newblock {\em SIAM J. Optim.}, 5(3):506--524, 1995.

\bibitem{Ren:93}
J.~Renegar.
\newblock Linear programming, complexity theory and elementary functional
  analysis.
\newblock {\em Math. Programming}, 70(3, Ser. A):279--351, 1995.

\bibitem{rob75a}
S.M. Robinson.
\newblock Stability theorems for systems of inequalities, part i: linear
  systems.
\newblock {\em SIAM J. Numerical Analysis}, 12:754--769, 1975.

\bibitem{rob76a}
S.M. Robinson.
\newblock Stability theorems for systems of inequalities, part ii:
  differentiable nonlinear systems.
\newblock {\em SIAM J. Numerical Analysis}, 13:497--513, 1976.

\bibitem{con:70}
R.T. Rockafellar.
\newblock {\em Convex analysis}.
\newblock Princeton Mathematical Series, No. 28. Princeton University Press,
  Princeton, N.J., 1970.

\bibitem{poly_solve_chord}
B.~Rosgen and L.~Stewart.
\newblock Complexity results on graphs with few cliques.
\newblock {\em Discrete Math. Theor. Comput. Sci.}, 9(1):127--135 (electronic),
  2007.

\bibitem{SahGon:76}
S.~Sahni and T.~Gonzales.
\newblock P-complete approximation problems.
\newblock {\em Journal of ACM}, 23:555--565, 1976.

\bibitem{face_sing}
A.~Singer.
\newblock A remark on global positioning from local distances.
\newblock {\em Proc. Natl. Acad. Sci. USA}, 105(28):9507--9511, 2008.

\bibitem{Sturmphd97}
F.J. Sturm.
\newblock {\em Primal-dual interior point approach Semidefinite programming}.
\newblock PhD thesis, Erasmus University, Rotterdam, Netherlands, 1997.

\bibitem{MR1778433}
J.F. Sturm.
\newblock Using {S}e{D}u{M}i 1.02, a {MATLAB} toolbox for optimization over
  symmetric cones.
\newblock {\em Optim. Methods Softw.}, 11/12(1-4):625--653, 1999.
\newblock Interior point methods.

\bibitem{S98lmi}
J.F. Sturm.
\newblock Error bounds for linear matrix inequalities.
\newblock {\em SIAM J. Optim.}, 10(4):1228--1248 (electronic), 2000.

\bibitem{HaoSun:16}
H.~Sun.
\newblock Admm for sdp relaxation of gp.
\newblock Master's thesis, University of Waterloo, 2016.

\bibitem{TangWolk:17}
Z.~Tang and H.~Wolkowicz.
\newblock {ADMM} for the second lifting {SDP} relaxation of {MC}.
\newblock Technical report, University of Waterloo, Waterloo, Ontario, 2017.
\newblock in progress.

\bibitem{tanigawa}
S.~Tanigawa.
\newblock Singularity degree of the positive semidefinite matrix completion
  problem.
\newblock Technical Report arXiv:1603.09586, Research Institute for
  Mathematical Sciences, Kyoto University, Sakyo-ku, Kyoto 606-8502, Japan,
  2016.

\bibitem{MR2166851}
P.~Tarazaga.
\newblock Faces of the cone of {E}uclidean distance matrices:
  characterizations, structure and induced geometry.
\newblock {\em Linear Algebra Appl.}, 408:1--13, 2005.

\bibitem{MR1366579}
P.~Tarazaga, T.L. Hayden, and J.~Wells.
\newblock Circum-{E}uclidean distance matrices and faces.
\newblock {\em Linear Algebra Appl.}, 232:77--96, 1996.

\bibitem{Tun:01}
L.~Tun{\c{c}}el.
\newblock On the {S}later condition for the {SDP} relaxations of nonconvex
  sets.
\newblock {\em Oper. Res. Lett.}, 29(4):181--186, 2001.

\bibitem{MR2724357}
L.~Tun{\c{c}}el.
\newblock {\em Polyhedral and Semidefinite Programming Methods in Combinatorial
  Optimization}, volume~27 of {\em Fields Institute Monographs}.
\newblock American Mathematical Society, Providence, RI, 2010.

\bibitem{lev_book}
L.~Tun{\c{c}}el.
\newblock {\em Polyhedral and semidefinite programming methods in combinatorial
  optimization}, volume~27 of {\em Fields Institute Monographs}.
\newblock American Mathematical Society, Providence, RI; Fields Institute for
  Research in Mathematical Sciences, Toronto, ON, 2010.

\bibitem{ScTuWominimal:07}
L.~Tun{\c{c}}el and H.~Wolkowicz.
\newblock Strong duality and minimal representations for cone optimization.
\newblock {\em Comput. Optim. Appl.}, 53(2):619--648, 2012.

\bibitem{VandenbergeAndersen:15}
L.~Vandenberghe and M.S. Andersen.
\newblock Chordal graphs and semidefinite optimization.
\newblock {\em Found. Trends Opt.}, 1(4):241--433, 2015.

\bibitem{WakiKimKojimaMura:06}
H.~Waki, S.~Kim, M.~Kojima, and M.~Muramatsu.
\newblock Sums of squares and semidefinite program relaxations for polynomial
  optimization problems with structured sparsity.
\newblock {\em SIAM J. Optim.}, 17(1):218--242, 2006.

\bibitem{waki_mur_sparse}
H.~Waki and M.~Muramatsu.
\newblock A facial reduction algorithm for finding sparse {SOS}
  representations.
\newblock {\em Oper. Res. Lett.}, 38(5):361--365, 2010.

\bibitem{MR3063940}
H.~Waki and M.~Muramatsu.
\newblock Facial reduction algorithms for conic optimization problems.
\newblock {\em J. Optim. Theory Appl.}, 158(1):188--215, 2013.

\bibitem{Whitney}
H.~Whitney.
\newblock Analytic extensions of differentiable functions defined in closed
  sets.
\newblock {\em Trans. Amer. Math. Soc.}, 36(1):63--89, 1934.

\bibitem{w8}
H.~Wolkowicz.
\newblock Geometry of optimality conditions and constraint qualifications: the
  convex case.
\newblock {\em Math. Programming}, 19(1):32--60, 1980.

\bibitem{w11}
H.~Wolkowicz.
\newblock Some applications of optimization in matrix theory.
\newblock {\em Linear Algebra Appl.}, 40:101--118, 1981.

\bibitem{SaVaWo:97}
H.~Wolkowicz, R.~Saigal, and L.~Vandenberghe, editors.
\newblock {\em Handbook of semidefinite programming}.
\newblock International Series in Operations Research \& Management Science,
  27. Kluwer Academic Publishers, Boston, MA, 2000.
\newblock Theory, algorithms, and applications.

\bibitem{WoZh:96}
H.~Wolkowicz and Q.~Zhao.
\newblock Semidefinite programming relaxations for the graph partitioning
  problem.
\newblock {\em Discrete Appl. Math.}, 96/97:461--479, 1999.
\newblock Selected for the special Editors' Choice, Edition 1999.

\bibitem{SWright:96}
S.~Wright.
\newblock {\em Primal-Dual Interior-Point Methods}.
\newblock Society for Industrial and Applied Mathematics (SIAM), Philadelphia,
  Pa, 1996.

\bibitem{KaReWoZh:94}
Q.~Zhao, S.E. Karisch, F.~Rendl, and H.~Wolkowicz.
\newblock Semidefinite programming relaxations for the quadratic assignment
  problem.
\newblock {\em J. Comb. Optim.}, 2(1):71--109, 1998.
\newblock Semidefinite programming and interior-point approaches for
  combinatorial optimization problems (Fields Institute, Toronto, ON, 1996).

\bibitem{RongZhu:07}
Y-R. Zhu.
\newblock {\em Recent Advances and Challenges in Quadratic Assignment and
  Related Problems}.
\newblock PhD thesis, University of Pennsylvania, 2007.
\newblock PhD Thesis.

\end{thebibliography}
\def\cprime{$'$} \def\cprime{$'$} \def\cprime{$'$}
  \def\udot#1{\ifmmode\oalign{$#1$\crcr\hidewidth.\hidewidth
  }\else\oalign{#1\crcr\hidewidth.\hidewidth}\fi} \def\cprime{$'$}
  \def\cprime{$'$} \def\cprime{$'$} \def\cprime{$'$}

\end{document}